\documentclass[11pt, oneside]{article}
\usepackage[utf8]{inputenc}
\usepackage{amsmath, amsthm, amssymb}
\usepackage[dvipsnames]{xcolor}
\usepackage{enumerate, comment, multirow, parskip}
\usepackage{dsfont}
\usepackage{enumitem}
\usepackage{mathtools}

\usepackage{comment}
\usepackage{makecell}

\usepackage{subcaption}
\usepackage{todonotes}
\usepackage{multicol}
\usepackage{tikz}
\usetikzlibrary{decorations.markings}
\usetikzlibrary{decorations.pathmorphing}
\usetikzlibrary{calc,backgrounds}
\usetikzlibrary{patterns}
\usepackage{standalone}
\usepackage{mathtools,mdframed}
\usepackage{adjustbox}
\usepackage[left=1in,right=1in,top=1in,bottom=1in,bindingoffset=0cm]{geometry}

\usepackage[normalem]{ulem}

\usepackage{hyperref}
\hypersetup{colorlinks=true,linkcolor=blue,anchorcolor=blue,citecolor=green,filecolor=blue,urlcolor=blue,bookmarksnumbered=true,pdfview=FitB}

\usepackage[capitalise]{cleveref}

\usepackage{tabularx}
\usepackage{booktabs}

\definecolor{green}{RGB}{50,205,50}
\definecolor{yellow}{RGB}{240,240,10}

\newtheorem{theorem}{Theorem}[section]
\newtheorem{lemma}[theorem]{Lemma}

\newtheorem{question}[theorem]{Question}

\newtheorem*{theorem*}{Theorem}
\newtheorem*{corollary*}{Corollary}

\theoremstyle{definition}

\theoremstyle{definition}

\newtheorem{claim}{Claim}
\newenvironment{proofc}{\begin{proof}[Proof of Claim]}{\end{proof}}

\newcommand{\owen}[1]{{\color{orange}\textsc{Owen:} #1}}

\title{On the recolorability of $(2K_2, K_4)$-free graphs}

\author{
Henry Echeverr\'ia\footnote{Instituto de Ingenier\'ia Matem\'atica-CIMFAV, Universidad de Valpara\'iso, Chile.
Email: {\tt henry.echeverria@postgrado.uv.cl}. Supported by ANID BECAS/DOCTORADO NACIONAL 21231147.
}
\and 
Owen Henderschedt\footnote{Department of Mathematics, Iowa State University, Ames, IA, U.S.A.
  Email: {\tt owenhen@iastate.edu}.}}

\begin{document}
\date{}
\maketitle
\begin{abstract}
 Given a graph $G$ and an integer $\ell>\chi(G)$, the reconfiguration graph of the $\ell$-colorings of $G$ has as its vertices as the proper $\ell$-colorings of $G$, with an edge between two colorings whenever they differ on exactly one vertex. We say that $G$ is \emph{recolorable} if this reconfiguration graph is connected for every $\ell>\chi(G)$. Belavadi and Cameron determined which $(F_1,F_2)$-free graphs are recolorable whenever $F_1$ and $F_2$ are graphs on at most four vertices, with the single exception of $(2K_2,K_4)$-free graphs. Gaspers and Huang showed such graphs are $4$-colorable. The $3$-colorable case within this class has also been resolved, leaving the open question of whether every $(2K_2,K_4)$-free graph with chromatic number $4$ is recolorable. In this paper, we provide evidence toward an affirmative answer by establishing recolorability for three subclasses: $(2K_2,K_4,C_5)$-free graphs, $(2K_2,K_4,H_a,H_b)$-free graphs for any distinct $a,b\in \{2,3,4\}$, and $(2K_2,K_4,H_4)$-free graphs containing an induced $W_5$, where $H_i$ denotes the unique $2K_2$-free graph obtained from a $W_5$ by keeping exactly $i$ edges from the universal vertex to the cycle.
\end{abstract}

\section{Introduction}

Let $\mathcal{P}$ be a combinatorial problem, and let $S_{\mathcal{P}}$ be the set of feasible solutions of $\mathcal{P}$. Given a set of transformation rules that allow us to transform one solution into another, the \textit{Reconfiguration Problem} asks whether any two solutions in $S_{\mathcal{P}}$ can be linked by a sequence of allowed transformations, with each step yielding a feasible solution. Equivalently, one defines the \textit{reconfiguration graph} of $\mathcal{P}$, whose vertex set is $S_{\mathcal{P}}$ and whose edges join solutions that differ by one transformation, and asks whether this graph is connected. For background on reconfiguration problems and their complexity, we refer to the survey of van den Heuvel \cite{vandenHeuvel2013} and the recent overview by Nishimura \cite{Nishimura2018}.

In this paper, we take $\mathcal{P}$ to be one of the more well-studied reconfiguration problems in graph theory: \textit{proper $\ell$-colorings of a graph $G$}. That is, we assign each vertex of a finite simple graph $G$ one of at most $\ell$ colors so that adjacent vertices receive distinct colors. Here, $S_{\mathcal{P}}$ is the set of all such colorings, for some fixed integer $\ell$, and the transformation rule is to change the color of exactly one vertex to a different one of the $\ell$ colors while preserving properness. The \emph{$\ell$-coloring reconfiguration graph} of $G$, denoted $\mathcal{R}_\ell(G)$, has the $\ell$-colorings of $G$ as its vertices, with an edge between two colorings whenever they differ on exactly one vertex. We say that $G$ is \emph{recolorable} if $\mathcal{R}_\ell(G)$ is connected for every $\ell>\chi(G)$. The condition $\ell>\chi(G)$ is necessary, since at $\ell=\chi(G)$ the reconfiguration graph can often have isolated vertices; for example, in any proper $\chi(G)$-coloring of $K_{\chi(G)}$, no recoloring step is possible.

There is also a long history concerning the computational complexity of graph recoloring. Although deciding whether a graph is $3$-colorable is $\mathsf{NP}$-complete \cite{GJS76}, deciding whether two given $3$-colorings of a graph $G$ lie in the same component of $\mathcal{R}_3(G)$ can be done in polynomial time \cite{CvHJ11}. In contrast, for every fixed $\ell\geq 4$, deciding whether two given $\ell$-colorings of $G$ lie in the same component of $\mathcal{R}_\ell(G)$ is $\mathsf{PSPACE}$-complete \cite{BC09}. Deciding whether $\mathcal{R}_\ell(G)$ is connected, can behave differently. Even for bipartite graphs $G$, deciding whether $\mathcal{R}_3(G)$ is connected is $\mathsf{coNP}$-complete \cite{CvHJ09}. For more background on the computational complexity of reconfiguration problems, we refer the reader to the aforementioned surveys \cite{ Nishimura2018, vandenHeuvel2013}.

Given a family of graphs $\mathcal{F}$, we say that a graph $G$ is
\emph{$\mathcal{F}$-free} if $G$ does not contain an induced subgraph
isomorphic to any graph in $\mathcal{F}$. When $\mathcal{F}=\{F\}$, we
simply say that $G$ is \emph{$F$-free}. We also write that $G$ is
$(F_1,\ldots,F_t)$-free when $\mathcal{F}=\{F_1,\ldots,F_t\}$. Given two
graphs $F_1$ and $F_2$, we let $F_1+F_2$ denote their disjoint union.
The class of $\mathcal{F}$-free graphs is hereditary, and understanding
which hereditary graph classes are recolorable has been an active theme
in graph recoloring.

The case of a single forbidden induced subgraph is now completely
understood. Belavadi, Cameron, and Merkel proved the following dichotomy.

\begin{theorem}[Belavadi, Cameron, and Merkel~\cite{BCM24}]\label{thm:1forbidden}
Every $F$-free graph $G$ is recolorable if and only if $F$ is an induced
subgraph of $P_4$ or $P_3+P_1$.
\end{theorem}

This theorem builds on a number of prior results. For graphs on at most three  vertices, the only negative case is $K_3$: Cereceda, van den Heuvel, and Johnson constructed triangle-free graphs that are not recolorable~\cite{CerecedaHeuvelJohnson08}. On the positive side, Merkel proved that every $3K_1$-free graph is recolorable~\cite{Merkel22}. The $P_2+P_1$ case follows from the recolorability of $P_4$-free graphs, proved by Bonamy and Bousquet~\cite{BB18}. Finally, $P_3$-free graphs are disjoint unions of cliques, and their recolorability follows, for example, from the work of Bonamy, Johnson, Lignos, Patel, and Paulusma~\cite{BJLPP}. For graphs on four vertices, the $P_4$ case is included above, while Feghali and Merkel proved that there exist
$2K_2$-free graphs that are not recolorable~\cite{FM21}.

The next natural step is to forbid two induced subgraphs. Belavadi, Cameron, and Merkel~\cite{BCM24} proved that every $(2K_2,C_4)$-free graph is recolorable. Belavadi and Cameron~\cite{BelavadiCameron2024} then studied this question for pairs $(F_1,F_2)$ where $F_1$ and $F_2$ have at most four vertices. In particular, they showed that every $(2K_2,K_3)$-free graph is recolorable. More generally, they proved recolorability for $(2K_2,F)$-free graphs when $F$ is a triangle, paw, claw, or diamond, and gave negative examples for several other choices of $F$. Combining their results with \cref{thm:1forbidden} and earlier negative examples, the only remaining unresolved pair in this range is $(2K_2,K_4)$.

\begin{question}[\cite{BCH25}]\label{question:MAIN}
Is every $(2K_2,K_4)$-free graph recolorable?
\end{question}

One reason why determining whether $(2K_2,K_4)$-free graphs are recolorable may be challenging is that such graphs may not be $3$-colorable. For example, the wheel $W_5$, obtained by taking a $5$-cycle and adding a universal vertex, is $(2K_2,K_4)$-free and has chromatic number $4$. Another example is $\overline{C_7}$, the complement of a $7$-cycle, which also has chromatic number $4$. In fact, Gaspers and Huang \cite{GH19} showed that every such graph can be $4$-colored.

\begin{theorem}[\cite{GH19}]\label{thm:2K2-K4-4-colorable}
Every $(2K_2,K_4)$-free graph is $4$-colorable.
\end{theorem}

Moreover, the remaining difficulty in determining whether a $(2K_2,K_4)$-free graph $G$ is recolorable lies in the case when $G$ has chromatic number $4$, due to the following theorem of Belavadi, Cameron, and Hildred.

\begin{theorem}[\cite{BCH25}]\label{theo:3color}
Every $3$-colorable $2K_2$-free graph is recolorable.
\end{theorem}

This paper provides evidence toward an affirmative answer to Question \ref{question:MAIN} by forbidding additional induced subgraphs. In light of the examples above, a natural first candidate is to forbid $W_5$. Even determining whether $(2K_2,K_4,W_5)$-free graphs are recolorable seems challenging; however, our first main result resolves a stronger restriction by forbidding induced $5$-cycles altogether.

\begin{theorem}\label{theorem:main3}
Every $(2K_2,K_4,C_5)$-free graph is recolorable.
\end{theorem}

We note that such graphs are not necessarily $3$-colorable, as shown by $\overline{C_7}$. Moreover, neither the $2K_2$-free assumption nor the $K_4$-free assumption can be dropped. Indeed, Feghali and Merkel constructed non-recolorable
$(2K_2,C_5)$-free graphs~\cite{FM21}, while $K_{3,3}$ with a perfect
matching removed, equivalently $C_6$, is a non-recolorable
$(K_4,C_5)$-free graph; see also~\cite{BCH25}.

It is then natural to assume the existence of an induced $C_5$ and ask which additional forbidden structures allow us to show that a $(2K_2,K_4)$-free graph is recolorable. Here, we systematically consider the possible graphs obtained by adding a single vertex to a $C_5$. Requiring the resulting graph to remain $2K_2$-free, then this new vertex can have degree $2$, $3$, $4$, or $5$ on the cycle, and in each case there is a unique such configuration. We denote these graphs by $H_2,H_3,H_4$, and $W_5$, respectively; they are shown in Figure \ref{fig:wheel-subgraphs}. Our next result shows that, among the three proper subgraphs $H_2,H_3,H_4$ of the wheel, forbidding any two is enough to guarantee recolorability.

\begin{theorem} \label{theorem:main2}
Let $H_a,H_b$ be distinct graphs in $\{H_2,H_3,H_4\}$. Then every $(2K_2,K_4,H_a,H_b)$-free graph is recolorable.
\end{theorem}

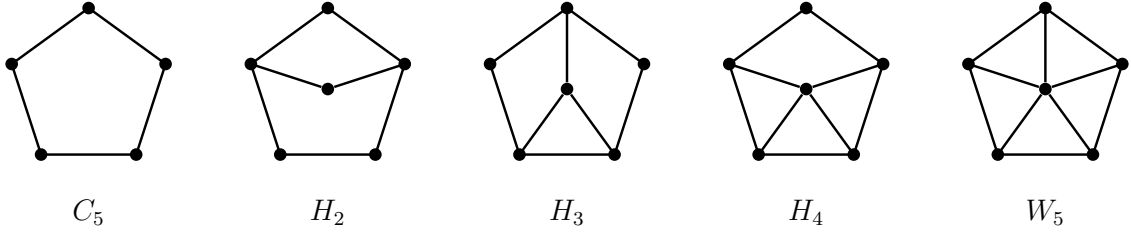
\begin{figure}[h!]
    \centering
    \begin{tikzpicture}[
    vtx/.style={circle, fill=black, inner sep=1.8pt},
    edge/.style={line width=1pt},
    graphlabel/.style={font=\large}
]

\def\r{1.15}
\def\sep{3.4}

\foreach \x/\name in {0/C_5,1/H_2,2/H_3,3/H_4,4/W_5} {

    \coordinate (c\x) at ({\x*\sep},0);

    \coordinate (v\x-0) at ($(c\x)+(90:\r)$);
    \coordinate (v\x-1) at ($(c\x)+(18:\r)$);
    \coordinate (v\x-2) at ($(c\x)+(-54:\r)$);
    \coordinate (v\x-3) at ($(c\x)+(-126:\r)$);
    \coordinate (v\x-4) at ($(c\x)+(162:\r)$);

    \draw[edge] (v\x-0) -- (v\x-1) -- (v\x-2) -- (v\x-3) -- (v\x-4) -- cycle;

    \foreach \i in {0,1,2,3,4} {
        \node[vtx] at (v\x-\i) {};
    }

    \node[graphlabel] at ($(c\x)+(0,-1.75)$) {$\name$};
}

\foreach \x in {1,2,3,4} {
    \node[vtx] (u\x) at (c\x) {};
}

\draw[edge] (u1) -- (v1-4);
\draw[edge] (u1) -- (v1-1);

\draw[edge] (u2) -- (v2-0);
\draw[edge] (u2) -- (v2-2);
\draw[edge] (u2) -- (v2-3);

\draw[edge] (u3) -- (v3-1);
\draw[edge] (u3) -- (v3-2);
\draw[edge] (u3) -- (v3-3);
\draw[edge] (u3) -- (v3-4);

\foreach \i in {0,1,2,3,4} {
    \draw[edge] (u4) -- (v4-\i);
}

\end{tikzpicture}
    \caption{All $2K_2$-free graphs $H$ such that $C_5$ is an induced subgraph of $H$ and $H$ is an induced subgraph of $W_5$.}
    \label{fig:wheel-subgraphs}
\end{figure}

Our final result is to show that if we only forbid $H_4$, the existence of a $W_5$ is strong enough to force recolorability.

\begin{theorem}\label{theorem:main1}
Every $(2K_2,K_4,H_4)$-free graph that contains a $W_5$ is recolorable.
\end{theorem}

After determining that a graph $G$ is recolorable, i.e. $\mathcal{R}_{\ell}(G)$ is connected for all $\ell>\chi(G)$, often the next question to investigate is the diameter of $\mathcal{R}_{\ell}(G)$. In words, this is the maximum number of transformation rules (recoloring steps) required to get from any $\ell$-coloring to any other $\ell$-coloring. In many of the results we have cited, the authors have determined an upper bound for this diameter. In this paper, the proofs of \cref{theorem:main3}, \cref{theorem:main2}, and \cref{theorem:main1} are algorithmic in the sense that they provide the recoloring steps to get from any one $\ell$-coloring to another. Optimizing these recoloring sequences was not a goal of ours, however, and hence we do not include any bounds on the diameter, although such bounds do exist.

The rest of the paper is organized as follows. In \cref{sec:RecoloringPrelims}, we present and prove some preliminary recoloring results that will be used in the proofs of our three main theorems. We also give a short proof outline for the template that all three of our proofs follow. In \cref{sec:THE-STRUCTURE}, we provide structural results on $(2K_2,K_4)$-free graphs containing a $C_5$, as well as additional structure when more subgraphs are forbidden, that will be useful for our main theorems. Then, in \cref{sec:FIRST-THEOREM}, we prove \cref{theorem:main1}, and in \cref{sec:SECOND-THEOREM}, we prove \cref{theorem:main2}. Finally, \cref{sec:THIRD-THEOREM} is dedicated to the proof of \cref{theorem:main3}. It contains two subsections: Subsection \ref{subsection:C7-comp-structure} provides the structural decomposition of $(2K_2,K_4,C_5)$-free graphs, and Subsection \ref{subsec:PROOF-OF-THIRD-THEOREM} presents the proof of \cref{theorem:main3}.

\section{Recoloring preliminaries and a proof outline}\label{sec:RecoloringPrelims}

In this section, we collect a toolkit of recoloring results that will be used in the proofs of our three main theorems. 

First, we will use the following degeneracy bound, proved independently by Dyer, Flaxman, Frieze, and Vigoda and by Cereceda, van den Heuvel, and Johnson. A graph $G$ is said to be $d$-degenerate if every nonempty subgraph of $G$ has a vertex of degree at most $d$. Equivalently, the vertices of $G$ can be ordered so that each vertex has at most $d$ neighbors appearing after it in the ordering.

\begin{theorem}[\cite{CvHJ09,DFFV06}]\label{thm:degeneracy}
If $G$ is $d$-degenerate, then $\mathcal{R}_{\ell}(G)$ is connected for every $\ell\geq d+2$.
\end{theorem}

The next tool is a reduction from showing that $\mathcal{R}_\ell(G)$ is connected for every $\ell>\chi(G)$ to showing only that $\mathcal{R}_{\chi(G)+1}(G)$ is connected, whenever $G$ belongs to a hereditary graph class. In our case, since every $(2K_2,K_4)$-free graph is $4$-colorable, it will be enough to prove that $\mathcal{R}_5(G)$ is connected. Like the graph classes in \cref{theorem:main2} and \cref{theorem:main3}, this applies whenever we work with subclasses of $(2K_2,K_4)$-free graphs obtained by forbidding additional induced subgraphs. This reduction is often implicit in graph recoloring arguments, and we include a proof for completeness.

\begin{lemma}\label{lem:only5}
    Let $\mathcal{G}$ be a hereditary class of graphs such that $\mathcal{R}_{\chi(G)+1}(G)$ is connected for every $G\in \mathcal{G}$, then $\mathcal{R}_{\ell}(G)$ is connected for every $G\in \mathcal{G}$ and every $\ell \geq  \chi(G)+1$. 
\end{lemma}
\begin{proof}
First, we write $\varphi \;\xleftrightarrow{\;\mathcal{R}_{\ell}\;}\varphi'$ if there is a path from $\varphi$ to $\varphi'$ in $\mathcal{R}_{\ell}(G)$. We prove by induction on $\ell$ the following statement: for every $G\in\mathcal{G}$ with $\chi(G)\leq \ell-1$, $\mathcal{R}_{\ell}(G)$ is connected. For $\ell=1$, the statement is trivial. Now let $\ell\geq 2$, and assume the statement holds for $\ell-1$. If $\ell=\chi(G)+1$, then the result follows directly from the hypothesis applied to $G$. Hence, in the inductive step, we may assume $\ell\geq \chi(G)+2$.

Let $G\in\mathcal{G}$, and write $\chi=\chi(G)$. Suppose $\ell\geq \chi+2$. Let $\varphi,\psi:V(G)\to[\ell]$ be two $\ell$-colorings of $G$. Let $\varphi'$ and $\psi'$ be $\ell$-colorings of $G$ which do not use the color $\ell$. Such colorings exist since $\ell\geq \chi+2$. We then construct paths
\[\varphi \;\xleftrightarrow{\;\mathcal{R}_{\ell}\;}\; \varphi' \;\xleftrightarrow{\;\mathcal{R}_{\ell}\;}\; \psi' \;\xleftrightarrow{\;\mathcal{R}_{\ell}\;}\; \psi.\]
The middle path exists because $\varphi'$ and $\psi'$ are $(\ell-1)$-colorings of $G$. Since $\chi(G)\leq \ell-2$, the induction hypothesis implies that they are connected in $\mathcal{R}_{\ell-1}(G)$, and hence they are connected in $\mathcal{R}_{\ell}(G)$, since each coloring along this path is also an $\ell$-coloring. It suffices to construct the first path, from $\varphi$ to $\varphi'$, since the argument for $\psi$ and $\psi'$ is
analogous.

Let $L=\{v\in V(G):\varphi(v)=\ell\}$. Note that $L$ is independent in $G$. If $L=\emptyset$, then we may take $\varphi'=\varphi$, and we are done. Thus, we may assume $L\neq\emptyset$. Let $G'=G-L$. Since $\mathcal{G}$ is hereditary, $G'\in\mathcal{G}$. Also $\chi(G')\leq \chi(G)\leq \ell-2$, and so, by the induction hypothesis, $\mathcal{R}_{\ell-1}(G')$ is connected.

Let $\varphi_{G'}$ and $\varphi'_{G'}$ be the restrictions of
$\varphi$ and $\varphi'$ to $G'$, respectively. Both are proper colorings of $G'$ using only colors in $\{1,\dots,\ell-1\}$. Hence there is a path $\varphi_{G'}=\gamma_0,\gamma_1,\ldots,\gamma_m=\varphi'_{G'}$ in $\mathcal{R}_{\ell-1}(G')$. For each $i$, extend $\gamma_i$ to a coloring $\hat{\gamma}_i$ of $G$ by defining
\[\hat{\gamma}_i(v)= \begin{cases}
\gamma_i(v), & \text{if } v\in V(G'),\\
\ell, & \text{if } v\in L.
\end{cases}\]
Each $\hat{\gamma}_i$ is a proper $\ell$-coloring of $G$, since no vertex
of $G'$ has color $\ell$ and $L$ is independent. Thus, $\varphi \;\xleftrightarrow{\;\mathcal{R}_{\ell}\;}\; \hat{\gamma}_m.$

It remains to recolor the vertices of $L$ to their colors under
$\varphi'$, one at a time. For each $v\in L$, we recolor $v$ from
$\ell$ to $\varphi'(v)\in[\ell-1]$. This is legal because the vertices
of $G'$ are currently colored as in $\varphi'$, and no neighbor of $v$
has color $\varphi'(v)$ in the proper coloring $\varphi'$. The
independence of $L$ also guarantees that recoloring the vertices of $L$
sequentially creates no conflict among them. After processing all
vertices of $L$, we obtain exactly $\varphi'$. Hence
$\varphi \;\xleftrightarrow{\;\mathcal{R}_{\ell}\;}\; \varphi'$, as
required.
\end{proof}

Next, we will use the (Optimal) Renaming Lemma introduced and proved by Bonamy and Bousquet \cite{BB18} which was recently improved (in terms of number of steps) by Cambie, van Batenburg and Cranston \cite{cambie2024sharp}. 

\begin{lemma}[Optimal Renaming Lemma \cite{cambie2024sharp}]\label{lem: renaming} Let $\alpha$ and $\beta$ be two $k$-colorings of $G$ that induce the same partition of vertices into color classes. Let $\ell\geq k+1$. Then $\alpha$ can be recolored to $\beta$ in $\mathcal{R}_{\ell}(G)$ by recoloring each vertex at most $\lfloor3n/2\rfloor$ times.
\end{lemma}

We say a pair $u,v$ of nonadjacent vertices are \textit{comparable} if either $N(u)\subseteq N(v)$ or vice versa. Belavadi and Cameron proved the following useful lemma, which allows one to deduce the recolorability of a graph from that of a smaller induced subgraph.

\begin{lemma}[Lemma 3 in~\cite{BelavadiCameron2024}]\label{lem: comparable}
Let $G$ be a graph with comparable vertices $u$ and $v$ with $N(u)\subseteq N(v)$. If $G-u$ is recolorable, then so is $G$.
\end{lemma}

The next lemma gives a useful reduction that will be applied repeatedly in the recoloring arguments.

\begin{lemma}\label{claim:CoverI}
Let $G$ be a graph, and let $\ell\geq\chi(G)+1$. Let $\varphi$ be an $\ell$-coloring of $G$, and suppose there exist a $\chi(G)$-coloring $\psi$ of $G$, with color classes $I_1,\ldots,I_{\chi(G)}$, a color $\alpha\in[\ell]$, and an index $j\in[\chi(G)]$ such that
\[
I_j\;\subseteq\;V_\alpha,\text{ where } V_\alpha:= \{v\in V(G): \varphi(v)=\alpha\}.
\]
Let $G':=G-V_\alpha$. If $G'$ is recolorable, then $\varphi\xleftrightarrow{\;\mathcal{R}_{\ell}\;}\; \psi$.
\end{lemma}

\begin{proof}
By relabeling the colors of $\psi$ if necessary, we may assume that every vertex of $I_j$ has color $\alpha$ under $\psi$. We exhibit a sequence of recolorings in $\mathcal{R}_\ell(G)$ transforming $\varphi$ into $\psi$.

We first bring $\varphi$ and $\psi$ to agree on $G'=G-V_\alpha$. Let $\varphi'$ and $\psi'$ be the restrictions of $\varphi$ and $\psi$ to $G'$; both are proper colorings of $G'$. Since $I_j\subseteq V_\alpha$, the class $I_j$ is entirely removed in $G'$, so $\psi'$ uses at most $\chi(G)-1$ colors; in particular $\chi(G')\le\chi(G)-1$. As $\varphi$ is an $\ell$-coloring and every vertex of color $\alpha$ lies in $V_\alpha$, the coloring $\varphi'$ uses colors from $[\ell]\setminus\{\alpha\}$ that is a set of size $\ell-1\ge\chi(G)\ge\chi(G')+1$. Since $G'$ is recolorable and $\ell-1\ge\chi(G')+1$, the reconfiguration graph $\mathcal{R}_{\ell-1}(G')$ with colors $[\ell]\setminus\{\alpha\}$ is connected, so we have $\varphi'\xleftrightarrow{\;\mathcal{R}_{\ell-1}(G')\;}\psi'$ such that in every step the color $\alpha$ is never used. Now, we extend this sequence to $G$ as follows: each recoloring of a vertex $v\in V(G')$ remains proper in $G$, because the only neighbors of $v$ outside $G'$ lie in $V_\alpha$ and are colored $\alpha$, which is never used for recoloring a single vertex. Then, we reach a coloring $\varphi_1$ of $G$ with $\varphi_1|_{G'}=\psi'$ and $\varphi_1|_{V_\alpha}\equiv\alpha$. Hence $\varphi\xleftrightarrow{\;\mathcal{R}_\ell(G)\;}\varphi_1$.

It remains to transform $\varphi_1$ into $\psi$. For each $i\in[\chi(G)]\setminus\{j\}$, let $I'_i=I_i\cap V_\alpha$. Note that the vertices of such set have not been recolored in any step. Then, for each $i\in[\chi(G)]\setminus\{j\}$, we recolor each vertex of $I'_i$ to its color under $\psi$. 
This step is legal: if $v\in I'_i$, then no neighbor of $v$ currently has color $\psi(v)$. 

Indeed, the neighbors of $v$ in $G'$ are already colored according to
$\psi'=\psi|_{G'}$, and since $\psi$ is proper, none of them has color
$\psi(v)$. Moreover, $V_\alpha$ is independent, so recoloring vertices
inside $V_\alpha$ creates no conflict among themselves.
After all these moves, every vertex carries its color under $\psi$, so we have reached $\psi$. Hence $\varphi\xleftrightarrow{\;\mathcal{R}_\ell(G)\;}\psi$.
\end{proof}

Throughout this paper, we are going to use some extra notation for increasing and preserving the understanding and the clarity of the proofs. If $A$ is a set of vertices, the step $A\to \alpha$ means that the vertices of $A$ are recolored with color $\alpha$, one at a time. Such a step is legal if, at the moment the step is performed, no vertex of color $\alpha$ has a neighbor in $A$.

\subsubsection*{Outline of the proofs of Theorems \ref{theorem:main3}, \ref{theorem:main2} and \ref{theorem:main1}}

The general strategy for proving our main theorems can be summarized as follows. Let $G$ be a $(2K_2,K_4)$-free graph, with additional forbidden subgraphs depending on the theorem. We first find a $4$-coloring $\varphi$ of $G$, or equivalently, a partition of $V(G)$ into four independent sets $I_1,I_2,I_3,I_4$. Such a coloring exists by \cref{thm:2K2-K4-4-colorable}. We say that a $4$-coloring of $G$ is \textit{canonical with respect to $\varphi$} if its color classes are precisely $I_1,I_2,I_3,I_4$, up to renaming colors. The additional forbidden subgraphs will determine the choice of $\varphi$, and hence which $4$-colorings are canonical.

We would like to prove that the reconfiguration graph $\mathcal{R}_{\ell}(G)$ is connected with $\ell\geq 5$. Let $\ell\geq 5$ (if $G$ is $(2K_2,K_4,C_5)$-free graph, by Lemma \ref{lem:only5}, it suffices to consider $\ell=5$) and given an arbitrary $\ell$-coloring $\varphi'$ of $G$, our goal is to transform $\varphi'$, through a sequence of legal recoloring steps in $\mathcal{R}_\ell(G)$, into a canonical coloring. For this, it suffices to transform $\varphi'$ into a coloring in which one of the independent sets $I_j$ is monochromatic, that is, every vertex in $I_j$ receives the same color. Once this is achieved, we apply Lemma \ref{claim:CoverI}, which reduces the problem to recoloring the smaller graph $G'=G\setminus I_j$. By construction, $G'$ is $3$-colorable, and hence is recolorable by \cref{theo:3color}. Thus every $\ell$-coloring can be recolored to a canonical coloring. Finally, by the Renaming Lemma (Lemma \ref{lem: renaming}), any two canonical colorings lie in the same component of $\mathcal{R}_{\ell}(G)$, and so $\mathcal{R}_{\ell}(G)$ is connected.

\textbf{Notation.} Throughout the proofs, we follow a few conventions to avoid confusion. We reserve the symbols $\varphi,\psi$ and its variants, such as $\varphi'$ and $\psi_1$, for colorings. Since we will always work with $\ell$-colorings with $\ell\geq 5$, we use the color palette $\{\alpha,\beta,\gamma,\delta,\varepsilon\}$; and if we need a sixth color we use $\tau$. Thus, uppercase letters are reserved for sets of vertices and graphs, while lowercase letters are reserved for individual vertices. 

\section{The structure of $(2K_2,K_4)$-free graphs containing a $C_5$}\label{sec:THE-STRUCTURE}

To prove Theorems \ref{theorem:main2}, and \ref{theorem:main1}, we will use the structure of $(2K_2,K_4)$-free graphs, together with additional structural properties that arise when further induced subgraphs are forbidden. In \cite{GH19}, Gaspers and Huang proved that $(2K_2,K_4)$-free graphs are $4$-colorable, and we will use their structural framework here.

Let $G$ be a $(2K_2,K_4)$-free graph. By Lemma \ref{lem: comparable}, we may assume that $G$ has no comparable vertices. If $G$ is also $C_5$-free, then $G$ will be shown to be recolorable in the proof of \cref{theorem:main3}, and in \cref{subsection:C7-comp-structure}, we will introduce the separate structural notation needed. Thus, throughout this section, we assume that $G$ contains an induced $5$-cycle $C$. Let $1,2,3,4,5$ denote the vertices of $C$ in cyclic order. We now define the following sets, where indices are taken modulo $5$.
\begin{align*}
    U &= \{v\in V(G)\setminus V(C): N(v) \cap V(C) = V(C)\},\\
    F_i &= \{v\in V(G)\setminus V(C): N(v) \cap V(C) = V(C)\setminus \{i\}\},\\
    Y_i &= \{v\in V(G)\setminus V(C): N(v)\cap V(C) = \{i,i+2,i-2\}\},\\
    R_i &= \{v\in V(G)\setminus V(C): N(v)\cap V(C) = \{i-1,i+1\}\},\\
    Z &= \{v\in V(G)\setminus V(C): N(v)\cap V(C) = \emptyset\}.
\end{align*}

We now establish structural properties of these sets, many of which can be found in \cite{GH19}. We begin with the properties that follow only from forbidding $2K_2$ and $K_4$. We also note that a vertex in $R_i$, $Y_i$, $F_i$, or $U$, together with $V(C)$, induces precisely an $H_2$, $H_3$, $H_4$, or $W_5$, respectively. Thus, when we additionally forbid one of the graphs in $\{H_2,H_3,H_4,W_5\}$, the corresponding family of sets must be empty. We separate and prove these additional consequences in subsequent lemmas.

\begin{lemma}[\cite{GH19}]
Let $G$ be a $(2K_2,K_4)$-free graph, let $C=12345$ be an induced $5$-cycle in $G$, and define $U$, $F_i$, $Y_i$, $R_i$, and $Z$ as above. The following hold, where indices are taken modulo $5$. Moreover, each assertion involving $i$ holds for every $i\in[5]$.
\begin{enumerate}[label={(1.\arabic*)}, ref={1.\arabic*}]
\item\label{item:1.1}
$V(G)=V(C)\cup U\cup Z\cup \bigcup_{i=1}^5(F_i\cup Y_i\cup R_i)$. Moreover, the sets $U$ and $Z$, and each of the sets $F_i$, $Y_i$, and $R_i$, are independent.

\item\label{item:1.2}
$U$ is anticomplete to $F_i\cup Y_i$.

\item\label{item:1.3}
$F_i$ is complete to $Y_{i-2}\cup Y_{i+2}\cup R_{i-1}\cup R_{i+1}$ and anticomplete to $Y_{i-1}\cup Y_i\cup Y_{i+1}$. Moreover, at least one of $F_i$ and $F_{i+2}$ is empty.

\item\label{item:1.4}
$Y_i$ is complete to $Y_{i+1}\cup R_i$. Moreover, each vertex of $Y_i$ is anticomplete to at least one of $Y_{i-2}$ and $Y_{i+2}$.

\item\label{item:1.5}
$R_i$ is complete to $R_{i+1}$. Moreover, either $R_i$ is anticomplete to $Y_{i+1}$ or $R_{i+1}$ is anticomplete to $Y_i$. Each vertex of $R_i$ is anticomplete to at least one of $Y_{i+1}$ and $Y_{i+2}$, and to at least one of $Y_{i-1}$ and $Y_{i-2}$.

\item\label{item:1.6}
$Z$ is anticomplete to $R_i$.
\end{enumerate}
\end{lemma}

Next we consider the additional structural claims when $G$ is also $H_3$-free with no comparable vertices.

\begin{lemma}\label{lem:H3-free-structure}
Let $G$ be a $(2K_2,K_4,H_3)$-free graph with no comparable vertices. Let $C=12345$ be an induced $5$-cycle in $G$, and define $U$, $F_i$, $Y_i$, $R_i$, and $Z$ with respect to $C$ as above. The following hold, where indices are taken modulo $5$. Moreover, each assertion involving $i$ holds for every $i\in[5]$.
\begin{enumerate}[label={(2.\arabic*)}, ref={2.\arabic*}]
\item\label{item:2.1} $Y_i = \emptyset$

\item\label{item:2.2}
$F_i$ is complete to $R_{i-2}\cup R_{i+2}$.

\item\label{item:2.3}
$R_i$ is anticomplete to $R_{i-2}\cup R_{i+2}$.

\item\label{item:2.4}
$Z=\emptyset$.

\end{enumerate}
\end{lemma}

\begin{proof}
(\ref{item:2.1}) holds trivially since $G$ is $H_3$-free. 

First, we prove (\ref{item:2.2}). Suppose, for a contradiction, that there exist $f\in F_i$ and $r\in R_{i+2}$ such that $fr\notin E(G)$. Then $(V(C)\setminus \{i+2\})\cup {r}$ induces a $5$-cycle. Since $f$ is adjacent to exactly three vertices of this cycle, namely $i-1$, $i+1$, and $i+3$, the vertices of this cycle together with $f$ induce an $H_3$, a contradiction. By symmetry, $F_i$ is complete to $R_{i+2}\cup R_{i-2}$.

Next, we prove (\ref{item:2.3}). Suppose, for a contradiction, that there exist adjacent vertices $r\in R_i$ and $r'\in R_{i+2}$. Then $(V(C)\setminus \{i\})\cup {r}$ induces a $5$-cycle. Since $r'$ is adjacent to exactly three vertices of this cycle, namely $i+1$, $i-1$, and $r$, the vertices of this cycle together with $r'$ induce an $H_3$, a contradiction. By symmetry, $R_i$ is anticomplete to $R_{i+2}\cup R_{i-2}$.

Finally, we prove (\ref{item:2.4}). Suppose, for a contradiction, that $z\in Z$. Fix $i\in[5]$. Since $z$ has no neighbors in $C$, we have $N(i)\nsubseteq N(z)$, and since $z$ and $i$ are nonadjacent and $G$ has no comparable vertices, it follows that $N(z)\nsubseteq N(i)$. Thus there exists some $x\in N(z)\setminus N(i)$. By (\ref{item:1.1}) and (\ref{item:1.6}), we have $x\notin Z\cup \bigcup_{j=1}^5 R_j$. Since $G$ is $H_3$-free, $Y_j=\emptyset$ for every $j\in[5]$, and since every vertex of $U$ is adjacent to $i$, we also have $x\notin U$. Therefore $x\in F_j$ for some $j$. Since $x\notin N(i)$, we must have $j=i$, so $F_i\neq\emptyset$. As this holds for every $i\in[5]$, we contradict (\ref{item:1.3}), which says that at least one of $F_i$ and $F_{i+2}$ is empty. Therefore $Z=\emptyset$.
\end{proof}

Next we turn to $(2K_2,K_4,H_4)$-free graphs without comparable vertices and get the following additional claims depending on if $G$ contains a $5$-wheel. We point out that a few of the items in the following two lemmas can be found in \cite{GH19}. However, since the proofs are short and simple, we include them here for completeness.

\begin{lemma}\label{lem:H4-free-structure}
Let $G$ be a $(2K_2,K_4,H_4)$-free graph with no comparable vertices. Let $C=12345$ be an induced $5$-cycle in $G$, and define $U$, $F_i$, $Y_i$, $R_i$, and $Z$ with respect to $C$ as above. The following hold, where indices are taken modulo $5$. Moreover, each assertion involving $i$ holds for every $i\in[5]$.
\begin{enumerate}[label={(3.\arabic*)}, ref={3.\arabic*}]
\item\label{item:3.1}
$F_i = \emptyset$

\item\label{item:3.2}
$U$ is complete to $R_i$.

\item\label{item:3.3}
$R_i$ is anticomplete to $Y_{i-1}\cup Y_{i+1}$.

\end{enumerate}
\end{lemma}

\begin{proof}
(\ref{item:3.1}) holds trivially since $G$ is $H_4$-free. 

Next, we prove (\ref{item:3.2}). Suppose, for a contradiction, that there exist $u\in U$ and $r\in R_i$ such that $ur\notin E(G)$. Then $(V(C)\setminus \{i\})\cup {r}$ induces a $C_5$, and together with $u$, these vertices induce an $H_4$, a contradiction. Thus, $U$ is complete to $R_i$.

Finally, we prove (\ref{item:3.3}). Suppose, for a contradiction, that there exist $r\in R_i$ and $y\in Y_{i+1}$ such that $ry\in E(G)$. Then $(V(C)\setminus \{i\})\cup {r}$ induces a $C_5$, and together with $y$, these vertices induce an $H_4$, a contradiction. Thus, by symmetry, $R_i$ is anticomplete to $Y_{i-1}\cup Y_{i+1}$.
\end{proof}

\begin{lemma}\label{lem:H4-W5-structure}
Let $G$ be a $(2K_2,K_4,H_4)$-free graph with no comparable vertices that contains a $W_5$ with induced $5$-cycle $C=12345$, and define $U$, $F_i$, $Y_i$, $R_i$, and $Z$ with respect to $C$ as above. The following hold, where indices are taken modulo $5$. Moreover, each assertion involving $i$ holds for every $i\in[5]$.

\begin{enumerate}[label={(4.\arabic*)}, ref={4.\arabic*}]

\item\label{item:4.1}
$R_i=\emptyset$.

\item\label{item:4.2}
$|U|=1$ and $U$ is complete to $Z$.

\item\label{item:4.3} $Y_i$ is anticomplete to $Y_{i-2}\cup Y_{i+2}$.

\end{enumerate}
\end{lemma}

\begin{proof}
We first prove (\ref{item:4.1}). To show that $R_i=\emptyset$ for all $i\in[5]$, we first note that $R_i$ is anticomplete to $R_{i+2}\cup R_{i-2}$. Indeed, if $rr'\in E(G)$ for some $r\in R_i$ and $r'\in R_{i+2}$, and $u\in U$, then $rr'$ and $ui$ induce a $2K_2$, a contradiction. Now by (\ref{item:1.1}), (\ref{item:1.4}), (\ref{item:1.5}), (\ref{item:1.6}), (\ref{item:3.2}), and (\ref{item:3.3}), for any $r\in R_i$,
\[N(r)\subseteq {i+1,i-1}\cup U\cup R_{i-1}\cup R_{i+1}\cup Y_i\cup Y_{i-2}\cup Y_{i+2}\subseteq N(i).\]
Since $G$ has no comparable vertices, $R_i=\emptyset$ for all $i\in[5]$.

We next prove (\ref{item:4.2}). Let $u\in U$, which exists since $G$ contains a $5$-wheel. We first show that $U$ is complete to $Z$. Suppose, for a contradiction, that there exists $z\in Z$ such that $uz\notin E(G)$. Since $G$ has no comparable vertices and $z$ is nonadjacent to every vertex of $C$, for each $i\in[5]$ there is a vertex $v\in N(z)\setminus N(i)$. Since $F_i=\emptyset$ by (\ref{item:3.1}), $R_i=\emptyset$ by (\ref{item:4.1}), $Z$ is independent by (\ref{item:1.1}), and every vertex of $U$ is adjacent to $i$, we have $v\in Y_j$ for some $j\in[5]$. By (\ref{item:1.2}), $u$ is nonadjacent to $v$. Hence $ui$ and $zv$ induce a $2K_2$, a contradiction. Therefore $U$ is complete to $Z$.

It remains to show that $|U|=1$. Since $U$ is independent by (\ref{item:1.1}), anticomplete to each $Y_i$ by (\ref{item:1.2}), complete to each vertex of $C$ by definition, and complete to $Z$ by the previous paragraph, any two vertices of $U$ have the same neighborhood. Since $G$ has no comparable vertices, it follows that $|U|=1$.

Finally to show (\ref{item:4.3}) we may assume for contradiction that by symmetry, $y_i\in Y_i$ and $y_{i+2}\in Y_{i+2}$ with $y_iy_{i+2}\in E(G)$, then $y_iy_{i+2}$ and $u(i+1)$ (where $U = \{u\}$) induces a $2K_2$, a contradiction.
\end{proof}

\section{Proof of \cref{theorem:main1}}\label{sec:FIRST-THEOREM}

The goal of this section is to prove that every $(2K_2,K_4,H_4)$-free graph containing an induced $5$-wheel is recolorable. First, we will use the following useful lemma for $2K_2$-free graphs. 

\begin{lemma}[Lemma~7 in~\cite{BelavadiCameron2024}]\label{lem: universal to independent} Let $G$ be a $2K_2$-free graph and suppose $V(G)$ can be partitioned into independent sets $A_1\dots A_t$ such that $A_1$ is (inclusion-wise) maximal. Then for each $i\in \{2,3,\dots t\}$, $A_1$ contains a vertex complete to $A_i$.
\end{lemma}

We now present the proof for \cref{theorem:main1}, which we restate here.

\begin{theorem*}[\ref{theorem:main1}]
Every $(2K_2,K_4,H_4)$-free graph that contains a $W_5$ is recolorable.
\end{theorem*}

\begin{proof}
Let $G$ be a $(2K_2,K_4,H_4)$-free graph containing a $5$-wheel with induced $5$-cycle $C=12345$ and center $u$. We use the notation developed in \cref{sec:THE-STRUCTURE}. By Lemma~\ref{lem: comparable}, we may assume that $G$ has no comparable vertices. Note that this is not immediate, since the class of $(2K_2,K_4,H_4)$-free graphs containing a $W_5$ is not hereditary. However, if $x,y\in V(G)$ with $xy\notin E(G)$ and $N(x)\subseteq N(y)$, then $G-x$ is certainly still $(2K_2,K_4,H_4)$-free, and moreover we claim that it still contains an induced $W_5$. Indeed, if $x\notin V(C)\cup \{u\}$, then this holds trivially. If $x\in V(C)\cup \{u\}$, then since $xy\notin E(G)$ and $N(x)\subseteq N(y)$, we have $y\notin V(C)\cup \{u\}$. Thus, by replacing $x$ with $y$, and noting that $G$ is $(2K_2,K_4)$-free, we still induce a $5$-wheel. Thus, we assume $G$ has no comparable vertices. By (\ref{item:3.1}), (\ref{item:4.1}), and (\ref{item:4.2}), we have $F_j=R_j=\emptyset$ for every $j\in[5]$, and $U=\{u\}$.

We would like to reduce to the case where there exists some $i\in[5]$ such that $Y_i$ and $Y_{i+2}$ are both nonempty. To justify this, we first consider what happens when no such pair exists. There are three possibilities: either every $Y_i$ is empty, exactly one of the sets $Y_i$ is nonempty, or the only nonempty $Y$-sets are two consecutive sets $Y_i$ and $Y_{i+1}$. We handle these three possibilities below. Afterward, by symmetry, we may assume that $Y_5$ and $Y_2$ are both nonempty.

If $Y_i=\emptyset$ for all $i\in[5]$, then $Z=\emptyset$. Indeed, if $z\in Z$, then $N(z)\subseteq U\subseteq N(1)$, so $z$ and $1$ are comparable, a contradiction. Hence $G=W_5$. This graph is recolorable by first recoloring the center $u$, if necessary, and then applying \cref{theo:3color} to the induced $5$-cycle. Next suppose that for some $i\in[5]$, the set $Y_i$ is nonempty and $Y_j=\emptyset$ for all $j\neq i$. Let $y\in Y_i$, and let $U={u}$. Since $U$ is anticomplete to $Y_i$ by (\ref{item:1.2}), the vertices $y$ and $u$ are nonadjacent. Thus, we have $N(y)\subseteq \{i,i+2,i-2\}\cup Z\subseteq N(u)$ contradicting that $G$ has no comparable vertices.

It remains to handle the case where the only nonempty $Y$-sets are two consecutive sets and by symmetry we suppose $Y_1\neq \emptyset$ and $Y_2\neq \emptyset$. We first claim that $Z=\emptyset$. Indeed, if $z\in Z$, then, since $Z$ is independent and anticomplete to $C$, we have $N(z)\subseteq U\cup Y_1\cup Y_2\subseteq N(4).$ Since $z$ is nonadjacent to $4$, this contradicts that $G$ has no comparable vertices. Thus $Z=\emptyset$. Due to the fact that $G$ contains no comparable vertices and $Y_i$ is independent by (\ref{item:1.1}), it is not hard to check that $|Y_1|=|Y_2|=1$ and by (\ref{item:1.3}) such vertices are adjacent. Then we define the following graph $S_1$. Let $S_1$ be the graph such that $V(S_1) = \{1,2,3,4,5,u, y_1,y_2\}$ where $\{1,2,3,4,5,u\}$ induces $5$-wheel, $y_1$ is adjacent to $\{4,1,3,y_2\}$ and $y_2$ is adjacent to $\{5,2,4,y_1\}$ (see left of \cref{fig:S2}). It is easy to verify that the following are four independent sets of $S_1$.
\[I_1 =  U \cup \{y_2\}, \qquad
        I_2 =  \{y_1,2\}, \qquad
        I_3 = \{1,4\}, \qquad
        I_4 =\{3,5\}.\]

\begin{claim}
    $S_1$ is recolorable.
\end{claim}

\begin{proofc}
    Given that $\ell\geq5$, we let $\{\alpha,\beta, \gamma,\delta,\varepsilon\}$ be five distinct colors. Without loss of generality, we may assume
$\varphi(u)=\alpha$, $\varphi(1)=\beta$, and $\varphi(y2)=\gamma$.

Note that if $\varphi(y_2)=\alpha$, then $I_1$ is monochromatic. Now, we suppose that some neighbor of $y_2$ is colored with $\alpha$, then the only possible option is $y_1$ since the vertices of the $W_5-u$ are not colored with $\alpha$. Note that if either $\varphi(4)=\beta$ or we can recolor $y_4$ with $\beta$, then $I_3$ is monochromatic. So, we suppose that $\varphi(4)\neq \beta$ and we cannot recolor $4$ with $\beta$. Then, some neighbor of $4$ has the color $\beta$. We have two options either $\varphi(y_2)=\beta$ or $\varphi(3)=\beta$. Suppose that $\varphi(y_2)=\beta$, then we can recolor $I_4$ with some color in the set $\{\delta,\varepsilon\}\setminus\varphi(4)$, and then $I_4$ is monochromatic. So, we may assume that $\varphi(3)=\beta$ and $\varphi(y_2)\neq\beta$. After relabeling, we can suppose that $\varphi(y_2)=\delta$, then there is not $\delta$ in $4$ and $5$, and the other vertices in $G-\{4,5,y_2\}$ are not colored with $\delta$, so we can recolor $U$ with $\delta$ and then $I_1$ is monochromatic. Therefore, in every case, we reach a monochromatic $I_i$ for some $i$, we apply Lemma \ref{claim:CoverI}.
\end{proofc}

We may now assume without loss of generality that $Y_2\neq \emptyset$ and $Y_5\neq \emptyset$. We define the following independent sets partitioning $V(G)$, which can be easily verified using (\ref{item:1.1})-(\ref{item:1.6}),(\ref{item:3.1}) and (\ref{item:4.1})-(\ref{item:4.3}).
        \[I_1 =  Z \cup \{1, 3\}, \qquad
        I_2 =  Y_3 \cup \{2,4\}, \qquad
        I_3 = Y_1 \cup Y_4\cup \{5\}, \qquad
        I_4 = U\cup Y_2 \cup Y_5.\]
    \begin{claim}
        The independence set $I_4$ is maximal.
    \end{claim}
    \begin{proofc}
        Clearly any $i\in V(C)$ has a neighbor in $I_4$ by definition of $U$. Similarly, $Z$ is complete to $U\neq \emptyset$ by (\ref{item:4.2}), so $Z$ has a neighbor in $I_4$. Finally, since $Y_2\neq\emptyset$ and $Y_5\neq\emptyset$, every vertex of $Y_1\cup Y_3$ has a neighbor in $Y_2$ by (\ref{item:1.4}), and every vertex of $Y_4$ has a neighbor in $Y_5$ by (\ref{item:1.4}). Thus every vertex in $V(G)\setminus I_4$ has a neighbor in $I_4$, and so $I_4$ is maximal.
        \end{proofc}

    \begin{claim}
    There exists vertices $x_1,x_2,x_3\in I_4$ such that $x_i$ is complete to $I_i$ for all $i\in [3]$ and moreover $x_1\in U$, $x_2\in Y_2$, and $x_3\in Y_5$.
    \end{claim}

\begin{proofc} Let $U = \{u\}$ by (\ref{item:4.2}) and since $Y_2\neq \emptyset$ and $Y_5\neq \emptyset$, let $y_2\in Y_2$ and $y_5\in Y_5$. We claim setting $x_1=u$, $x_2=y_2$, and $x_3=y_5$ we have $x_i$ complete to $I_i$. Indeed $u$ is complete to $\{1,3\}$ by definition and $Z$ by (\ref{item:4.2}). Next $y_2$ is complete to $\{2,4\}$ by definition and $Y_3$ by (\ref{item:1.4}). Finally $y_5$ is complete to $\{5\}$ by definition and $Y_1\cup Y_4$ by (\ref{item:1.4}).
\end{proofc}

We say a $4$-coloring of $G$ is \textit{canonical} if its colors classes are precisely $I_1,\dots , I_4$. To show that $G$ is recolorable, it suffices to show that for any $\ell$-coloring $\varphi$, with $\ell\geq 5$, we can recolor $\varphi$ to a canonical coloring. To do so, by Lemma \ref{claim:CoverI} it suffices to recolor $\varphi$ so that one of the $I_j$ is monochromatic. Since $\ell\geq 5$, we use $\{\alpha, \beta, \gamma, \delta, \varepsilon\}$ to denote $5$ distinct colors. Without loss of generality, assume $\varphi(x_1) = \alpha$. Up to renaming the remaining colors, we can assume that \[\varphi(x_2,x_3) \in \{(\alpha, \alpha), (\alpha, \beta), (\beta, \alpha), (\beta, \beta), (\beta, \gamma)\}.\]

If $\varphi(x_2,x_3) = (\alpha, \alpha)$, then since $x_i$ is complete to $I_i$, all vertices in $G-I_4$ are not colored with $\alpha$. Thus, we recolor $I_4$ with $\alpha$ and we are done by Lemma \ref{claim:CoverI}. We break the four remaining possibilities for $\varphi(x_2,x_3)$ into four separate claims. First we present two claims that we will use in all cases of $\varphi(x_2,x_3)$.

\begin{claim}\label{claim:Z-clearing}
Let $i\in[5]$, and suppose vertex $i$ is colored $\theta$. If no vertex of $U\cup Y_{i-1}\cup Y_{i+1}$ is colored $\theta$, then we can recolor $Z\to\theta$.
\end{claim}

\begin{proofc}
By definition, $Z$ is anticomplete to $V(C)$. Thus it suffices to check that no vertex in $U\cup \bigcup_{i=1}^5Y_i$ is colored $\theta$. Every vertex of $Y_i\cup Y_{i-2}\cup Y_{i+2}$ is adjacent to $i$ by definition, and since $\varphi(i)=\theta$, none of these vertices are colored $\theta$. By assumption, no vertex of $U\cup Y_{i-1}\cup Y_{i+1}$ has color $\theta$ and therefore $Z\to\theta$ is legal.
\end{proofc}

\begin{claim}\label{claim:Y-set-sweep}
Let $i\in[5]$ and let $\theta$ be a color. Suppose that $Y_i$ contains a vertex colored $\theta$ and that $Z$ contains no vertex colored $\theta$. Then we an recolor $Y_i\to\theta$.
\end{claim}

\begin{proofc}
Let $y\in Y_i$ be colored $\theta$. The set $Y_i$ is anticomplete to $U$ by (\ref{item:1.2}), to $Y_{i-2}\cup Y_{i+2}$ by (\ref{item:4.3}), and to $\{i-1,i+1\}$ by definition. On the other hand, $Y_i$ is complete to $\{i,i-2,i+2\}$ by definition and to $Y_{i-1}\cup Y_{i+1}$ by (\ref{item:1.4}). Since $y$ is colored $\theta$, none of the vertices in these complete sets is colored $\theta$. Finally, $Z$ contains no vertex colored $\theta$ by assumption. Therefore no neighbor of $Y_i$ is colored $\theta$, and so $Y_i\to\theta$ is legal.
\end{proofc}

\color{black}

\begin{claim}\label{claim:OPTION-1}
If $\varphi(x_2,x_3) = (\alpha, \beta)$, then we can recolor $\varphi$ so that one of the $I_j$ is monochromatic.
\end{claim}

\begin{proofc}
Since $2$ and $3$ are adjacent and each is complete to $U\cup Y_5$, by renaming colors, we can assume $\varphi(3)=\gamma$ and $\varphi(2)=\delta$. If $Y_4$ contains no vertex colored $\alpha$, then we can recolor $I_4\to \alpha$. Indeed, $Y_1\cup \{5\}$ contains no vertex colored $\alpha$ since such vertices are adjacent to $x_2$ by definition and (\ref{item:1.4}). Thus, we may assume $Y_4$ contains a vertex colored $\alpha$.

Since $I_1$ is complete to $x_1$, $I_2\cup Y_1\cup \{5\}$ are complete to $x_2$, and $Y_5$ is complete to $Y_4$ by (\ref{item:1.4}), the only sets containing vertices colored $\alpha$ are $U,Y_2$, and $Y_4$. Since these sets are independent by (\ref{item:1.2}) and (\ref{item:4.3}), we can recolor $U \cup Y_2 \cup Y_4 \to \alpha$. Next we recolor $Z\to\gamma$. This is legal by \cref{claim:Z-clearing} with $i=3$ and $\theta=\gamma$.

We now consider $\varphi(5)$. First suppose that $\varphi(5)=\gamma$. Then the remaining sets with unknown colors are $\{1,4\}\cup Y_1\cup Y_3\cup Y_5$. Since $Y_3$ is anticomplete to $Y_1\cup Y_5$ by (\ref{item:4.3}), and since $Y_3$ is anticomplete to $4$ by definition, we can recolor $Y_3\to\delta$. Indeed, the vertex $1$ is not colored $\delta$, since it is adjacent to $2$, and the sets $Y_2$ and $Y_4$ have already been recolored with $\alpha$. If $\varphi(4)=\delta$, then $I_2$ is monochromatic with color $\delta$, and we are done. Thus we may assume $\varphi(4)\neq\delta$. Since $Y_5$ is complete to $2$, no vertex of $Y_5$ has color $\delta$. Hence $Y_1\to\delta$ is legal, because $Y_1$ is anticomplete to $Y_3\cup Y_4$ by (\ref{item:4.3}), and has no neighbor colored $\delta$ in $\{1,3,4\}\cup Y_2\cup Y_5$. It remains to handle ${1,4}\cup Y_5$. Since $x_3\in Y_5$ is colored $\beta$ and $Z$ has color $\gamma$, \cref{claim:Y-set-sweep} gives $Y_5\to\beta$. Then $4\to\beta$ is legal, since $4$ is anticomplete to $Y_5\cup{1}$, and all other vertices colored $\beta$ have been cleared. After this step, no vertex in $I_2\cup I_3\cup I_4$ is colored $\varepsilon$, and so we can recolor $I_1\to\varepsilon$. Thus $I_1$ is monochromatic.

Now suppose $\varphi(5)\neq\gamma$. Since $x_3\in Y_5$ is colored $\beta$ and $Z$ has color $\gamma$, \cref{claim:Y-set-sweep} gives $Y_5\to\beta$. Now $\varphi(3)=\gamma$ and $\varphi(5)\neq\gamma$, so no vertex in $\{4,5\}\cup Y_1\cup Y_3$ is colored $\gamma$. Also, $U\cup Y_2\cup Y_4$ is colored $\alpha$, and $Z$ is anticomplete to $\{1\}$. Hence we can recolor $1\to\gamma$. Then $I_1$ is monochromatic with color $\gamma$. This completes the case and proves the desired claim.
\end{proofc}

\begin{claim}\label{claim:OPTION-2}
If $\varphi(x_2,x_3) = (\beta, \alpha)$, then we can recolor $\varphi$ so that one of the $I_j$ is monochromatic.
\end{claim}

\begin{proofc} Since $4$ and $5$ are adjacent to both $x_1$ and $x_2$, and are adjacent to each other, we may assume by renaming the colors that $\varphi(4)=\gamma$ and $\varphi(5) = \delta$. 

If we can recolor $x_2\to\alpha$, then we are in the case $\varphi(x_2,x_3)=(\alpha,\alpha)$, and we are done. Thus $x_2\to\alpha$ is blocked. Since $I_1\cup I_3$ contains no vertex colored $\alpha$ because $I_1$ is complete to $x_1$ and $I_3$ is complete to $x_3$, and since $2$ and $4$ are not colored $\alpha$ because both are adjacent to $x_1$, some vertex of $Y_3$ is colored $\alpha$. Also, $Z$ contains no vertex colored $\alpha$, since $Z$ is complete to $x_1$ by (\ref{item:4.2}). Hence, by \cref{claim:Y-set-sweep}, we recolor $Y_3\to\alpha$. Since $x_3\in Y_5$ is colored $\alpha$ and $Z$ contains no vertex colored $\alpha$, another application of \cref{claim:Y-set-sweep} gives $Y_5\to\alpha$. Finally, by \cref{claim:Z-clearing} with $i=4$ and $\theta=\gamma$, we recolor $Z\to\gamma$.

Since $2$ is adjacent to $x_1$ and $x_2$, up to relabeling colors, we may assume $\varphi(2)\in \{\gamma,\delta,\varepsilon\}$.

First suppose $\varphi(2) = \gamma$. We first recolor $Y_4\to \delta$ which is legal since $U\cup Y_3\cup Y_5 \cup Z \cup \{2,4\}$ have been recolored something distinct from $\delta$, $\{3,5\}\cup Y_1$ are anticomplete to $Y_4$ by (\ref{item:4.3}) and $\{1\}\cup Y_2$ contains no vertices colored $\delta$, since they are complete to $5$. If $\varphi(3)\neq \delta$, then by an analogous argument we could recolor $Y_1\to \delta$ and then $I_3$ would be monochromatic with color $\delta$. Thus, we may assume $\varphi(3) = \delta$.
Since $x_2\in Y_2$ is colored $\beta$ and $Z$ has color $\gamma$, by \cref{claim:Y-set-sweep} we recolor $Y_2\to\beta$.
Finally, note $\varphi(1)\notin \{\alpha,\gamma,\delta\}$ by $x_1, 2$, and $5$. If $\varphi(1)\neq\beta$, then, since $Y_1$ is complete to $1$, and every other set has been recolored to $\alpha,\beta,\gamma$, or $\delta$, we can recolor $I_1\to \varphi(1)$. If $\varphi(1)=\beta$, then every vertex in $I_1\cup I_2\cup I_4$ has been recolored to some color in $\{\alpha,\beta,\gamma,\delta\}$ and hence we can recolor $I_3\to \varepsilon$. In either case we arrive at some $I_j$ monochromatic.

Next suppose $\varphi(2) = \delta$. We first recolor $Y_1\to \delta$ which is possible since it is anticomplete to $\{2,5\}$ and every other set is either recolored to $\alpha,\beta$, or $\gamma$, or complete to $2$ or $5$ which are colored $\delta$. 
Next, since $x_2\in Y_2$ is colored $\beta$ and $Z$ has color $\gamma$, by \cref{claim:Y-set-sweep} we recolor $Y_2\to\beta$. We then recolor $3\to\beta$, which is legal since every possible $\beta$-neighbor of $3$ has either been recolored away from $\beta$ or lies in $Y_2$, which is anticomplete to $3$.
Next recolor $\{1\}\to \gamma$, which is possible since $U\cup Y_1\cup Y_2\cup Y_3\cup Y_5 \cup \{2,3,5\}$ got recolored to $\alpha,\beta$, or $\delta$, $\{4\}\cup Z$ is anticomplete to $\{1\}$ and $Y_4$ contains no vertex colored $\gamma$ as it is complete to $4$. Finally, since the vertices in $I_1\cup I_2\cup I_4$ have been recolored to some color in $\{\alpha, \beta, \gamma, \delta\}$, we can recolor $I_3\to \varepsilon$ as desired.

Finally, suppose $\varphi(2) = \varepsilon$. We first recolor $3\to \beta$. This is possible since $\{2,4,5\}\cup U \cup Y_3\cup Y_5$ are recolored distinct from $\beta$, $\{3\}\cup Y_2\cup Y_4$ are anticomplete to $3$, and $Y_1$ contains no vertex colored $\beta$ since it is complete to $x_2$ by (\ref{item:1.4}). Next we recolor $Y_1 \cup Y_4\to \delta$. This is legal as $\{1\}\cup Y_2$ are complete to $\{5\}$ so no vertices in such sets are colored $\delta$ and every other set is recolored to some color distinct from $\delta$. Thus, $I_3$ is monochromatic with color $\delta$.
\end{proofc}

\begin{claim}\label{claim:OPTION-3}
If $\varphi(x_2,x_3) = (\beta, \beta)$, then we can recolor $\varphi$ so that one of the $I_j$ is monochromatic.
\end{claim}

\begin{proofc}
Since $3$ and $4$ are adjacent, $3$ is complete to $\{x_1,x_3\}$, and $4$  is complete to $\{x_1,x_2\}$, without loss of generality by renaming colors we may assume $\varphi(3)=\gamma$ and $\varphi(4)=\delta$.

If we can recolor $x_2$ or $x_3$ to $\alpha$, then we can recolor some $I_j$ monochromatic by \cref{claim:OPTION-1} or \cref{claim:OPTION-2}. Thus, we may assume that $x_2$ and $x_3$ each have some neighbor say $v_2$ and $v_3$ respectively, colored $\alpha$. Since $\varphi(x_1)=\alpha$ and by (\ref{item:1.4}), $v_2\in Y_1\cup Y_3$ and $v_3\in Y_1\cup Y_4$ (note we may have $v_2 = v_3$). Moreover, the existence of even one vertex colored $\alpha$ in $Y_1\cup Y_3$ and $Y_1\cup Y_4$ implies no vertex of $Y_2\cup Y_5$ is colored $\alpha$. Since $Y_1$ is anticomplete to $U\cup Y_3\cup Y_4$ by (\ref{item:1.2}) and (\ref{item:4.3}), and $V(C)\cup Z$ contain no vertices colored $\alpha$ as they are complete to $x_1$ by definition and (\ref{item:4.2}), we can recolor $Y_1\to \alpha$.

Next we can recolor $\{1\}\to \beta$. Indeed $1$ is anticomplete to $Y_2\cup Y_5 \cup \{3,4\}\cup Z$ by definition, $U\cup Y_1$ are colored or have been recolored to $\alpha$, and since $\{2,5\}\cup Y_3\cup Y_4$ are complete to either $x_2$ or $x_3$, they contain no vertices colored $\beta$. We now consider the possibilities for $\varphi(2,5)$. Since $2$ is complete to $\{x_1,x_2,3\}$, $\varphi(2)\notin \{\alpha, \beta, \gamma\}$ and since $5$ is complete to $\{x_1, x_3, 4\}$, $\varphi(5)\notin \{\alpha, \beta, \delta\}$. Thus, by renaming colors if necessary, we have the following cases for $\varphi(2,5)$ (where $\tau$ is a sixth distinct color):
\[(\delta, \gamma), \quad (\delta, \varepsilon), \quad (\varepsilon, \gamma), \quad (\varepsilon, \varepsilon), \quad (\varepsilon, \tau).\]

First suppose $\varphi(2,5)=(\delta,\gamma)$ and suppose there exists a vertex in $Y_3$ colored $\alpha$. Since $Z$ contains no vertex colored $\alpha$, by \cref{claim:Y-set-sweep} we recolor $Y_3\to\alpha$. Then, by \cref{claim:Z-clearing} with $i=2$ and $\theta=\delta$, we recolor $Z\to\delta$. Since $x_2,x_3\in Y_2\cup Y_5$ are colored $\beta$ and $Z$ has color $\delta$, by \cref{claim:Y-set-sweep} we can recolor $Y_2\to\beta,$ and $Y_5\to\beta.$ Now we recolor $Y_4\to\gamma$, which is legal since $Y_4$ is anticomplete to $\{3,5\}$, and every other possible neighbor colored $\gamma$ has been recolored. Thus no vertex of $I_2\cup I_3\cup I_4$ is colored $\varepsilon$, so we can recolor $I_1\to\varepsilon$.
If $Y_3$ contains no vertex colored $\alpha$, then we can recolor $Y_4\to\alpha$. Indeed, $Y_4$ is anticomplete to $U\cup Y_1\cup Y_2$ by (\ref{item:1.2}) and (\ref{item:4.3}), while $C\cup Z\cup Y_3\cup Y_5$ contains no vertex colored $\alpha$. Next, by \cref{claim:Z-clearing} with $i=3$ and $\theta=\gamma$, we recolor $Z\to\gamma$. Finally, we recolor $Y_3\to\delta$, yielding $I_2$ monochromatic with color $\delta$.

Next suppose $\varphi(2,5)=(\delta,\varepsilon)$ and suppose there exists a vertex in $Y_3$ colored $\alpha$. Since $Z$ contains no vertex colored $\alpha$, by \cref{claim:Y-set-sweep} we recolor $Y_3\to\alpha$. Then, by \cref{claim:Z-clearing} with $i=2$ and $\theta=\delta$, we recolor $Z\to\delta$. Since every set outside of $I_3$ has been recolored except for $Y_2\cup Y_5$, both of which are complete to $5$, no vertex outside $I_3$ is colored $\varepsilon$. Hence we can recolor $I_3\to\varepsilon$. If $Y_3$ contains no vertex colored $\alpha$, then we can recolor $Y_4\to\alpha$. Now, by \cref{claim:Z-clearing} with $i=5$ and $\theta=\varepsilon$, we recolor $Z\to\varepsilon$. Finally, we recolor $Y_3\to\delta$, yielding $I_2$ monochromatic with color $\delta$.

Next suppose $\varphi(2,5)\in\{(\varepsilon,\gamma),(\varepsilon,\varepsilon),(\varepsilon,\tau)\}$ and suppose there exists a vertex in $Y_3$ colored $\alpha$. Since $Z$ contains no vertex colored $\alpha$, by \cref{claim:Y-set-sweep} we recolor $Y_3\to\alpha$. Then, by \cref{claim:Z-clearing} with $i=2$ and $\theta=\varepsilon$, we recolor $Z\to\varepsilon$. Next recolor $1\to\delta$. Now no vertex outside $I_4$ is colored $\beta$: the vertex $1$ has been recolored to $\delta$, the sets $Y_3$ and $Z$ have been recolored to $\alpha$ and $\varepsilon$, respectively, and $Y_4$ contains no vertex colored $\beta$ since it is complete to $x_3$. Hence we can recolor $I_4\to\beta$.

If $Y_3$ contains no vertex colored $\alpha$, we can recolor $Y_4\to \alpha$. Indeed, $Y_4$ is anticomplete to $U\cup Y_1$ by (\ref{item:1.2}) and (\ref{item:4.3}), and $V(C)\cup Y_2\cup Y_3\cup Y_5\cup Z$ contain no vertices colored $\alpha$ by definition or because they are complete to either $U$ or $Y_1$. Next, by \cref{claim:Z-clearing} with $i=5$ and $\theta=\varphi(5)$, we recolor $Z\to\varphi(5)$. Since $x_2,x_3\in Y_2\cup Y_5$ are colored $\beta$ and $Z$ has color $\varphi(5)\neq\beta$, two applications of \cref{claim:Y-set-sweep} give $Y_2\to\beta$, and $Y_5\to\beta$. If $\varphi(5)\neq \varepsilon$ then no vertices in $V(G)\setminus I_2$ are colored $\varepsilon$ and we can recolor $I_2\to \varepsilon$ as desired. If $\varphi(5)=\varepsilon$, then no vertex of $V(G)\setminus I_1$ is colored $\gamma$: the sets $Y_2\cup Y_5$ have been recolored with $\beta$, the set $Y_4$ has been recolored with $\alpha$, the set $Z$ has been recolored with $\varepsilon$, and $Y_3$ contains no vertex colored $\gamma$ since it is complete to $3$. Hence we can recolor $I_1\to\gamma$.

In any case, we can recolor some $I_j$ monochromatic which completes the claim. 
\end{proofc}

\begin{claim}\label{claim:OPTION-4}
If $\varphi(x_2,x_3) = (\beta,\gamma)$, then we can recolor $\varphi$ so that one of the $I_j$ is monochromatic.
\end{claim}

\begin{proofc}
Since $5$ is adjacent to $x_1,x_2$ and $x_3$, we have $\varphi(5)\notin\{\alpha,\beta,\gamma\}$. Thus, by renaming colors if necessary, we may assume that $\varphi(5)=\delta$.

Next we note that no vertex in $Y_2\cup Y_5$ is colored with $\alpha$. Indeed, if some vertex of $Y_2$ were colored $\alpha$, then we could use this vertex as $x_2$ and apply \cref{claim:OPTION-1}; similarly, if some vertex of $Y_5$ were colored $\alpha$, then we could use this vertex as $x_3$ and apply \cref{claim:OPTION-2}. Thus no vertex in $I_1\cup V(C)\cup Y_2\cup Y_5$ is colored $\alpha$. Since $Y_1$ is anticomplete to $U\cup Y_3\cup Y_4$ by (\ref{item:1.2}) and (\ref{item:4.3}), we can recolor $Y_1\to\alpha$.

First suppose that there exists a vertex in $Y_4$ colored $\alpha$. Then we can recolor $Y_4\to\alpha$. Indeed, $Y_4$ is anticomplete to $U\cup Y_1$ by (\ref{item:1.2}) and (\ref{item:4.3}), the set $V(C)\cup Y_2\cup Y_5\cup Z$ contains no vertex colored $\alpha$, and $Y_3$ contains no vertex colored $\alpha$ since it is complete to $Y_4$ by (\ref{item:1.4}). Next, by \cref{claim:Z-clearing} with $i=5$ and $\theta=\delta$, we recolor $Z\to\delta$. Since $x_2\in Y_2$ is colored $\beta$, $x_3\in Y_5$ is colored $\gamma$, and $Z$ has color $\delta$, two applications of \cref{claim:Y-set-sweep} give $Y_2\to\beta$ and $Y_5\to\gamma$.

If we can recolor $U\to\beta$, then after renaming colors we are in \cref{claim:OPTION-1}. Thus we may assume that $U\to\beta$ is not legal. Since $I_2$ is complete to $x_2$ and contains no vertex colored $\beta$, and every other set except possibly $\{1,3\}$ has been recolored to a color distinct from $\beta$, we may assume that either $1$ or $3$ is colored $\beta$. If both $1$ and $3$ are colored $\beta$, then no vertex of $V(G)\setminus I_2$ is colored $\varepsilon$, and so we can recolor $I_2\to\varepsilon$. Thus we may assume that $\varphi(1,3)\in\{(\beta,\varepsilon),(\varepsilon,\beta)\}$.

If $\varphi(1,3)=(\beta,\varepsilon)$, then we recolor $2\to\delta$. This is legal since $2$ is anticomplete to ${4,5}\cup Y_3\cup Z$, and every other possible neighbor of $2$ has a color distinct from $\delta$. Now ${4}\cup Y_3$ is complete to $3$, and every other set except ${3}$ has been recolored to a color distinct from $\varepsilon$. Hence we can recolor $I_1\to\varepsilon$.

If $\varphi(1,3)=(\varepsilon,\beta)$, then we recolor $1\to\beta$. Indeed, since $3$ is colored $\beta$, the set ${2,4}\cup Y_3$ contains no vertex colored $\beta$. Also, $1$ is anticomplete to $Y_2$, and every other possible neighbor of $1$ has a color distinct from $\beta$. Hence $1\to\beta$ is legal. Now no vertex of $V(G)\setminus I_2$ is colored $\varepsilon$, and so we can recolor $I_2\to\varepsilon$.

Finally, suppose that no vertex in $Y_4$ is colored $\alpha$. We can recolor $Y_3\to\alpha$. Indeed, $Y_3$ is anticomplete to $U\cup Y_1\cup Y_5$ by (\ref{item:1.2}) and (\ref{item:4.3}), and $V(C)\cup Y_2\cup Y_4\cup Y_5\cup Z$ contains no vertex colored $\alpha$.

Since $2$ is adjacent to $x_1,x_2$ and $x_3$, we have $\varphi(2)\notin\{\alpha,\beta,\gamma\}$. Since $\varphi(5)=\delta$, either $\varphi(2)=\delta$, or by renaming colors if necessary, $\varphi(2)=\varepsilon$.

First suppose $\varphi(2)=\delta$. By \cref{claim:Z-clearing} with $i=2$ and $\theta=\delta$, we recolor $Z\to\delta$. Since $x_2\in Y_2$ is colored $\beta$, $x_3\in Y_5$ is colored $\gamma$, and $Z$ has color $\delta$, two applications of \cref{claim:Y-set-sweep} give $Y_2\to\beta$ and $Y_5\to\gamma$. Next we recolor $\{1,4\}\to\gamma$. This is legal since $\{1,4\}$ is anticomplete to $Y_5$, and $\{3\}\cup Y_4$ contains no vertex colored $\gamma$ since it is complete to $x_3$. Then we recolor $3\to\beta$, which is legal since $3$ is anticomplete to $Y_2\cup Y_4$, and every other possible neighbor of $3$ has a color distinct from $\beta$. Finally, we recolor $Y_4\to\beta$, which is legal since $Y_4$ is anticomplete to $\{3\} \cup Y_2$, and every other possible neighbor of $Y_4$ has a color distinct from $\beta$. Therefore no vertex of $V(G)\setminus I_4$ is colored $\varepsilon$, and so we can recolor $I_4\to\varepsilon$.

Now suppose $\varphi(2)=\varepsilon$. By \cref{claim:Z-clearing} with $i=2$ and $\theta=\varepsilon$, we recolor $Z\to\varepsilon$. Next we recolor $4\to\varepsilon$. This is legal since $4$ is anticomplete to $Z$, while $\{1,3\}\cup Y_2\cup Y_4\cup Y_5$ is complete to $2$, and $\{5\}\cup U\cup Y_1\cup Y_3$ has been recolored with colors distinct from $\varepsilon$. Then two applications of \cref{claim:Y-set-sweep} give $Y_2\to\beta$ and $Y_5\to\gamma$. Next recolor $3\to\beta$, which is legal since every possible neighbor of $3$ other than vertices in $\{1\}\cup Y_2\cup Y_4$ has a color distinct from $\beta$, and $3$ is anticomplete to $\{1\}\cup Y_2\cup Y_4$. Finally, we recolor $Y_4\to\delta$ and then $Y_1\to\delta$. The first step is legal since $Y_4$ is anticomplete to $5$, and every other possible neighbor of $Y_4$ has a color distinct from $\delta$. The second step is legal since $Y_1$ is anticomplete to $\{5\} \cup Y_4$, and every other possible neighbor of $Y_1$ has a color distinct from $\delta$. Hence $I_3$ is monochromatic with color $\delta$. This completes the claim.
\end{proofc}

By \cref{claim:OPTION-1}-\cref{claim:OPTION-4}, for all cases we can recolor $\varphi$ so that one of the $I_j$ is monochromatic and apply Lemma \ref{claim:CoverI}. This completes the proof. \end{proof}

\section{Proof of \cref{theorem:main2}}\label{sec:SECOND-THEOREM}

Before proving \cref{theorem:main2}, we prove that two specific graphs are recolorable. These will be graphs that we reduce to in the proof of \cref{theorem:main2}. First, let $S_2$ be the graph with vertex set $V(C)\cup \{y_i:i\in[5]\}$, where $C=12345$ is a $5$-cycle. The edges of $S_2$ are as follows, with indices taken modulo $5$. Each vertex $y_i$ is adjacent to $i$, $i-2$, and $i+2$, and $y_iy_{i+1}\in E(S_2)$ for every $i\in[5]$. See \cref{fig:S2} for an illustration of $S_2$. 

Next, let $S_3$ be the graph with vertex set $V(C)\cup \bigcup_{i=1}^5 Y_i$, where $C=12345$ is a $5$-cycle and $Y_i=\{y_i^+,y_i^-\}$ for each $i\in[5]$. The edges of $S_3$ are as follows, with indices taken modulo $5$. Each vertex of $Y_i$ is adjacent to $i$, $i-2$, and $i+2$; the set $Y_i$ is complete to $Y_{i-1}\cup Y_{i+1}$, and $y_i^+y_{i+2}^-$ for each $i\in[5]$. See \cref{fig:S3} for an illustration of $S_3$.

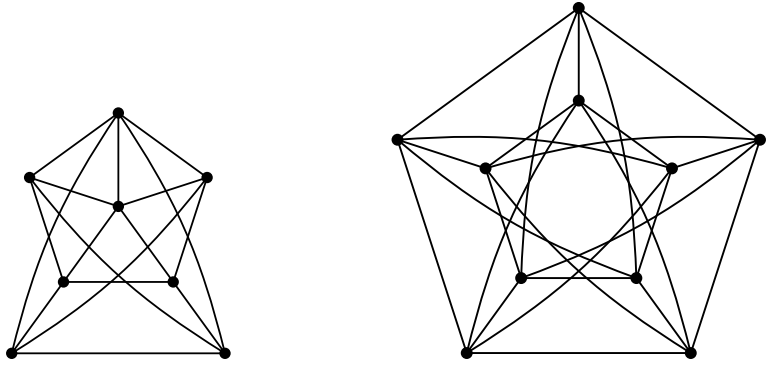
\begin{figure}[h!]
    \centering
    \begin{tikzpicture}[
    line cap=round,
    line join=round,
    vertex/.style={circle, fill=black, inner sep=2.2pt},
    every path/.style={draw, line width=1pt}
]

\def\r{1.8}
\def\R{3.5}

\coordinate (1) at (90:\r);
\coordinate (2) at (18:\r);
\coordinate (3) at (-54:\r);
\coordinate (4) at (-126:\r);
\coordinate (5) at (162:\r);

\coordinate (y1) at (90:\R);
\coordinate (y2) at (18:\R);
\coordinate (y3) at (-54:\R);
\coordinate (y4) at (-126:\R);
\coordinate (y5) at (162:\R);

\draw (1) -- (2);
\draw (2) -- (3);
\draw (3) -- (4);
\draw (4) -- (5);
\draw (5) -- (1);

\draw (y3) -- (y4);



\draw (y3) to (3);
\draw[bend left = 10] (y3) to (5);
\draw[bend left = -10] (y3) to (1);

\draw (y4) to (4);
\draw[bend left = 10] (y4) to (1);
\draw[bend left = -10] (y4) to (2);


\draw (0,0) to (1);
\draw (0,0) to (2);
\draw (0,0) to (3);
\draw (0,0) to (4);
\draw (0,0) to (5);

\node[vertex] at (1) {};
\node[vertex] at (2) {};
\node[vertex] at (3) {};
\node[vertex] at (4) {};
\node[vertex] at (5) {};

\node[vertex] at (y3) {};
\node[vertex] at (y4) {};

\node[vertex] at (0,0) {};


\end{tikzpicture}
    \hspace{2cm}\begin{tikzpicture}[
    line cap=round,
    line join=round,
    vertex/.style={circle, fill=black, inner sep=2.2pt},
    every path/.style={draw, line width=1pt}
]

\def\r{1.8}
\def\R{3.5}

\coordinate (1) at (90:\r);
\coordinate (2) at (18:\r);
\coordinate (3) at (-54:\r);
\coordinate (4) at (-126:\r);
\coordinate (5) at (162:\r);

\coordinate (y1) at (90:\R);
\coordinate (y2) at (18:\R);
\coordinate (y3) at (-54:\R);
\coordinate (y4) at (-126:\R);
\coordinate (y5) at (162:\R);

\draw (1) -- (2);
\draw (2) -- (3);
\draw (3) -- (4);
\draw (4) -- (5);
\draw (5) -- (1);

\draw (y1) -- (y2);
\draw (y2) -- (y3);
\draw (y3) -- (y4);
\draw (y4) -- (y5);
\draw (y5) -- (y1);

\draw (y1) to (1);
\draw[bend left = 10] (y1) to (3);
\draw[bend left = -10] (y1) to (4);

\draw (y2) to (2);
\draw[bend left = 10] (y2) to (4);
\draw[bend left = -10] (y2) to (5);

\draw (y3) to (3);
\draw[bend left = 10] (y3) to (5);
\draw[bend left = -10] (y3) to (1);

\draw (y4) to (4);
\draw[bend left = 10] (y4) to (1);
\draw[bend left = -10] (y4) to (2);

\draw (y5) to (5);
\draw[bend left = 10] (y5) to (2);
\draw[bend left = -10] (y5) to (3);

\node[vertex] at (1) {};
\node[vertex] at (2) {};
\node[vertex] at (3) {};
\node[vertex] at (4) {};
\node[vertex] at (5) {};

\node[vertex] at (y1) {};
\node[vertex] at (y2) {};
\node[vertex] at (y3) {};
\node[vertex] at (y4) {};
\node[vertex] at (y5) {};


\end{tikzpicture}
    \caption{The graphs $S_1$ (left) and $S_2$ (right)}
    \label{fig:S2}
\end{figure}

\begin{lemma}\label{lemma:S2-recolorable}
$S_2$ is recolorable. 
\end{lemma}

\begin{proof}
It is easy to verify that the following four independent sets partition $V(S_2)$:
\[
I_1=\{2,y_1\},\qquad
I_2=\{5,y_4\},\qquad
I_3=\{1,3,y_2\},\qquad
I_4=\{4,y_3,y_5\}.
\]
We say that a $4$-coloring of $S_2$ is canonical if its color classes are precisely $I_1,I_2,I_3,I_4$, up to renaming colors. To show that $S_2$ is recolorable, we will show that for every $\ell\geq 5$ and every $\ell$-coloring $\varphi$ of $S_3$, we can recolor $\varphi$ to a canonical coloring. Then, by the Renaming Lemma, $\mathcal{R}_{\ell}(S_2)$ is connected. 

By Lemma \ref{claim:CoverI}, it is enough to recolor $\varphi$ so that one of the sets $I_j$ is monochromatic. Indeed, if $I_j$ is monochromatic, say with color $\alpha$, then deleting the color class $I_\alpha$ leaves an induced subgraph of $S_2-I_j$, which is $3$-colorable; since this subgraph is also $2K_2$-free, it is recolorable by \cref{theo:3color}. 

Given that $\ell\geq5$, we let $\{\alpha,\beta, \gamma,\delta,\varepsilon\}$ be five distinct colors. Without loss of generality, we may assume
$\varphi(1)=\alpha$, $\varphi(2)=\beta$, and $\varphi(y_4)=\gamma$. Since $y_3$ is adjacent to both $1$ and $y_4$, but not to $2$, we may also assume by renaming $\varphi(y_3)\in\{\beta,\delta\}$.

First suppose $\varphi(y_3)=\beta$. If $\varphi(y_2)\neq\gamma$, then $5\to\gamma$ is legal, and $I_2$ is monochromatic with color $\gamma$. Thus we may assume $\varphi(y_2)=\gamma$. Since $5$ is adjacent to vertices of colors $\alpha,\beta,\gamma$, we may assume, after relabeling the remaining colors, that $\varphi(5)=\delta$. If $\varphi(4)\neq\beta$, then $y_1\to\beta$ is legal, and $I_1$ is monochromatic with color $\beta$. Hence we may assume $\varphi(4)=\beta$. But then $y_4\to\delta$ is legal: the vertex $y_4$ is not adjacent to $5$, and its only possible $\delta$-neighbor among the outer vertices would be $y_5$, which cannot have color $\delta$ because $y_5$ is adjacent to $5$. Therefore $I_2$ is monochromatic with color $\delta$. \

Now suppose $\varphi(y_3)=\delta$. If $\varphi(y_1)\neq\delta$, then $y_5\to\delta$ and then $4\to\delta$ are legal, making $I_4$ monochromatic with color $\delta$. Thus we may assume $\varphi(y_1)=\delta$. Next, if $\varphi(y_2)\neq\gamma$, then $5\to\gamma$ is legal, and $I_2$ is monochromatic with color $\gamma$. Hence we may assume $\varphi(y_2)=\gamma$. But now $2\to\delta$ is legal, and so $I_1$ is monochromatic with color $\delta$.

In every case, we can recolor to make one of the sets $I_1,I_2,I_3,I_4$ monochromatic. Therefore, by Lemma \ref{claim:CoverI}, the coloring can be recolored to a canonical coloring. Hence $S_2$ is recolorable. 
\end{proof}

\begin{figure}[h!]
    \centering
    \begin{tikzpicture}[
    line cap=round,
    line join=round,
    cycleedge/.style={draw, line width=1.3pt},
    mainedge/.style={draw, line width=1.7pt},
    chordedge/.style={draw, line width=1.25pt},
    blob/.style={draw=black!50, fill=gray!10, line width=1pt},
    cyclev/.style={
        circle,
        fill=black,
        draw=black,
        minimum size=18pt,
        inner sep=0pt,
        text=white,
        font=\scriptsize\bfseries
    },
    yv/.style={
        circle,
        fill=black,
        draw=black,
        minimum size=15pt,
        inner sep=0pt,
        text=white,
        font=\scriptsize\bfseries
    }
]

\def\r{2.15}
\def\R{5.00}
\def\d{0.42}
\def\A{7.25}

\foreach \i/\ang in {
    1/90,
    2/18,
    3/-54,
    4/-126,
    5/162
}{
    \coordinate (c\i) at (\ang:\r);             
    \coordinate (b\i) at (\ang:\R);             
    \coordinate (p\i) at ($(b\i)+(\ang+90:\d)$); 
    \coordinate (m\i) at ($(b\i)+(\ang-90:\d)$); 
}

\coordinate (a13) at (18:\A);
\coordinate (a24) at (-54:\A);
\coordinate (a35) at (-126:\A);
\coordinate (a41) at (162:\A);
\coordinate (a52) at (90:\A);

\draw[cycleedge] (c1) -- (c2);
\draw[cycleedge] (c2) -- (c3);
\draw[cycleedge] (c3) -- (c4);
\draw[cycleedge] (c4) -- (c5);
\draw[cycleedge] (c5) -- (c1);

\draw[mainedge] (b1) .. controls (72:5.55)   and (36:5.55)   .. (b2);
\draw[mainedge] (b2) .. controls (0:5.55)    and (-36:5.55)  .. (b3);
\draw[mainedge] (b3) .. controls (-72:5.55)  and (-108:5.55) .. (b4);
\draw[mainedge] (b4) .. controls (-144:5.55) and (180:5.55)  .. (b5);
\draw[mainedge] (b5) .. controls (144:5.55)  and (108:5.55)  .. (b1);

\draw[mainedge] (b1) -- (c1);
\draw[mainedge] (b1) to[bend left=8]  (c3);
\draw[mainedge] (b1) to[bend right=8] (c4);

\draw[mainedge] (b2) -- (c2);
\draw[mainedge] (b2) to[bend left=8]  (c4);
\draw[mainedge] (b2) to[bend right=8] (c5);

\draw[mainedge] (b3) -- (c3);
\draw[mainedge] (b3) to[bend left=8]  (c5);
\draw[mainedge] (b3) to[bend right=8] (c1);

\draw[mainedge] (b4) -- (c4);
\draw[mainedge] (b4) to[bend left=8]  (c1);
\draw[mainedge] (b4) to[bend right=8] (c2);

\draw[mainedge] (b5) -- (c5);
\draw[mainedge] (b5) to[bend left=8]  (c2);
\draw[mainedge] (b5) to[bend right=8] (c3);

\foreach \i/\ang in {
    1/90,
    2/18,
    3/-54,
    4/-126,
    5/162
}{
    \begin{scope}[rotate around={\ang+90:(b\i)}]
        \filldraw[blob] (b\i) ellipse[x radius=1.08cm, y radius=0.70cm];
    \end{scope}
}

\def\A{7.55}

\coordinate (a13) at (18:\A);
\coordinate (a24) at (-54:\A);
\coordinate (a35) at (-126:\A);
\coordinate (a41) at (162:\A);
\coordinate (a52) at (90:\A);

\draw[chordedge]
    (p1) to[out=45,in=112,looseness=.95] (a13)
         to[out=-68,in=-45,looseness=.95] (m3);

\draw[chordedge]
    (p2) to[out=-27,in=40,looseness=.95] (a24)
         to[out=-140,in=-117,looseness=.95] (m4);

\draw[chordedge]
    (p3) to[out=-99,in=-32,looseness=.95] (a35)
         to[out=148,in=171,looseness=.95] (m5);

\draw[chordedge]
    (p4) to[out=-171,in=-104,looseness=.95] (a41)
         to[out=76,in=99,looseness=.95] (m1);

\draw[chordedge]
    (p5) to[out=117,in=-176,looseness=.95] (a52)
         to[out=4,in=27,looseness=.95] (m2);
\node[cyclev] at (c1) {1};
\node[cyclev] at (c2) {2};
\node[cyclev] at (c3) {3};
\node[cyclev] at (c4) {4};
\node[cyclev] at (c5) {5};

\node[yv] at (p1) {$+$};
\node[yv] at (m1) {$-$};

\node[yv] at (p2) {$+$};
\node[yv] at (m2) {$-$};

\node[yv] at (p3) {$+$};
\node[yv] at (m3) {$-$};

\node[yv] at (p4) {$+$};
\node[yv] at (m4) {$-$};

\node[yv] at (p5) {$+$};
\node[yv] at (m5) {$-$};


\end{tikzpicture}
    \caption{The graph $S_3$}
    \label{fig:S3}
\end{figure}
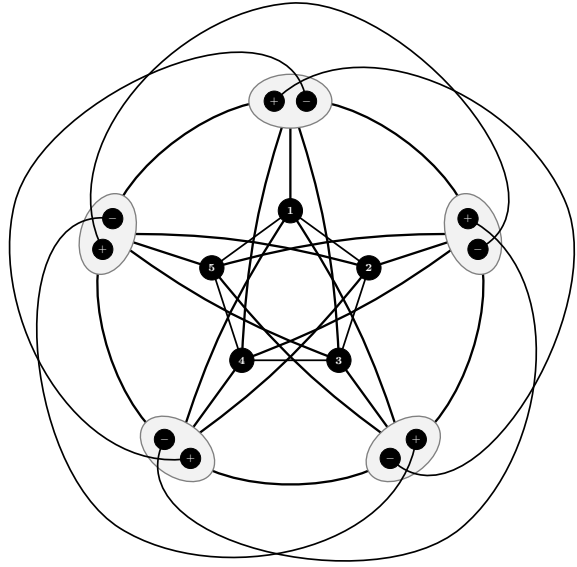

\begin{lemma}\label{lemma:S3-recolorable}
$S_3$ is recolorable. 
\end{lemma}

\begin{proof}
It is easy to verify that the following four independent sets partition $V(S_3)$:
\[I_1 =  \{1, y_2^+, y_5^-\}, \qquad
I_2 =  \{2,5,y_1^+, y_1^-\}, \qquad
I_3 = \{3, y_4^+, y_4^-, y_2^-\}, \qquad
I_4 =  \{4, y_3^+, y_3^-, y_5^+\}.\]

We say that a $4$-coloring of $S_3$ is canonical if its color classes are precisely $I_1,I_2,I_3,I_4$, up to renaming colors. To show that $S_3$ is recolorable, we will show that for every $\ell\geq 5$ and every $\ell$-coloring $\varphi$ of $S_3$, we can recolor $\varphi$ to a canonical coloring. Then, by the Renaming Lemma, $\mathcal{R}_{\ell}(S_3)$ is connected.

By Lemma \ref{claim:CoverI}, it is enough to recolor $\varphi$ so that one of the sets $I_j$ is monochromatic. Indeed, if $I_j$ is monochromatic with color $\alpha$, then deleting the color class $I_\alpha$ leaves an induced subgraph of $S_3-I_j$, which is $3$-colorable; since this subgraph is also $2K_2$-free, it is recolorable by \cref{theo:3color}. 

Since $\ell\geq 5$, we let $\{\alpha,\beta,\gamma, \delta, \varepsilon\}$ be $5$ distinct colors. By renaming colors if necessary, we may assume that $\varphi(1)=\alpha$, $\varphi(2)=\beta$, and $\varphi(y_4^+)=\gamma$.

First suppose $\varphi(5)=\beta$. Since $\{2,5\}$ is a dominating set, no other vertex is colored $\beta$. Hence we can recolor every vertex of $I_2$ with $\beta$, making $I_2$ monochromatic. Thus, we may assume that $\varphi(5)\neq\beta$. Moreover, if we could recolor $5$ with $\beta$, then we would reduce to the previous case. Hence some neighbor of $5$ is colored $\beta$. Since every vertex in $Y_2\cup Y_5$ is adjacent to $2$, no vertex in $Y_2\cup Y_5$ is colored $\beta$. Therefore at least one vertex in $\{4,y_3^+,y_3^-\}$ is colored $\beta$.

First suppose $\varphi(4)=\beta$. We first recolor both vertices in $\{y_3^+,y_3^-\}$ to $\beta$. This is possible since $\{2,4,y_3^+,y_3^-\}$ is an independent set, and every other vertex of $S_3$ is adjacent to either $2$ or $4$, and hence is not colored $\beta$. Since $y_1^-$ is adjacent to $1$, $4$, and $y_4^+$, we may assume, after possibly renaming, that $\varphi(y_1^-)=\delta$.

We now recolor $5$ and $y_1^+$ to $\delta$. Indeed, every neighbor of $5$ that is not already colored $\alpha$ or $\beta$ lies in $Y_2\cup Y_5$, and every vertex in $Y_2\cup Y_5$ is adjacent to $y_1^-$, so none of these vertices is colored $\delta$. Similarly, every neighbor of $y_1^+$ is either one of $1,3,4$, a vertex in $Y_2\cup Y_5$, or $y_3^-$; none of these vertices is colored $\delta$. Next, we recolor $y_2^+$ to $\alpha$. This is legal since $y_2^+$ is adjacent to $2$, $4$, $5$, every vertex in $Y_1\cup Y_3$, and $y_4^-$, none of which is colored $\alpha$. Finally, we recolor all vertices of $I_3$ to $\gamma$. This is legal because $I_3$ is independent, $y_4^+$ already has color $\gamma$, and every possible neighbor outside $I_3$ is adjacent to one of the vertices already colored $\gamma$ or has already been recolored away from $\gamma$. Thus $I_3$ is monochromatic.

Next suppose $\varphi(4)\neq \beta$ and $\varphi(y_3^-)=\beta$. We first recolor $y_3^+$ to $\beta$. This is legal because $y_3^+$ and $y_3^-$ are nonadjacent, and every neighbor of $y_3^+$ outside $\{4,y_3^-\}$ is adjacent to $2$, and hence is not colored $\beta$; moreover we recall that $\varphi(5) \neq \beta$. We now consider two cases according to $\varphi(4)$.

\textbf{Case 1: $\varphi(4)=\alpha$}. We first recolor $y_5^+$ and $y_5^-$ to $\alpha$. This is legal because each neighbor of $Y_5$ is adjacent to either $1$ or $4$, and hence no neighbor of $Y_5$ is colored $\alpha$. We then recolor $y_1^-$ to $\beta$, which is legal because every possible neighbor of $y_1^-$ colored $\beta$, lies in $Y_2\cup Y_5$, and vertices in $Y_2$ are adjacent to $2$ while vertices in $Y_5$ have just been recolored to $\alpha$. Finally, we recolor $5$ to $\varphi(y_1^+)$. This is legal because every neighbor of $5$ is adjacent to $y_1^+$, except possibly $y_3^+$, and $y_3^+$ is colored $\beta$, while $\varphi(y_1^+)\neq \beta$. Since $y_1^+$ is adjacent to $1$ and to $y_3^-$, we have $\varphi(y_1^+)\notin\{\alpha,\beta\}$. Since $\gamma$ is the only remaining fixed color at this point, up to relabeling, we may assume that $\varphi(y_1^+)\in\{\gamma,\delta\}$.

Suppose first that $\varphi(y_1^+)=\gamma$, so $5$ has also been recolored to $\gamma$. Since $3$ is adjacent to $2$, $4$, and $y_1^+$, we have $\varphi(3)\notin\{\alpha,\beta,\gamma\}$, and hence, up to relabeling, we may assume $\varphi(3)=\delta$. Now recolor $y_4^+$ and $y_2^-$ to $\delta$. These steps are legal since every neighbor of $y_4^+$ or $y_2^-$ that could have color $\delta$ is adjacent to either $3$ or $y_1^+$, both of which are colored $\delta$. If $y_4^-$ can be recolored to $\delta$, then $I_3$ is monochromatic with color $\delta$. Otherwise, the only possible neighbor of $y_4^-$ colored $\delta$, is $y_2^+$, so $\varphi(y_2^+)=\delta$. In this case, no vertex outside $I_3$ is colored $\varepsilon$, and so we can recolor all of $I_3$ to $\varepsilon$.

Now suppose $\varphi(y_1^+)=\delta$, so $5$ has also been recolored to $\delta$. We recolor $y_1^-$ to $\delta$, which is legal since every neighbor of $y_1^-$ that could have color $\delta$ is adjacent to either $5$ or $y_1^+$. If $2$ can be recolored to $\delta$, then $I_2$ is monochromatic with color $\delta$. Otherwise, the only possible $\delta$-colored neighbor of $2$ is $y_4^-$, so $\varphi(y_4^-)=\delta$. Then we recolor $y_2^+$ to $\gamma$, which is legal since all possible vertices colored $\gamma$ are nonneighbors of $y_2^+$. Finally, no vertex outside $I_3$ is colored $\varepsilon$, so we recolor all of $I_3$ to $\varepsilon$.

\textbf{Case 2: $\varphi(4)\neq\alpha$}. Since $\varphi(4)\neq\beta$ and $4$ is adjacent to $y_4^+$, after possibly relabeling, we may assume that $\varphi(4)=\delta$. We first recolor $y_1^-$ to $\beta$, which is legal since every possible neighbor of $y_1^-$ colored $\beta$, lies in $Y_2\cup Y_5\cup \{3\}$, and these vertices are adjacent to $2$. Next, recolor $y_3^-$ and $y_5^+$ to $\delta$. These recolorings are legal because the only neighbors of these vertices not adjacent to $4$ are already colored $\alpha$ or $\beta$. If $y_3^+$ can be recolored to $\delta$, then $I_4$ is monochromatic with color $\delta$. Thus, we may assume that $y_3^+$ cannot be recolored to $\delta$. Since the only possible $\delta$-colored neighbor of $y_3^+$ is $y_5^-$, we have $\varphi(y_5^-)=\delta$. Now recolor $y_2^+$ to $\alpha$, then recolor $3$ and $y_2^-$ to $\alpha$, and finally recolor $5$ to $\gamma$. These steps are legal since, after the previous recolorings, the neighbors of $y_2^+$, $3$, and $y_2^-$ avoid $\alpha$, while the neighbors of $5$ avoid $\gamma$. At this point, every vertex has color in $\{\alpha,\beta,\gamma,\delta\}$ except possibly $y_1^+$ and $y_4^-$. If $y_4^-$ is not colored $\varepsilon$, then we recolor all of $I_2$ to $\varepsilon$. If $y_4^-$ is colored $\varepsilon$ and $y_1^+$ is not, then we recolor all of $I_3$ to $\varepsilon$. Finally, if both $y_1^+$ and $y_4^-$ are colored $\varepsilon$, then first recolor $y_4^-$ to $\gamma$, which is legal since $y_4^+$ is not adjacent to $y_4^-$, and then recolor all of $I_2$ to $\varepsilon$. In every case, one of the sets $I_j$ becomes monochromatic.

Finally suppose $\varphi(4),\varphi(y_3^-)\neq\beta$ and $\varphi(y_3^+)=\beta$. We first recolor both $y_1^+$ and $y_1^-$ to $\beta$. This is legal since every neighbor of these vertices is either adjacent to $2$, and hence is not colored $\beta$, or lies in $\{4,y_3^-,y_4^+\}$, none of which is colored $\beta$. Since $y_3^-$ is adjacent to $1$ and $y_4^+$, and is not colored $\beta$, we may assume, after possibly relabeling, that $\varphi(y_3^-)=\delta$. We recolor $y_5^+$ to $\delta$, then recolor $3$ to $\gamma$, and then recolor $y_5^-$ to $\alpha$. These steps are legal because the possible obstructions to these recolorings are adjacent to $y_3^-$, $y_4^+$, and $1$, respectively. If $4$ is not colored $\alpha$, then we recolor $y_2^+$ to $\alpha$, making $I_1$ monochromatic. Otherwise, $4$ is colored $\alpha$, and we recolor all of $I_4$ to $\delta$, which is legal since all the vertices not in $I_4$ contain no vertices colored $\delta$. In either case, some $I_j$ becomes monochromatic.

In every case, we can recolor to make one of the sets $I_1,I_2,I_3,I_4$ monochromatic. Therefore, by Lemma \ref{claim:CoverI}, the coloring can be recolored to a canonical coloring. Hence $S_3$ is recolorable. 
\end{proof}

We now prove \cref{theorem:main2} and recall the theorem for convenience.

\begin{theorem*}[\ref{theorem:main2}]
Let $H_a,H_b$ be distinct graphs in $\{H_2,H_3,H_4\}$. Then every $(2K_2,K_4,H_a,H_b)$-free graph is recolorable.
\end{theorem*}

\begin{proof}
Let $G$ be a $(2K_2,K_4,H_a,H_b)$-free graph, where $a\neq b$. By \cref{lem: comparable}, we may assume that $G$ has no comparable vertices. Moreover, if $G$ is $C_5$-free, then by \cref{theorem:main2}, we know that $G$ is recolorable. Thus, we may assume that $G$ contains an induced $5$-cycle. Let $C=12345$ be an induced $5$-cycle in $G$. We write \[
V(G)=V(C)\cup U\cup Z\cup \bigcup_{i=1}^5(R_i\cup Y_i\cup F_i)
\]
using the sets defined in \cref{sec:THE-STRUCTURE}. We now consider the possible choices of $\{a,b\}$.

First suppose $\{a,b\}=\{3,4\}$, that is, $G$ is $(2K_2,K_4,H_3,H_4)$-free. If $G$ contains a $W_5$, then $G$ is recolorable by \cref{theorem:main1}. Thus, we may assume that $G$ contains no $W_5$, and so $U=\emptyset$. We have $Y_i\cup F_i=\emptyset$ for all $i\in[5]$ by (\ref{item:2.1}) and (\ref{item:3.1}), and $Z=\emptyset$ by (\ref{item:2.4}). Thus,
\[
V(G)=V(C)\cup \bigcup_{i=1}^5 R_i.
\]

We now claim that $R_i=\emptyset$. To see this, suppose $r\in R_i$. By (\ref{item:1.5}) and (\ref{item:2.3}), $R_i$ is complete to $R_{i-1}\cup R_{i+1}$ and anticomplete to $R_{i-2}\cup R_{i+2}$. Since $R_i$ is independent, it follows that
\[
N(r)=\{i-1,i+1\}\cup R_{i-1}\cup R_{i+1}\subseteq N(i),
\]
contradicting that $G$ has no comparable vertices. Thus, $R_i=\emptyset$ for every $i\in[5]$. Hence $G\cong C_5$, which is $3$-colorable and therefore recolorable by \cref{theo:3color}.

Next suppose $\{a,b\}=\{2,4\}$, that is, $G$ is $(2K_2,K_4,H_2,H_4)$-free. If $G$ contains a $W_5$, then $G$ is recolorable by \cref{theorem:main1}. Thus, we may assume that $G$ contains no $W_5$, and so $U=\emptyset$. Since $G$ is $H_2$-free, $R_i=\emptyset$ for all $i\in[5]$, and since $G$ is $H_4$-free, $F_i=\emptyset$ for all $i\in[5]$. Thus,
\[
V(G)=V(C)\cup Z\cup \bigcup_{i=1}^5 Y_i.
\]

\begin{claim}
$Z=\emptyset$.
\end{claim}

\begin{proofc}
Suppose, for a contradiction, that $z\in Z$. Let
\[
\operatorname{supp}(z)=\{i: N(z)\cap Y_i\neq\emptyset\}.\]

Since we may assume $G$ is connected, $\operatorname{supp}(z)\neq\emptyset$. By symmetry, we may assume that $1\in \operatorname{supp}(z)$. Since $G$ has no comparable vertices, there must exist some $v\in N(z)\setminus N(1)$. In other words, $\{2,5\}\cap \operatorname{supp}(z)\neq\emptyset$. By symmetry, we may assume that $2\in \operatorname{supp}(z)$. Again, since $G$ has no comparable vertices, there must exist some $u\in N(z)\setminus N(4)$. In other words, $\{3,5\}\cap \operatorname{supp}(z)\neq\emptyset$. By symmetry, we may assume that $3\in \operatorname{supp}(z)$.

Let $y_1,y_2,y_3\in N(z)$ with $y_i\in Y_i$. By (\ref{item:1.4}), $y_1y_2,y_2y_3\in E(G)$. Moreover, $y_1y_3\notin E(G)$, otherwise ${y_1,y_2,y_3,z}$ induces a $K_4$. But then $4,y_1,z,y_3,5$ induce a $5$-cycle, and $y_2$ is adjacent to every vertex of this cycle. Thus these vertices together with $y_2$ induce a $W_5$, a contradiction.
\end{proofc}

By the previous claim, $V(G)=V(C)\cup \bigcup_{i=1}^5 Y_i$. By (\ref{item:1.4}), each vertex in $Y_i$ is anticomplete to at least one of $Y_{i-2}$ and $Y_{i+2}$. We claim that $|Y_i|\leq 2$ for all $i\in[5]$. Indeed, suppose $|Y_i|\geq 3$. Then, by the pigeonhole principle, without loss of generality there exist distinct vertices $y,y'\in Y_i$ that are both anticomplete to $Y_{i-2}$. Since $y$ and $y'$ are not comparable, there exist vertices $x\in N(y)\setminus N(y')$ and $x'\in N(y')\setminus N(y)$. Since $y$ and $y'$ have the same neighbors on $C$, and since $Y_i$ is complete to $Y_{i-1}\cup Y_{i+1}$ by (\ref{item:1.4}), while both $y$ and $y'$ are anticomplete to $Y_{i-2}$, we must have $x,x'\in Y_{i+2}$. But then $xy$ and $x'y'$ induce a $2K_2$, a contradiction.

\begin{claim}
$G\in\{C_5, S_2,S_3\}$.
\end{claim}

\begin{proofc}
For each vertex $y\in Y_i$, by (\ref{item:1.4}), $y$ is anticomplete to at least one of $Y_{i-2}$ and $Y_{i+2}$. Thus, every vertex of $Y_i$ has one of the following three types. We write $y_i^+$ for a vertex that is anticomplete to $Y_{i-2}$ and has a neighbor in $Y_{i+2}$, $y_i^-$ for a vertex that is anticomplete to $Y_{i+2}$ and has a neighbor in $Y_{i-2}$, and $y_i^0$ for a vertex that is anticomplete to both $Y_{i-2}$ and $Y_{i+2}$. Since $y_i^0$ is comparable with $y_i^+$ and $y_i^-$, then we have $|Y_i|\leq 2$ for every $i\in[5]$.

We first note that if $|Y_i|=2$, then $Y_i=\{y_i^+,y_i^-\}$ since $y_i^0$ is comparable with $y_i^+$ and $y_i^-$, as is two vertices of type $y_i^+$ or $y_i^-$. Hence, for every $i\in[5]$,
\[Y_i\in\{\emptyset,\{y_i^0\},\{y_i^+\},\{y_i^-\},\{y_i^+,y_i^-\}\}.\] 
We note that by definition, if $y_i^+$ exists, then $y_{i+2}^-$ exists, and if $y_i^-$ exists, then $y_{i-2}^+$ exists.

First suppose no $y^+$ vertex exists. Then no $y^-$ vertex exists either, and so every nonempty $Y_i$ consists of a single vertex $y_i^0$. We claim that either all of the sets $Y_i$ are empty, or all of them are nonempty. Indeed, suppose $Y_i=\{y_i^0\}$. If $Y_{i-2}=\emptyset$, then $N(i+1)\subseteq N(y_i^0)$, contradicting that $G$ has no comparable vertices. Similarly, if $Y_{i+2}=\emptyset$, then $N(i-1)\subseteq N(y_i^0)$, again a contradiction. Therefore, whenever $Y_i$ is nonempty, both $Y_{i-2}$ and $Y_{i+2}$ are nonempty. Since the indices are taken modulo $5$, this forces either $Y_i=\emptyset$ for every $i\in[5]$, or $Y_i={y_i^0}$ for every $i\in[5]$. In the first case $G\cong C_5$ and is $3$-colorable, and in the second case $G\cong S_2$.

Next suppose there exists some $i\in[5]$ such that $y_i^+\in Y_i$. Then $y_{i+2}^-\in Y_{i+2}$. Suppose $Y_{i+2}=\{y_{i+2}^-\}$. Then $N(i-1)\subseteq N(y_i^+)$, contradicting that $G$ has no comparable vertices. Thus, $Y_{i+2}=\{y_{i+2}^+,y_{i+2}^-\}$. Then $y_{i-1}^-\in Y_{i-1}$. This propagation continues. If $Y_{i-1}=\{y_{i-1}^-\}$, then $N(i+1)\subseteq N(y_{i+2}^+)$, a contradiction. Thus, $Y_{i-1}=\{y_{i-1}^+,y_{i-1}^-\}$, and hence $y_{i+1}^-\in Y_{i+1}$. Similarly, if $Y_{i+1}=\{y_{i+1}^-\}$, then $N(i-2)\subseteq N(y_{i-1}^+)$, a contradiction. Thus, $Y_{i+1}=\{y_{i+1}^+,y_{i+1}^-\}$, and hence $y_{i-2}^-\in Y_{i-2}$. Finally, suppose $Y_{i-2}=\{y_{i-2}^-\}$. Then $N(i)\subseteq N(y_{i+1}^+)$, a contradiction. Thus, $Y_{i-2}=\{y_{i-2}^+,y_{i-2}^-\}$, and hence $y_i^-\in Y_i$. Therefore $Y_j=\{y_j^+,y_j^-\}$ for every $j\in[5]$, and so $G\cong S_3$.

Therefore $G\in\{C_5,S_2,S_3\}$.
\end{proofc}

If $G\cong C_5$ then $G$ is recolorable by \cref{theo:3color}. Otherwise, by Lemma \ref{lemma:S2-recolorable} and Lemma \ref{lemma:S3-recolorable}, both $S_2$ and $S_3$ are recolorable, as desired.

Finally, suppose ${a,b}={2,3}$, that is, $G$ is $(2K_2,K_4,H_2,H_3)$-free. Since $G$ is $(H_2,H_3)$-free, we have $R_i\cup Y_i=\emptyset$ for all $i\in[5]$. Moreover, by (\ref{item:2.4}), $Z=\emptyset$. If $\bigcup_{i=1}^5 F_i=\emptyset$, then since $G$ has no comparable vertices, $|U|\leq 1$. Hence $G$ is either isomorphic to $C_5$, which is recolorable by \cref{theo:3color}, or isomorphic to $W_5$. In the latter case, $W_5$ is recolorable by first recoloring the universal vertex, if necessary, and then applying \cref{theo:3color} to the induced $5$-cycle. Thus, without loss of generality, we may assume $F_1\neq\emptyset$. By (\ref{item:1.3}), $F_3\cup F_4=\emptyset$, and at least one of $F_2$ and $F_5$ is empty. By symmetry, we may assume $F_5=\emptyset$, and hence \[V(G)=V(C)\cup U\cup F_1\cup F_2.\]

First note that $F_1\cup F_2\cup U$ is independent by (\ref{item:1.1}) and (\ref{item:1.2}), and (\ref{item:1.3}). Since $G$ has no comparable vertices, if $U\neq\emptyset$, then $F_1\cup F_2=\emptyset$. Moreover, $|U|=1$, since two vertices in $U$ would have the same open neighborhood. Hence $G\cong W_5$, which is recolorable as described above: first recolor the universal vertex, if necessary, and then apply \cref{theo:3color} to the induced $5$-cycle. Thus, we may assume $U=\emptyset$. Similarly, $|F_1|\leq 1$ and $|F_2|\leq 1$. If $F_1=F_2=\emptyset$, then $G\cong C_5$, which is recolorable by \cref{theo:3color}. If $F_i=\{f\}$ and $F_j=\emptyset$, where $\{i,j\}=\{1,2\}$, then $N(i)\subseteq N(f)$, contradicting that $G$ has no comparable vertices. Thus, $F_1$ and $F_2$ are both nonempty. Since $V(G) = V(C)\cup \{f_1,f_2\}$ where $f_i\in F_i$, it is easy to verify that $G$ is $3$-degenerate and hence recolorable by \cref{thm:degeneracy}.

This completes all cases, and hence completes the proof.
\end{proof}

\section{Proof of \cref{theorem:main3}}\label{sec:THIRD-THEOREM}

In this section, we show that  $(2K_2,K_4,C_5)$-free graphs are recolorable. In order to prove Theorem \ref{theorem:main3}, it is enough to consider graphs with no comparable vertices due to Lemma \ref{lem: comparable}. Let $G$ be a $(2K_2, K_4, C_5)$-free graphs with no comparable vertices. If $G$ is also $\overline{C}_7$-free, then $G$ is 3-colorable by The Strong Perfect Theorem \cite{chudnovsky2006strong}. Thus, by \cref{theo:3color}, $G$ is recolorable. Hence to prove \cref{theorem:main3}, it suffices to assume that $G$ with no comparable vertices that contains an induced~$\overline{C}_7$. Before proving \cref{theorem:main3}, we provide several structural properties of $(2K_2, K_4, C_5)$-free graphs that contain a $\overline{C}_7$.

\subsection{Structural properties for \cref{theorem:main3}} \label{subsection:C7-comp-structure}

    Let $G$ be a $(2K_2,K_4,C_5)$-free graph with no pair of comparable vertices. Suppose that $G$ has an induced $\overline{C}_7$ and let $C = \overline{C}_7$ with $V(C) = 1234567$. We define the following sets of vertices.
\begin{align*}
Y_i &= \{v\in V(G)\setminus V(C): N(v)\cap V(C) = \{i-1, i, i+1 \}\}\\
    F_i &= \{v\in V(G)\setminus V(C):  N(v)\cap V(C) = \{i-1, i, i+1,i+2 \}\}.\\
    \end{align*}

\begin{lemma}\label{claim:Thm3-all-of-G}
The following holds. \[V(G) = V(C)\cup \bigcup_{i=1}^7(Y_i \cup F_i).\]
\end{lemma}

\begin{proof}
Let $v\in V(G)\setminus V(C)$. First suppose $N(v)\cap V(C)\neq \emptyset$. If $i\in N(v)$, then $\{i-1, i+1\}\cap N(v)\neq \emptyset$ since otherwise $\{i-1,i, i+1, v\}$ induces a $2K_2$. Thus, $|N(v)\cap V(C)|\geq 2$. Next we show that $|N(v)\cap V(C)|\leq 4$. Suppose for contradiction $\{u_1, \dots, u_5\}\subseteq N(v)\cap V(C)$ written according to their order around $C$. Then $u_1u_3,u_1u_4, u_2u_4, u_2u_5, u_3u_5\in E(G)$ and since $G$ is $C_5$-free there exists some $i\in [5]$ such that $u_iu_{i+1}\in E(G)$. Thus, $\{u_i, u_{i+1}, u_{i+3}, v\}$ forms a $K_4$.

Letting\[R_i = \{v\in V(G)\setminus V(C): N(v)\cap V(C) = \{i,i+1\} \}\quad \text{and} \quad Z = \{v\in V(G)\setminus V(C) : N(v) \cap V(C) = \emptyset\},\] the above paragraph implies that \[V(G) = V(C) \cup \bigcup_{i=1}^7 (R_i \cup Y_i\cup F_i) \cup Z.\] It suffices to show that $R_i = \emptyset$ for all $i\in [7]$ and $Z=\emptyset$.

Suppose for contradiction $r_i\in R_i$. We claim that $r_i \in R_i$ is comparable with $i+4$. Indeed if not there must exist $v\in N(r_i)\setminus N(i+4)$. Note $v\notin V(C)$. If $|(N(v)\cap V(C))\cup (N(r_i)\cap V(C))|\leq 4$, then there exists $i,j\in [7]$ such that $\{i,j, v,r_i\}$ induces a $2K_2$. Thus, $v\notin Y_{i-1}\cup Y_i\cup Y_{i+1} \cup Y_{i+2} \cup F_{i-1}\cup F_{i}\cup F_{i+1}$. Thus,
\[N(r_i)\subseteq Y_{i+3}\cup Y_{i+4}\cup Y_{i+5} \cup F_{i+2} \cup F_{i+3} \cup F_{i+4} \cup F_{i+5} \cup \{i\}\cup \{i+1\}\subseteq N(i+4).\]
which is a contradiction.

Now suppose for contradiction $z\in Z$. Since $G$ is connected, we can choose such a $z$ with neighbor $v$ such that $N(v)\cap N(C) \neq \emptyset$. By our previous argument, $|N(v)\cap V(C)|\leq 4$. But then there exists $i,j\in [7]$ such that $\{i,j, v, z\}$ induces a $2K_2$, which is a contradiction.
\end{proof}

\begin{lemma}\label{claim:theorem3-basic-properties}
The following hold, with indices taken modulo $7$.
\begin{enumerate}[label={(5.\arabic*)}, ref={5.\arabic*}]
    \item\label{item:5.1} $Y_i$ and $F_i$ are independent.
    
    \item\label{item:5.2} $Y_i$ is complete to $Y_{i+2} \cup Y_{i-2}$.
    
    \item\label{item:5.3} $Y_i$ is anticomplete to $Y_{i+1}$ and $Y_{i-1}$. Moreover, at least one of $Y_i$ and $Y_{i+3}$ is empty, and at least one of $Y_i$ and $Y_{i-3}$ is empty.
    
    \item\label{item:5.4} $Y_i$ is anticomplete to $F_i$ and $F_{i-1}$.
    
    \item\label{item:5.5} $F_i$ is anticomplete to $F_{i+1}$ and $F_{i-1}$.
    
    \item\label{item:5.6} $F_i$ is complete to $Y_{i-3}$, $F_{i+3}$, and $F_{i-3}$.
\end{enumerate}
\end{lemma}

\begin{proof}
First consider (\ref{item:5.1}). Suppose there is an edge $e$ in $Y_1$. Then $e$ and $64$ form a $2K_2$, a contradiction. Similarly, if there is an edge $e \in F_1$, then $e$ and $64$ form a $2K_2$. 

Next consider (\ref{item:5.2}). Suppose that there exist two non-adjacent vertices $y \in Y_1$ and $y' \in Y_3$. Then $\{y, 2, y', 3, 1\}$ induces a $C_5$, a contradiction.    

Next consider (\ref{item:5.3}).   Suppose there are two adjacent vertices $y \in Y_1$ and $y' \in Y_2$. Then $\{y, y', 6, 4\}$ forms a $2K_2$, a contradiction. Next, suppose $y \in Y_1$ and $y' \in Y_4$ are adjacent; then $\{y, 2, 6, 3, y'\}$ induces a $C_5$, a contradiction. Finally, suppose $y \in Y_1$ and $y' \in Y_5$ are adjacent; then $\{y, 7, 3, 6, y'\}$ induces a $C_5$, a contradiction. For the second part, suppose $y_1 \in Y_1$ and $y_4 \in Y_4$. If $y_1 y_4 \notin E(G)$, then $\{y_1 2, y_4 3\}$ forms a $2K_2$. If $y_1 y_4 \in E(G)$, then $\{y_1, y_4, 3, 6, 2\}$ induces a $C_5$. Either way we obtain a contradiction.

Next consider (\ref{item:5.4}).  Suppose $y \in Y_1$ and $f \in F_7$ are adjacent. Then $\{y, f, 7, 2\}$ induces a $K_4$, a contradiction. Similarly, if $y \in Y_1$ and $f \in F_1$ are adjacent, then $\{y, f, 7, 2\}$ induces a $K_4$. 

Next consider (\ref{item:5.5}). Suppose $f \in F_1$ and $f' \in F_2$ are adjacent. Then $\{f, f', 1, 3\}$ induces a $K_4$, a contradiction. Similarly, if $f \in F_1$ and $f' \in F_7$ are adjacent, then $\{f, f', 7, 2\}$ induces a $K_4$.

Next consider (\ref{item:5.6}). Suppose for contradiction there exits $f\in F_i$ and $y\in Y_{i-3}$ with $fy\notin E(G)$. Then $\{f,i+2, y, i+3\}$ induce a $2K_2$, a contradiction. Similarly, if $f'\in F_{i+3}$ and $ff'\notin E(G)$, then $\{f, i-1,f', i-2\}$ induces a $2K_2$, a contradiction.\end{proof}

We now establish three additional structural claims for $(2K_2,K_4,C_5)$-free graphs containing an induced $\overline{C}_7$. These claims will provide the key justifications for several recoloring steps in the proof of \cref{theorem:main3}. To keep the recoloring sequences readable, we will indicate applications of these claims by writing $\overset{*}{\to}$, $\overset{**}{\to}$, and $\overset{***}{\to}$, respectively.

\begin{lemma}[$\overset{*}{\to}$]\label{claim:thm3-star}
Let $y_i\in Y_i$. If $y_i$ has a neighbor in one of $F_{i-2}$ and $F_{i+1}$, then $y_i$ is anticomplete to the other.
\end{lemma}

\begin{proof}
Let $y_i\in Y_i$. Suppose, for a contradiction, that $y_i$ has a neighbor $f_{i-2}\in F_{i-2}$ and a neighbor $f_{i+1}\in F_{i+1}$. By (\ref{item:5.6}), $F_{i-2}$ is complete to $F_{i+1}$, so $f_{i-2}f_{i+1}\in E(G)$. Also, both $f_{i-2}$ and $f_{i+1}$ are adjacent to $i$, and $y_i$ is adjacent to $i$. Hence $\{i,y_i,f_{i-2},f_{i+1}\}$ forms a $K_4$, a contradiction.
\end{proof}

\begin{lemma}[$\overset{**}{\to}$]\label{claim:thm3-2star}
Let $f_i\in F_i$, $f_{i+2}\in F_{i+2}$, and $y_{i+4}\in Y_{i+4}$ such that $f_if_{i+2}\in E(G)$ and $f_{i+2}y_{i+4}\in E(G)$. Then $f_i$ is anticomplete to $Y_{i-1}$.
\end{lemma}

\begin{proof}
Suppose not. Let $y_{i-1}\in Y_{i-1}$ such that $f_iy_{i-1}\in E(G)$. By (\ref{item:5.6}), $F_i$ is complete to $Y_{i-3}=Y_{i+4}$, so $f_iy_{i+4}\in E(G)$. Also by (\ref{item:5.6}), $F_{i+2}$ is complete to $Y_{i-1}$, so $f_{i+2}y_{i-1}\in E(G)$. Finally, by (\ref{item:5.2}), $Y_{i+4}$ is complete to $Y_{i-1}$, since $(i+4)+2\equiv i-1 \pmod 7$. Therefore the four vertices $\{f_i,f_{i+2},y_{i+4},y_{i-1}\}$ form a $K_4$, contradicting that $G$ is $K_4$-free. Hence $f_i$ is anticomplete to $Y_{i-1}$.
\end{proof}

\begin{lemma}[$\overset{***}{\to}$]\label{claim:thm3-3star}
Let $y_i\in Y_i$, $f_{i+1}\in F_{i+1}$, $y_{i-2}\in Y_{i-2}$, and $f_{i+3}\in F_{i+3}$ such that $y_if_{i+1}\in E(G)$ and $y_{i-2}f_{i+3}\in E(G)$. Then $f_{i+1}$ is nonadjacent to $f_{i+3}$.
\end{lemma}

\begin{proof}
Suppose not. Then $f_{i+1}f_{i+3}\in E(G)$. By (\ref{item:5.2}), $Y_i$ is complete to $Y_{i-2}$, so $y_iy_{i-2}\in E(G)$. Also, by (\ref{item:5.6}), $F_{i+3}$ is complete to $Y_i$, so $y_if_{i+3}\in E(G)$, and $F_{i+1}$ is complete to $Y_{i-2}$, so $y_{i-2}f_{i+1}\in E(G)$. Therefore the four vertices $\{y_i,y_{i-2},f_{i+1},f_{i+3}\}$ form a $K_4$, contradicting that $G$ is $K_4$-free. Hence $f_{i+1}$ is nonadjacent to $f_{i+3}$.
\end{proof}

\subsection{Recoloring proof of \cref{theorem:main3}}\label{subsec:PROOF-OF-THIRD-THEOREM}
We now present the recoloring proof for \cref{theorem:main3} which we restate for convenience.

\begin{theorem*}[\ref{theorem:main3}]
    If $G$ is a $(2K_2,K_4,C_5)$-free graph with no comparable vertices that contains an induced $\overline{C}_7$, then $G$ is recolorable.
\end{theorem*}

\begin{proof}
Let $G$ be a $(2K_2,K_4,C_5)$-free with no comparable vertices and contains an induced subgraph $C\cong \overline{C}_7$.

To show that $G$ is recolorable, we use the following strategy. We consider a specific partition of $V(G)$ into four independent sets $I_1,I_2,I_3,I_4$, which we define momentarily. We say that a $4$-coloring of $G$ is \textit{canonical} if its color classes are precisely $I_1,I_2,I_3$, and $I_4$; in particular, any permutation of the color classes of a canonical coloring is again canonical. To prove that $\mathcal{R}_{\ell}(G)$ is connected, it suffices to show that every $\ell$-coloring of $G$ can be recolored to a canonical coloring. Indeed, since $\ell>4$, any two canonical colorings are connected in $\mathcal{R}_{\ell}(G)$.

Let $\varphi$ be an $\ell$-coloring of $G$. We adapt the notation developed in \cref{subsection:C7-comp-structure} namely where $C = 1234567$ and by Lemma \ref{claim:Thm3-all-of-G} \[V(G) = V(C)\cup \bigcup _{i=1}^7 Y_i \cup \bigcup_{i=1}^7F_i.\]

Throughout the recoloring sequences below, if one of the sets appearing in a recoloring step is empty, then that step is simply omitted. If all of the sets $Y_i$ are empty, we may follow either Case 1 or Case 2 of the two recoloring templates below, deleting every step involving a $Y$-set. Thus, from now on we may assume that some $Y_i$ is nonempty, and by
cyclic symmetry we assume $Y_1\neq\emptyset$.

By (\ref{item:5.3}) we have $Y_4\cup Y_5=\emptyset$. We will consider two cases depending on if $Y_6=\emptyset$ or $Y_6\neq \emptyset$. In both cases we get a different set of $4$ independent sets which will define what a canonical coloring is. Note if $Y_6\neq \emptyset$, then $Y_2\cup Y_3=\emptyset$ by (\ref{item:5.3}).

\begin{center}
\begin{minipage}[t]{0.45\linewidth}
\centering
\textbf{Case 1: $Y_6=\emptyset$}

\vspace{-.5cm}

\begin{align*}
    I_1 &= \{1\}\cup\{2\}\cup F_4\cup F_5,\\
    I_2 &= \{3\}\cup \{4\}\cup F_6\cup F_7\cup Y_7,\\
    I_3 &= \{5\} \cup \{6\}\cup F_1\cup Y_1\cup Y_2,\\
    I_4 &= \{7\}\cup F_2\cup F_3\cup Y_3.
\end{align*}

\end{minipage}
\hfill
\begin{minipage}[t]{0.45\linewidth}
\centering
\textbf{Case 2: $Y_6\neq\emptyset$}

\vspace{-.5cm}

\begin{align*}
    I_1 &= \{1\}\cup \{2\}\cup F_4\cup F_5,\\
    I_2 &= \{3\}\cup \{4\}\cup F_6\cup Y_6\cup Y_7,\\
    I_3 &= \{5\}\cup F_1\cup F_7\cup Y_1,\\
    I_4 &= \{6\} \cup \{7\}\cup F_2\cup F_3.
\end{align*}

\end{minipage}
\end{center}

In Case 1, if $\varphi(1)=\varphi(2)$, then we recolor $I_1$ with $\varphi(1)$ and apply Lemma \ref{claim:CoverI}. In Case 2, if $\varphi(6)=\varphi(7)$, then we recolor $I_4$ with $\varphi(6)$ and apply Lemma \ref{claim:CoverI}. Thus, in Case 1 we may assume $\varphi(1)\neq \varphi(2)$, and in Case 2 we may assume $\varphi(6)\neq \varphi(7)$. Since $146$, $246$, and $247$ form triangles in Case 1 and Case 2, we may assume, respectively, that
\[
\varphi(1,2,4,6)=(\alpha,\beta,\gamma,\delta)
\qquad\text{and}\qquad
\varphi(6,7,2,4)=(\alpha,\beta,\gamma,\delta).
\]

Thus, in Case~1 the possible values of $\varphi(3,5,7)$ are the same as the
possible values of $\varphi(1,3,5)$ in Case~2. Let $\tau$ denote $\varphi(3,5,7)$ in
Case~1 and $\varphi(1,3,5)$ in Case~2. We list below the fourteen possible
values of $\tau$.
\[
\renewcommand{\arraystretch}{1.3}
\begin{array}{lll}
(\beta,\gamma,\alpha) & (\beta,\varepsilon,\alpha) & (\gamma,\varepsilon,\delta) \\
(\beta,\gamma,\delta) & (\beta,\varepsilon,\delta) & (\varepsilon,\gamma,\alpha) \\
(\beta,\delta,\alpha) & (\beta,\delta,\varepsilon) & (\varepsilon,\gamma,\delta) \\
(\gamma,\delta,\alpha) & (\gamma,\delta,\varepsilon) & (\varepsilon,\delta,\alpha) \\
(\beta,\gamma,\varepsilon) & (\gamma,\varepsilon,\alpha) \\
\end{array}
\]

In the tables below we list, for each case and each value of $\tau$, a
legal recoloring sequence together with the resulting conclusion. We label
these subcases $x.y$, where $x \in \{1,2\}$ denotes the case and
$y \in \{1,\dots,14\}$ denotes the $y$-th value of $\tau$ in the list above.
For instance, Subcase~1.1 is Case~1 with $\tau = (\beta, \gamma, \alpha)$. The subcase column is hyperlinked to the rigorous justification for why such recoloring sequence is legal. Thus, these to tables together with the justifications, completes the proof.
\end{proof}


\begin{center}
\textbf{Recoloring sequences for Case 1: $Y_6=\emptyset$.}

\vspace{0.5ex}
\scriptsize
\setlength{\tabcolsep}{3pt}
\renewcommand{\arraystretch}{1.35}
\begin{tabularx}{\linewidth}{
    >{\centering\arraybackslash}p{0.05\linewidth}
    >{\centering\arraybackslash}p{0.11\linewidth}
    >{\arraybackslash}X
    >{\centering\arraybackslash}p{0.11\linewidth}
}
\hline
\textbf{Subcase} & \textbf{$\varphi(3,5,7)$} & \textbf{Legal recoloring sequence} & \textbf{Conclusion} \\
\hline

\midrule

\hyperref[subcase:1.1]{1.1} & $(\beta,\gamma,\alpha)$
&$\begin{aligned}
&F_3\to \alpha,\quad Y_1\to \gamma,\quad Y''_7\to \beta,\quad F_1\overset{*}{\to} \gamma,\quad Y_2\to \delta,\quad F_2\to \delta,\quad Y_3\to \delta,\quad F_6\to \beta,\\
&F_7\to \gamma,\quad Y'_7\overset{*}{\to} \gamma,\quad F_4\to \varepsilon,\quad F_5\to \varepsilon,\quad 1\to \varepsilon,\quad 2\to \varepsilon
\end{aligned}$
&$  I_1 = \varepsilon$
\\
\midrule

\hyperref[subcase:1.2]{1.2} & $(\beta,\gamma,\delta)$
&
$\begin{aligned}
&F_6\to \beta,\quad F_2\to \delta,\quad Y_3\to \delta,\quad F_4\to \alpha,\quad Y''_2\to \delta,\quad F''_3\to \delta,\quad F'_5\overset{***}{\to} \alpha,\quad F''_5\to \beta,\\ &F'_3\overset{***}{\to} \alpha,\quad Y_7\to \beta,\quad Y'_2\overset{*}{\to} \varepsilon,\quad F'_3\to \delta
\end{aligned}$
&
$  I_4 = \delta$
\\
\midrule

\hyperref[subcase:1.3]{1.3} & $(\beta,\delta,\alpha)$
&
$\begin{aligned}
&F_3\to \alpha,\quad F_6\to \beta,\quad Y_1\to \gamma,\quad Y_2\to \delta,\quad F_2\to \delta,\quad F_7\to \gamma,\quad Y'_7\overset{*}{\to} \gamma,\quad
Y''_7\to \beta,\\ &F_1\overset{*}{\to} \gamma,\quad F_5\to \beta,\quad Y_3\to \delta,\quad F_4\to \varepsilon,\quad F_5\to \varepsilon,\quad 1\to \varepsilon,\quad 2\to \varepsilon
\end{aligned}$
&
$  I_4 = \varepsilon$
\\
\midrule

\hyperref[subcase:1.4]{1.4} & $(\gamma,\delta,\alpha)$
&
$\begin{aligned}
&Y_2\to \delta,\quad F_3\to \alpha,\quad Y_7\to \gamma,\quad F_7\to \gamma,\quad Y''_3\to \alpha,\quad F''_6\to \gamma,\quad F'_4\overset{***}{\to} \beta,\quad Y'_3\to \alpha,\\ &Y'_1\overset{*}{\to} \delta,\quad F'_6\to \gamma
\end{aligned}$
&
$  I_2 = \gamma$
\\
\midrule

\hyperref[subcase:1.5]{1.5} & $(\beta,\gamma,\varepsilon)$ & $\begin{aligned} &Y_1\to \gamma,\quad F_5'' \to \beta, \quad F'_3 \overset{***}{\to} \alpha, \quad F_3''\to \varepsilon, \quad Y_2\to \delta, \quad Y_3' \overset{*}{\to} \delta, \quad Y''_3\to \varepsilon,\quad F_4 \to \alpha, \\  &Y'_3 \to \varepsilon, \quad F_3' \to \varepsilon, \quad F_2 \to \varepsilon \end{aligned}$ & $  I_3 = \delta$ \\ \midrule

\hyperref[subcase:1.6]{1.6} & $(\beta,\varepsilon,\alpha)$
&
$\begin{aligned}
&F_3\to \alpha,\quad F_6\to \beta,\quad Y_1\to \gamma,\quad Y_2\to \delta,\quad F_2\to \delta,\quad F_7\to \gamma, \quad Y'_7\overset{*}{\to} \gamma,\quad Y''_7\to \beta,\\
&Y'_3\overset{*}{\to} \alpha,\quad Y''_3\to \delta,\quad F_1\to \varepsilon,\quad Y_1\to \varepsilon,\quad Y_2\to \varepsilon,\quad 6\to \varepsilon
\end{aligned}$
&
$  I_3 = \varepsilon$
\\
\midrule

\hyperref[subcase:1.7]{1.7} & $(\beta,\varepsilon,\delta)$
&
$\begin{aligned}
&F_6\to \beta,\quad F_2\to \delta,\quad Y_3\to \delta,\quad F_4\to \alpha,\quad Y''_2\to \delta,\quad F''_3\to \delta, \quad F_5''\to \beta, \quad F'_5\overset{***}{\to} \alpha,\\ &F'_3\overset{***}{\to} \alpha,\quad Y_7\to \beta,\quad Y'_2\overset{*}{\to} \varepsilon,\quad F'_3\to \delta
\end{aligned}$
&
$  I_4 = \delta$
\\
\midrule

\hyperref[subcase:1.8]{1.8} & $(\beta,\delta,\varepsilon)$
&
$\begin{aligned}
&F_6\to \beta,\quad Y_1\to \gamma,\quad Y''_3\to \varepsilon,\quad Y'_3\overset{*}{\to} \delta,\quad F_4\to \alpha,\quad Y'_3\to \varepsilon,\quad Y_2\to \delta,\quad F_3\to \varepsilon,\\
&F_2\to \varepsilon
\end{aligned}$
&
$  I_4 = \varepsilon$
\\
\midrule

\hyperref[subcase:1.9]{1.9} & $(\gamma,\delta,\varepsilon)$
&
$\begin{aligned}
&Y_2\to \delta,\quad F_3\to \varepsilon,\quad Y_7\to \gamma,\quad F_5\to \beta,\quad
Y''_3\to \varepsilon,\quad Y_1''\to \gamma, \quad F_4''\to \alpha, \quad F_6'' \to \gamma, \\ &F_4' \overset{***}{\to}\beta, \quad F_6' \overset{***}{\to} \beta,\quad Y_1'\to \gamma, \quad Y'_3\overset{*}{\to} \delta, \quad Y'_3\to \varepsilon,\quad F_2\to \varepsilon
\end{aligned}$
&
$  I_4 = \varepsilon$
\\
\midrule

\hyperref[subcase:1.10]{1.10} & $(\gamma,\varepsilon,\alpha)$
&
$\begin{aligned}
&Y'_3\overset{*}{\to} \alpha,\quad
F_6'' \to \gamma, \quad 
Y_7 \to \gamma, \quad
F_4' \overset{***}{\to} \beta, \quad
Y''_3\to \alpha,\quad
Y'_1\overset{*}{\to} \varepsilon,\quad
F_6' \to \gamma, \quad 
F_7\to \gamma,\quad
\end{aligned}$
&
$  I_2 = \gamma$
\\
\midrule

\hyperref[subcase:1.11]{1.11} & $(\gamma,\varepsilon,\delta)$
&
$\begin{aligned}
&F_2\to \delta,\quad
Y_7\to \gamma,\quad
Y_3\to \delta,\quad
F_7\to \gamma,\quad
Y_1\to \varepsilon,\quad
Y_2\to \varepsilon,\quad
6\to \varepsilon
\end{aligned}$
&
$  I_3 = \varepsilon$
\\
\midrule

\hyperref[subcase:1.12]{1.12} & $(\varepsilon,\gamma,\alpha)$
&
$\begin{aligned}
&F_3\to \alpha,\quad Y_2\to \delta,\quad Y_1\to \gamma,\quad
Y'_7\overset{*}{\to} \gamma,\quad Y''_7\to \varepsilon, \quad F_6\to \varepsilon,\quad F_5\to \beta,\quad F_4\to \beta,\\ &1\to \beta
\end{aligned}$
&
$  I_1 = \beta$
\\
\midrule

\hyperref[subcase:1.13]{1.13} & $(\varepsilon,\gamma,\delta)$
&
$\begin{aligned}
&F_2\to \delta,\quad Y_1\to \gamma,\quad Y_3\to \delta,\quad F''_3\to \delta,\quad F'_5\overset{***}{\to} \alpha,\quad F''_5\to \beta, \quad
F_4\to \alpha,\quad Y_7\to \beta,\\ 
&F_7\to \varepsilon,\quad Y'_2 \to \gamma,\quad F'_3\to \delta
\end{aligned}$
&
$  I_4 = \delta$
\\
\midrule

\hyperref[subcase:1.14]{1.14} & $(\varepsilon,\delta,\alpha)$
&
$\begin{aligned}
&Y_2 \to \delta, \quad F_3 \to \alpha, \quad Y'_7\overset{*}{\to} \gamma,\quad F_5\to \beta,\quad Y_7\to \varepsilon,\quad F_7\to \varepsilon,\quad F''_6\to \varepsilon,\quad
F'_4\overset{***}{\to} \beta,\\ 
&Y_3\to \alpha,\quad Y'_1\overset{*}{\to} \delta,\quad F'_6\to \varepsilon,\quad 4\to \varepsilon
\end{aligned}$
&
$  I_2 = \varepsilon$
\\
\midrule

\hline
\end{tabularx}
\end{center}

\newpage


\begin{center}
\textbf{Recoloring sequences for Case 2: $Y_6\neq\emptyset$.}

\vspace{0.5ex}
\scriptsize
\setlength{\tabcolsep}{3pt}
\renewcommand{\arraystretch}{1.35}
\begin{tabularx}{\linewidth}{
    >{\centering\arraybackslash}p{0.05\linewidth}
    >{\centering\arraybackslash}p{0.11\linewidth}
    >{\arraybackslash}X
    >{\centering\arraybackslash}p{0.11\linewidth}
}
\hline
\textbf{Subcase} & \textbf{$\varphi(1,3,5)$} & \textbf{Legal recoloring sequence} & \textbf{Conclusion} \\
\hline

\midrule

\hyperref[subcase:2.1]{2.1} & $(\beta,\gamma,\alpha)$
&
$
\begin{aligned}
&F_4\to \beta,\quad F_1\to \alpha,\quad F_3\to \beta,\quad Y_7\to \delta,\quad F_5\to \gamma,\quad Y_6\to \gamma,\quad F_6\to \gamma,\quad Y_1\to \delta,\\ &F_7\to \delta,\quad
F_2\to \varepsilon,\quad F_3\to \varepsilon,\quad 6\to \varepsilon,\quad 7\to \varepsilon
\end{aligned}
$
&
$  I_4=\varepsilon$
\\
\midrule

\hyperref[subcase:2.2]{2.2} & $(\beta,\gamma,\delta)$
&
$
\begin{aligned}
&Y_1\to \delta,\quad F_3\to \beta,\quad F_2\to \alpha,\quad F_7\to \delta,\quad F_1\to \alpha,\quad Y_7\to \delta,\quad F_5\to \gamma,\quad F_4\to \beta,\\ &F_6 \to \varepsilon, \quad Y_6\to \varepsilon,\quad Y_7\to \varepsilon,\quad 3\to \varepsilon,\quad 4\to \varepsilon
\end{aligned}
$
&
$  I_2=\varepsilon$
\\
\midrule

\hyperref[subcase:2.3]{2.3} & $(\beta,\delta,\alpha)$
&
$
\begin{aligned}
&Y_7\to \delta,\quad F_5\to \gamma,\quad F_4\to \beta,\quad F_6\to \gamma,\quad Y_6\to \gamma,\quad F_7\to \delta,\quad Y_1\to \delta,\quad F_1\to \alpha,\\
&F_2\to \varepsilon,\quad F_3\to \varepsilon,\quad 6\to \varepsilon,\quad 7\to \varepsilon
\end{aligned}
$
&
$  I_4=\varepsilon$
\\
\midrule

\hyperref[subcase:2.4]{2.4} & $(\gamma,\delta,\alpha)$
&
$
\begin{aligned}
&F_3\to \beta,\quad F_5\to \gamma,\quad F_1\to \alpha,\quad Y_7\to \delta,\quad Y'_1\overset{*}{\to} \alpha,\quad Y''_6\to \gamma,\quad F_6\to \delta,\quad F'_2\overset{**}{\to} \alpha,\\
&F'_4\to \beta,\quad Y'_6\to \gamma,\quad Y''_1\to \delta,\quad F''_2 \overset{*}{\to} \alpha,\quad F''_4\to \gamma,\quad Y_1\to \varepsilon,\quad F_1\to \varepsilon,\quad F_7\to \varepsilon,\\ &5\to \varepsilon
\end{aligned}
$
&
$  I_3=\varepsilon$
\\
\midrule

\hyperref[subcase:2.5]{2.5} & $(\beta,\gamma,\varepsilon)$
&
$
\begin{aligned}
&F_3\to \beta,\quad F_1\to \alpha,\quad F_6\to \gamma,\quad Y_6\to \gamma,\quad Y_1\to \delta,\quad Y_7\to \delta,\quad F_7\to \delta,\quad F_2\to \alpha,\\ &Y_1\to \varepsilon,\quad F_1\to \varepsilon,\quad F_7\to \varepsilon
\end{aligned}
$
&
$  I_3=\varepsilon$
\\
\midrule

\hyperref[subcase:2.6]{2.6} & $(\beta,\varepsilon,\alpha)$
&
$
\begin{aligned}
&F_1\to \alpha,\quad F_4\to \beta,\quad Y_7\to \delta,\quad F_5\to \gamma,\quad Y_7\to \varepsilon,\quad Y_6\to \gamma,\quad F_6\to \gamma,\quad F_1\to \delta,\\ &Y_1\to \delta,\quad F_7\to \delta,\quad 5\to \delta
\end{aligned}
$
&
$  I_3=\delta$
\\
\midrule

\hyperref[subcase:2.7]{2.7} & $(\beta,\varepsilon,\delta)$
&
$
\begin{aligned}
&Y_1\to \delta,\quad F_3\to \beta,\quad F_1\to \alpha,\quad Y_7\to \delta,\quad F_5\to \gamma,\quad Y_7\to \varepsilon,\quad F_1\to \delta,\quad F_7\to \delta
\end{aligned}
$
&
$  I_3=\delta$
\\
\midrule

\hyperref[subcase:2.8]{2.8} & $(\beta,\delta,\varepsilon)$
&
$
\begin{aligned}
&F_4\to \beta,\quad F_6\to \gamma,\quad Y_6\to \gamma,\quad Y_7\to \delta,\quad Y_1\to \delta,\quad F_2\to \alpha,\quad F_1\to \varepsilon,\quad F_7\to \varepsilon,\\ &Y_1\to \varepsilon
\end{aligned}
$
&
$  I_3=\varepsilon$
\\
\midrule

\hyperref[subcase:2.9]{2.9} & $(\gamma,\delta,\varepsilon)$
&
$
\begin{aligned}
&F_3\to \beta,\quad Y_7\to \delta,\quad F_1 \to \alpha, \quad F_4'' \to \gamma, \quad F_2' \overset{**}{\to} \alpha, \quad F_2'' \to \beta, \quad Y_1 \to \varepsilon, \quad Y_6' \overset{*}{\to} \delta, \\ &F_4'\to \gamma
\end{aligned}
$
&
$  I_4=\varepsilon$
\\
\midrule

\hyperref[subcase:2.10]{2.10} & $(\gamma,\varepsilon,\alpha)$
&
$
\begin{aligned}
&F_3\to \beta,\quad F_5\to \gamma,\quad F_4'' \to \gamma, \quad F_2'' \to \beta, \quad F_1\to \alpha,\quad Y_1 \overset{**}{\to} \alpha, \quad Y_6' \overset{*}{\to} \delta, \quad F_4' \to \gamma
\end{aligned}
$
&
$  I_1=\gamma$
\\
\midrule

\hyperref[subcase:2.11]{2.11} & $(\gamma,\varepsilon,\delta)$
&
$
\begin{aligned}
&F_3\to \beta,\quad F_1\to \alpha,\quad Y_1\to \delta,\quad F_7\to \delta,\quad Y_7\to \delta,\quad Y_6\to \varepsilon,\quad F_4\to \gamma,\quad F_5\to \gamma
\end{aligned}
$
&
$  I_1=\gamma$
\\
\midrule

\hyperref[subcase:2.12]{2.12} & $(\varepsilon,\gamma,\alpha)$
&
$
\begin{aligned}
&Y_6\to \gamma,\quad F_3\to \beta,\quad F_4\to \varepsilon,\quad F_5\to \varepsilon,\quad 2\to \varepsilon
\end{aligned}
$
&
$  I_1=\varepsilon$
\\
\midrule

\hyperref[subcase:2.13]{2.13} & $(\varepsilon,\gamma,\delta)$
&
$
\begin{aligned}
&Y''_7\to \gamma,\quad Y_1\to \delta,\quad F_7\to \delta,\quad F_1 \overset{*}{\to} \delta
\end{aligned}
$
&
$  I_3=\delta$
\\
\midrule

\hyperref[subcase:2.14]{2.14} & $(\varepsilon,\delta,\alpha)$
&
$
\begin{aligned}
&F_3\to \beta,\quad F_4'' \to \varepsilon, \quad F_2'' \to \beta, \quad Y_1 \overset{**}{\to} \alpha, \quad Y_6' \overset{*}{\to} \delta, \quad Y_6'' \to \gamma, \quad F_4' \to \varepsilon, \quad F_5\to \varepsilon, \\ &2\to \varepsilon
\end{aligned}
$
&
$  I_1=\varepsilon$
\\
\midrule

\hline
\end{tabularx}
\end{center}

\newcommand{\subcaseend}{\hfill{\tiny$\blacksquare$}}

\bigskip

\textbf{Justification of recoloring sequences for \cref{theorem:main3}}

Before giving the recoloring sequence justifications, we explain how each justification table should be read.
Each row of a table has the form $X\to \theta$, meaning that every vertex of $X$ (which will always be an independent set) is recolored, one at a time, with color $\theta$. To justify such a step, we sort the vertices outside $X$ according to their possible relationship with the color $\theta$:
\[
\begin{array}{lll}
\text{(i) definitely do not have color $\theta$;} &
\text{(ii) definitely do have color $\theta$;} &
\text{(iii) may have color $\theta$.}
\end{array}
\]
A set of vertices is of type (i) when it has already been assigned or recolored with a color different from $\theta$, or when it is complete to a set of vertices known to have color $\theta$. A set of vertices is of type (ii) when its color is forced to be $\theta$ by the current subcase, or when it has already been recolored with color $\theta$ earlier in the sequence. Finally, a vertex or set is of type (iii) when, from the information available at that moment, it may still contain vertices of color $\theta$.

Only vertices and sets of types (ii) and (iii) can obstruct the recoloring. Thus, the obstruction column records precisely the vertices and sets of types (ii) and (iii) at that point in the sequence. The justification column then verifies that $X$ is anticomplete to all listed obstructions. Since the set being recolored is independent in every step, this proves that no vertex of $X$ has a neighbor colored $\theta$, and hence the recoloring $X\to\theta$ is legal. Vertices and sets of type (i) are omitted from the obstruction column, since they cannot block the recoloring.

\noindent\phantomsection\label{subcase:1.1}\textbf{Subcase 1.1.} Suppose $\varphi(3,5,7)=(\beta, \gamma, \alpha)$. 

Since $\varphi(1,2,3,4,5,6,7)=(\alpha,\beta,\beta,\gamma,\gamma,\delta,\alpha)$, the colors allowed by the cycle are
$Y_1,Y_2,F_1:\{\gamma,\delta,\varepsilon\}$,
$Y_3,F_3:\{\alpha,\delta,\varepsilon\}$,
$Y_7:\{\beta,\gamma,\varepsilon\}$,
$F_2:\{\delta,\varepsilon\}$,
$F_4:\{\alpha,\varepsilon\}$,
$F_5,F_6:\{\beta,\varepsilon\}$, and
$F_7:\{\gamma,\varepsilon\}$.

Let $Y'_7=\{y\in Y_7: y \text{ has a neighbor in }F_5\}$ and let $Y''_7=Y_7\setminus Y'_7$. Then $Y''_7$ is anticomplete to $F_5$, and by Lemma \ref{claim:thm3-star}, $Y'_7$ is anticomplete to $F_1$.

The following table gives the recoloring sequence and verifies each step.
\[
\begin{array}{c|c|c}
\textup{step} & \textup{obstructions} & \textup{Justification: anticomplete by} \\
\hline
F_3\to \alpha
& 1,7,Y_3,F_4
& \textup{definition of }F_3,\ \textup{(\ref{item:5.4}), (\ref{item:5.5})}
\\
Y_1\to \gamma
& 4,5,Y_2,Y_7,F_1,F_7
& \textup{definition of }Y_1,\ \textup{(\ref{item:5.3}), (\ref{item:5.4})}
\\
Y''_7\to \beta
& 2,3,Y'_7,F_5,F_6
& \textup{definition of }Y_7,\ \textup{(\ref{item:5.1}), (\ref{item:5.4}), definition of }Y''_7
\\
F_1\overset{*}{\to} \gamma
& 4,5,Y_1,Y_2,Y'_7,F_7
& \textup{definition of }F_1,\ \textup{(\ref{item:5.4}), (\ref{item:5.5}), Lemma \ref{claim:thm3-star}}
\\
Y_2\to \delta
& 6,Y_3,F_2
& \textup{definition of }Y_2,\ \textup{(\ref{item:5.3}), (\ref{item:5.4})}
\\
F_2\to \delta
& 6,Y_2,Y_3
& \textup{definition of }F_2,\ \textup{(\ref{item:5.4})}
\\
Y_3\to \delta
& 6,Y_2,F_2
& \textup{definition of }Y_3,\ \textup{(\ref{item:5.3}), (\ref{item:5.4})}
\\
F_6\to \beta
& 2,3,Y_7,F_5
& \textup{definition of }F_6,\ \textup{(\ref{item:5.4}), (\ref{item:5.5})}
\\
F_7\to \gamma
& 4,5,Y_1,Y'_7,F_1
& \textup{definition of }F_7,\ \textup{(\ref{item:5.4}), (\ref{item:5.5})}
\\
Y'_7\overset{*}{\to} \gamma
& 4,5,Y_1,F_1,F_7
& \textup{definition of }Y_7,\ \textup{(\ref{item:5.3}), (\ref{item:5.4}), Lemma \ref{claim:thm3-star}}
\\
F_4\to \varepsilon
& F_5
& \textup{(\ref{item:5.5})}
\\
F_5\to \varepsilon
& F_4
& \textup{(\ref{item:5.5})}
\\
1\to \varepsilon
& F_4,F_5
& \textup{definition of }F_4\textup{ and }F_5
\\
2\to \varepsilon
& F_4,F_5,1
& \textup{definition of }F_4\textup{ and }F_5
\end{array}
\]

Thus every step is legal, and the sequence ends with $I_1=\varepsilon$. Hence Lemma \ref{claim:CoverI} applies. \subcaseend

\noindent\phantomsection\label{subcase:1.2}\textbf{Subcase 1.2.} Suppose $\varphi(3,5,7)=(\beta,\gamma,\delta)$. 

Since $\varphi(1,2,3,4,5,6,7)=(\alpha,\beta,\beta,\gamma,\gamma,\delta,\delta)$, the colors allowed by the cycle are
$Y_1,F_1:\{\gamma,\varepsilon\}$,
$Y_2:\{\gamma,\delta,\varepsilon\}$,
$Y_3,F_3:\{\alpha,\delta,\varepsilon\}$,
$Y_7:\{\beta,\gamma,\varepsilon\}$,
$F_2:\{\delta,\varepsilon\}$,
$F_4:\{\alpha,\varepsilon\}$,
$F_5:\{\alpha,\beta,\varepsilon\}$,
$F_6:\{\beta,\varepsilon\}$, and
$F_7:\{\gamma,\varepsilon\}$.

Let $Y'_2=\{y\in Y_2: y \text{ has a neighbor in }F_3\}$ and let $Y''_2=Y_2\setminus Y'_2$. Then $Y''_2$ is anticomplete to $F_3$, and by Lemma \ref{claim:thm3-star}, $Y'_2$ is anticomplete to $F_7$. Similarly, let $Y'_7=\{y\in Y_7: y \text{ has a neighbor in }F_5\}$ and let $Y''_7=Y_7\setminus Y'_7$. Then $Y''_7$ is anticomplete to $F_5$. Now let $F'_3=\{f\in F_3: f \text{ has a neighbor in }Y'_2\}$ and $F''_3=F_3\setminus F'_3$, and let $F'_5=\{f\in F_5: f \text{ has a neighbor in }Y'_7\}$ and $F''_5=F_5\setminus F'_5$. Then $F''_3$ is anticomplete to $Y'_2$, $F''_5$ is anticomplete to $Y'_7$, and by Lemma \ref{claim:thm3-3star}, $F'_3$ is anticomplete to $F'_5$.

The following table gives the recoloring sequence and verifies each step.
\[
\begin{array}{c|c|c}
\textup{step} & \textup{obstructions} & \textup{Justification: anticomplete by} \\
\hline
F_6\to \beta
& 2,3,Y_7,F_5
& \textup{definition of }F_6\textup{, }(\ref{item:5.4}),(\ref{item:5.5})
\\
F_2\to \delta
& 6,7,Y_2,Y_3,F_3
& \textup{definition of }F_2\textup{, }(\ref{item:5.4}),(\ref{item:5.5})
\\
Y_3\to \delta
& 6,7,F_2,Y_2,F_3
& \textup{definition of }Y_3\textup{, }(\ref{item:5.3}),(\ref{item:5.4})
\\
F_4\to \alpha
& 1,F_3,F_5
& \textup{definition of }F_4\textup{, }(\ref{item:5.5})
\\
Y''_2\to \delta
& 6,7,F_2,Y_3,Y'_2,F_3
& \textup{definition of }Y_2\textup{, }(\ref{item:5.1}),(\ref{item:5.3}),(\ref{item:5.4}),\textup{ definition of }Y''_2
\\
F''_3\to \delta
& 6,7,F_2,Y_3,Y''_2,Y'_2,F'_3
& \textup{definition of }F_3\textup{, }(\ref{item:5.1}),(\ref{item:5.4}),(\ref{item:5.5}),\textup{ definition of }Y''_2\textup{ and }F''_3
\\
F'_5\overset{***}{\to} \alpha
& 1,F_4,F'_3,F''_5
& \textup{definition of }F_5\textup{, }(\ref{item:5.1}),(\ref{item:5.5}),\textup{ Lemma \ref{claim:thm3-3star}}
\\
F''_5\to \beta
& 2,3,F_6,Y_7
& \textup{definition of }F_5\textup{, }(\ref{item:5.5}),\textup{ definitions of }Y''_7\textup{ and }F''_5
\\
F'_3\overset{***}{\to} \alpha
& 1,F_4,F'_5
& \textup{definition of }F_3\textup{, }(\ref{item:5.5}),\textup{ Lemma \ref{claim:thm3-3star}}
\\
Y_7\to \beta
& 2,3,F_6,F''_5
& \textup{definition of }Y_7\textup{, }(\ref{item:5.4}),\textup{ definitions of }Y''_7\textup{ and }F''_5
\\
Y'_2\overset{*}{\to} \varepsilon
& Y_1,F_1,F_7
& \textup{(\ref{item:5.3}),(\ref{item:5.4}), Lemma \ref{claim:thm3-star}}
\\
F'_3\to \delta
& 6,7,F_2,Y_3,Y''_2,F''_3
& \textup{definition of }F_3\textup{, }(\ref{item:5.1}),(\ref{item:5.4}),(\ref{item:5.5}),\textup{ definition of }Y''_2
\end{array}
\]
Thus every step is legal, and the sequence ends with $I_4=\delta$. Hence Lemma \ref{claim:CoverI} applies. \subcaseend

\noindent\phantomsection\label{subcase:1.3}\textbf{Subcase 1.3.} Suppose $\varphi(3,5,7)=(\beta,\delta,\alpha)$. 

Since $\varphi(1,2,3,4,5,6,7)=(\alpha,\beta,\beta,\gamma,\delta,\delta,\alpha)$, the colors allowed by the cycle are
$Y_1,Y_2,F_1:\{\gamma,\delta,\varepsilon\}$,
$Y_3:\{\alpha,\delta,\varepsilon\}$,
$Y_7:\{\beta,\gamma,\varepsilon\}$,
$F_2:\{\delta,\varepsilon\}$,
$F_3,F_4:\{\alpha,\varepsilon\}$,
$F_5:\{\beta,\varepsilon\}$,
$F_6:\{\beta,\gamma,\varepsilon\}$, and
$F_7:\{\gamma,\varepsilon\}$.

Let $Y'_7=\{y\in Y_7: y \text{ has a neighbor in }F_5\}$ and let $Y''_7=Y_7\setminus Y'_7$. Then $Y''_7$ is anticomplete to $F_5$, and by Lemma \ref{claim:thm3-star}, $Y'_7$ is anticomplete to $F_1$.

The following table gives the recoloring sequence and verifies each step.
\[
\begin{array}{c|c|c}
\textup{step} & \textup{obstructions} & \textup{Justification: anticomplete by} \\
\hline
F_3\to \alpha
& 1,7,Y_3,F_4
& \textup{definition of }F_3\textup{, }(\ref{item:5.4}),(\ref{item:5.5})
\\
F_6\to \beta
& 2,3,Y_7,F_5
& \textup{definition of }F_6\textup{, }(\ref{item:5.4}),(\ref{item:5.5})
\\
Y_1\to \gamma
& 4,Y_2,Y_7,F_1,F_7
& \textup{definition of }Y_1\textup{, }(\ref{item:5.3}),(\ref{item:5.4})
\\
Y_2\to \delta
& 5,6,F_1,F_2,Y_3
& \textup{definition of }Y_2\textup{, }(\ref{item:5.3}),(\ref{item:5.4})
\\
F_2\to \delta
& 5,6,Y_2,Y_3,F_1
& \textup{definition of }F_2\textup{, }(\ref{item:5.4}),(\ref{item:5.5})
\\
F_7\to \gamma
& 4,Y_1,F_1,Y_7
& \textup{definition of }F_7\textup{, }(\ref{item:5.4}),(\ref{item:5.5})
\\
Y'_7\overset{*}{\to} \gamma
& 4,Y_1,Y''_7,F_1,F_7
& \textup{definition of }Y_7\textup{, }(\ref{item:5.1}),(\ref{item:5.3}),(\ref{item:5.4}),\textup{ Lemma \ref{claim:thm3-star}}
\\
Y''_7\to \beta
& 2,3,F_6,F_5
& \textup{definition of }Y_7\textup{, }(\ref{item:5.4}),\textup{ definition of }Y''_7
\\
F_1\overset{*}{\to} \gamma
& 4,Y_1,F_7,Y'_7
& \textup{definition of }F_1\textup{, }(\ref{item:5.4}),(\ref{item:5.5}),\textup{ Lemma \ref{claim:thm3-star}}
\\
F_5\to \beta
& 2,3,F_6,Y''_7
& \textup{definition of }F_5\textup{, }(\ref{item:5.5}),\textup{ definition of }Y''_7
\\
Y_3\to \delta
& 5,6,Y_2,F_2
& \textup{definition of }Y_3\textup{, }(\ref{item:5.3}),(\ref{item:5.4})
\\
F_4\to \varepsilon
& \emptyset
& \textup{none}
\\
F_5\to \varepsilon
& F_4
& \textup{(\ref{item:5.5})}
\\
1\to \varepsilon
& F_4,F_5
& \textup{definition of }F_4\textup{ and }F_5
\\
2\to \varepsilon
& F_4,F_5,1
& \textup{definition of }F_4\textup{ and }F_5\textup{, and }12\notin E(G)
\end{array}
\]

Thus every step is legal, and the sequence ends with $I_1=\varepsilon$. Hence Lemma \ref{claim:CoverI} applies. \subcaseend

\noindent\phantomsection\label{subcase:1.4}\textbf{Subcase 1.4.} Suppose $\varphi(3,5,7)=(\gamma,\delta,\alpha)$. 

Since $\varphi(1,2,3,4,5,6,7)=(\alpha,\beta,\gamma,\gamma,\delta,\delta,\alpha)$, the colors allowed by the cycle are
$Y_1:\{\gamma,\delta,\varepsilon\}$,
$Y_2,F_1,F_2:\{\delta,\varepsilon\}$,
$Y_3:\{\alpha,\delta,\varepsilon\}$,
$Y_7:\{\beta,\gamma,\varepsilon\}$,
$F_3:\{\alpha,\varepsilon\}$,
$F_4:\{\alpha,\beta,\varepsilon\}$,
$F_5:\{\beta,\varepsilon\}$,
$F_6:\{\beta,\gamma,\varepsilon\}$, and
$F_7:\{\gamma,\varepsilon\}$.

Let $Y'_1=\{y\in Y_1: y \text{ has a neighbor in }F_6\}$ and let $Y''_1=Y_1\setminus Y'_1$. Then $Y''_1$ is anticomplete to $F_6$. Similarly, let $Y'_3=\{y\in Y_3: y \text{ has a neighbor in }F_4\}$ and let $Y''_3=Y_3\setminus Y'_3$. Then $Y''_3$ is anticomplete to $F_4$. Let $F'_6=\{f\in F_6: f \text{ has a neighbor in }Y'_1\}$ and let $F''_6=F_6\setminus F'_6$. Then $F''_6$ is anticomplete to $Y'_1$. Finally, let $F'_4=\{f\in F_4: f \text{ has a neighbor in }Y'_3\}$ and let $F''_4=F_4\setminus F'_4$. Then $F''_4$ is anticomplete to $Y'_3$. Also, by Lemma \ref{claim:thm3-3star}, $F'_4$ is anticomplete to $F'_6$.

The following table gives the recoloring sequence and verifies each step.
\[
\begin{array}{c|c|c}
\textup{step} & \textup{obstructions} & \textup{Justification: anticomplete by} \\
\hline
Y_2\to \delta
& 5,6,Y_1,F_1,F_2,Y_3
& \textup{definition of }Y_2\textup{, }(\ref{item:5.3}),(\ref{item:5.4})
\\
F_3\to \alpha
& 1,7,Y_3,F_4
& \textup{definition of }F_3\textup{, }(\ref{item:5.4}),(\ref{item:5.5})
\\
Y_7\to \gamma
& 3,4,Y_1,F_6,F_7
& \textup{definition of }Y_7\textup{, }(\ref{item:5.3}),(\ref{item:5.4})
\\
F_7\to \gamma
& 3,4,Y_1,Y_7,F_6
& \textup{definition of }F_7\textup{, }(\ref{item:5.4}),(\ref{item:5.5})
\\
Y''_3\to \alpha
& 1,7,F_3,Y'_3,F_4
& \textup{definition of }Y_3\textup{, }(\ref{item:5.1}),(\ref{item:5.4}),\textup{ definition of }Y''_3
\\
F''_6\to \gamma
& 3,4,Y_1,Y_7,F_7,F'_6
& \textup{definition of }F_6\textup{, }(\ref{item:5.1}),(\ref{item:5.4}),(\ref{item:5.5}),\textup{ definition of }F''_6
\\
F'_4\overset{***}{\to}\beta
& 2,F_5,F''_4,F'_6
& \textup{definition of }F_4\textup{, }(\ref{item:5.1}),(\ref{item:5.5}),\textup{ Lemma \ref{claim:thm3-3star}}
\\
Y'_3\to \alpha
& 1,7,F_3,Y''_3,F''_4
& \textup{definition of }Y_3\textup{, }(\ref{item:5.1}),(\ref{item:5.4}),\textup{ definition of }F''_4
\\
Y'_1\overset{*}{\to}\delta
& 5,6,Y''_1,Y_2,F_1,F_2
& \textup{definition of }Y_1\textup{, }(\ref{item:5.1}),(\ref{item:5.3}),(\ref{item:5.4}),\textup{ Lemma \ref{claim:thm3-star}}
\\
F'_6\to \gamma
& 3,4,Y''_1,Y_7,F_7,F''_6
& \textup{definition of }F_6\textup{, }(\ref{item:5.1}),(\ref{item:5.4}),(\ref{item:5.5}),\textup{ definition of }Y''_1
\end{array}
\]

Thus every step is legal, and the sequence ends with $I_2=\gamma$. Hence Lemma \ref{claim:CoverI} applies. \subcaseend

\noindent\phantomsection\label{subcase:1.5}\textbf{Subcase 1.5.} Suppose $\varphi(3,5,7)=(\beta,\gamma,\varepsilon)$.

Since $\varphi(1,2,3,4,5,6,7)=(\alpha,\beta,\beta,\gamma,\gamma,\delta,\varepsilon)$, the colors allowed by the cycle are
$Y_1,F_1:\{\gamma,\delta\}$,
$Y_2:\{\gamma,\delta,\varepsilon\}$,
$Y_3,F_3:\{\alpha,\delta,\varepsilon\}$,
$Y_7:\{\beta,\gamma\}$,
$F_2:\{\delta,\varepsilon\}$,
$F_4:\{\alpha,\varepsilon\}$,
$F_5:\{\alpha,\beta\}$,
$F_6:\{\beta\}$, and
$F_7:\{\gamma\}$.

Let $Y'_3=\{y\in Y_3: y \text{ has a neighbor in }F_4\}$ and let $Y''_3=Y_3\setminus Y'_3$. Then $Y''_3$ is anticomplete to $F_4$, and by Lemma \ref{claim:thm3-star}, $Y'_3$ is anticomplete to $F_1$. Let $F'_3=\{f\in F_3: f \text{ has a neighbor in }Y_2\}$ and let $F''_3=F_3\setminus F'_3$. Then $F''_3$ is anticomplete to $Y_2$. Let $F'_5=\{f\in F_5: f \text{ has a neighbor in }Y_7\}$ and let $F''_5=F_5\setminus F'_5$. Then $F''_5$ is anticomplete to $Y_7$, and by Lemma \ref{claim:thm3-3star}, $F'_3$ is anticomplete to $F'_5$.

The following table gives the recoloring sequence and verifies each step.
\[
\begin{array}{c|c|c}
\textup{step} & \textup{obstructions} & \textup{Justification: anticomplete by} \\
\hline
Y_1\to \gamma
& 4,5,F_1,Y_2,Y_7,F_7
& \textup{definition of }Y_1\textup{, }(\ref{item:5.3}),(\ref{item:5.4})
\\
F''_5\to \beta
& 2,3,Y_7,F'_5,F_6
& \textup{definition of }F_5\textup{, }(\ref{item:5.1}),(\ref{item:5.5}),\textup{ definition of }F''_5
\\
F'_3\overset{***}{\to}\alpha
& 1,Y_3,F''_3,F_4,F'_5
& \textup{definition of }F_3\textup{, }(\ref{item:5.1}),(\ref{item:5.4}),(\ref{item:5.5}),\textup{ Lemma \ref{claim:thm3-3star}}
\\
F''_3\to \varepsilon
& 7,Y_2,Y_3,F_2,F_4
& \textup{definition of }F_3\textup{, }(\ref{item:5.4}),(\ref{item:5.5}),\textup{ definition of }F''_3
\\
Y_2\to \delta
& 6,F_1,F_2,Y_3
& \textup{definition of }Y_2\textup{, }(\ref{item:5.3}),(\ref{item:5.4})
\\
Y'_3\overset{*}{\to}\delta
& 6,Y_2,F_1,F_2,Y''_3
& \textup{definition of }Y_3\textup{, }(\ref{item:5.1}),(\ref{item:5.3}),(\ref{item:5.4}),\textup{ Lemma \ref{claim:thm3-star}}
\\
Y''_3\to \varepsilon
& 7,F_2,F''_3,F_4
& \textup{definition of }Y_3\textup{, }(\ref{item:5.4}),\textup{ definition of }Y''_3
\\
F_4\to \alpha
& 1,F'_3,F'_5
& \textup{definition of }F_4\textup{, }(\ref{item:5.5})
\\
Y'_3\to \varepsilon
& 7,F_2,F''_3,Y''_3
& \textup{definition of }Y_3\textup{, }(\ref{item:5.1}),(\ref{item:5.4})
\\
F'_3\to \varepsilon
& 7,F_2,F''_3,Y_3
& \textup{definition of }F_3\textup{, }(\ref{item:5.1}),(\ref{item:5.4}),(\ref{item:5.5})
\\
F_2\to \varepsilon
& 7,F_3,Y_3
& \textup{definition of }F_2\textup{, }(\ref{item:5.4}),(\ref{item:5.5})
\end{array}
\]

Thus every step is legal, and the sequence ends with $I_4=\varepsilon$. Hence Lemma \ref{claim:CoverI} applies. \subcaseend

\noindent\phantomsection\label{subcase:1.6}\textbf{Subcase 1.6.} Suppose $\varphi(3,5,7)=(\beta,\varepsilon,\alpha)$. 

Since $\varphi(1,2,3,4,5,6,7)=(\alpha,\beta,\beta,\gamma,\varepsilon,\delta,\alpha)$, the colors allowed by the cycle are
$Y_1,Y_2,F_1:\{\gamma,\delta,\varepsilon\}$,
$Y_3:\{\alpha,\delta,\varepsilon\}$,
$Y_7:\{\beta,\gamma,\varepsilon\}$,
$F_2:\{\delta,\varepsilon\}$,
$F_3:\{\alpha,\delta\}$,
$F_4:\{\alpha\}$,
$F_5:\{\beta\}$,
$F_6:\{\beta,\gamma\}$, and
$F_7:\{\gamma,\varepsilon\}$.

Let $Y'_7=\{y\in Y_7: y \text{ has a neighbor in }F_5\}$ and let $Y''_7=Y_7\setminus Y'_7$. Then $Y''_7$ is anticomplete to $F_5$, and by Lemma \ref{claim:thm3-star}, $Y'_7$ is anticomplete to $F_1$. Similarly, let $Y'_3=\{y\in Y_3: y \text{ has a neighbor in }F_1\}$ and let $Y''_3=Y_3\setminus Y'_3$. Then $Y''_3$ is anticomplete to $F_1$, and by Lemma \ref{claim:thm3-star}, $Y'_3$ is anticomplete to $F_4$.

The following table gives the recoloring sequence and verifies each step.
\[
\begin{array}{c|c|c}
\textup{step} & \textup{obstructions} & \textup{Justification: anticomplete by} \\
\hline
F_3\to \alpha
& 1,7,Y_3,F_4
& \textup{definition of }F_3\textup{, }(\ref{item:5.4}),(\ref{item:5.5})
\\
F_6\to \beta
& 2,3,Y_7,F_5
& \textup{definition of }F_6\textup{, }(\ref{item:5.4}),(\ref{item:5.5})
\\
Y_1\to \gamma
& 4,Y_2,F_1,Y_7,F_7
& \textup{definition of }Y_1\textup{, }(\ref{item:5.3}),(\ref{item:5.4})
\\
Y_2\to \delta
& 6,F_1,F_2,Y_3
& \textup{definition of }Y_2\textup{, }(\ref{item:5.3}),(\ref{item:5.4})
\\
F_2\to \delta
& 6,Y_2,F_1,Y_3
& \textup{definition of }F_2\textup{, }(\ref{item:5.4}),(\ref{item:5.5})
\\
F_7\to \gamma
& 4,Y_1,F_1,Y_7
& \textup{definition of }F_7\textup{, }(\ref{item:5.4}),(\ref{item:5.5})
\\
Y'_7\overset{*}{\to} \gamma
& 4,Y_1,Y''_7,F_1,F_7
& \textup{definition of }Y_7\textup{, }(\ref{item:5.1}),(\ref{item:5.3}),(\ref{item:5.4}),\textup{ Lemma \ref{claim:thm3-star}}
\\
Y''_7\to \beta
& 2,3,F_6,F_5
& \textup{definition of }Y_7\textup{, }(\ref{item:5.4}),\textup{ definition of }Y''_7
\\
Y'_3\overset{*}{\to} \alpha
& 1,7,F_3,Y''_3,F_4
& \textup{definition of }Y_3\textup{, }(\ref{item:5.1}),(\ref{item:5.4}),\textup{ Lemma \ref{claim:thm3-star}}
\\
Y''_3\to \delta
& 6,Y_2,F_2,F_1
& \textup{definition of }Y_3\textup{, }(\ref{item:5.3}),(\ref{item:5.4}),\textup{ definition of }Y''_3
\\
F_1\to \varepsilon
& 5
& \textup{definition of }F_1
\\
Y_1\to \varepsilon
& 5,F_1
& \textup{definition of }Y_1\textup{, }(\ref{item:5.4})
\\
Y_2\to \varepsilon
& 5,F_1,Y_1
& \textup{definition of }Y_2\textup{, }(\ref{item:5.3}),(\ref{item:5.4})
\\
6\to \varepsilon
& 5,F_1,Y_1,Y_2
& \textup{definition of }F_1\textup{, }Y_1\textup{, and }Y_2\textup{, and }56\notin E(G)
\end{array}
\]

Thus every step is legal, and the sequence ends with $I_3=\varepsilon$. Hence Lemma \ref{claim:CoverI} applies. \subcaseend

\noindent\phantomsection\label{subcase:1.7}\textbf{Subcase 1.7.} Suppose $\varphi(3,5,7)=(\beta,\varepsilon,\delta)$.

Since $\varphi(1,2,3,4,5,6,7)=(\alpha,\beta,\beta,\gamma,\varepsilon,\delta,\delta)$, the colors allowed by the cycle are
$Y_1,F_1:\{\gamma,\varepsilon\}$,
$Y_2:\{\gamma,\delta,\varepsilon\}$,
$Y_3:\{\alpha,\delta,\varepsilon\}$,
$Y_7:\{\beta,\gamma,\varepsilon\}$,
$F_2:\{\delta,\varepsilon\}$,
$F_3:\{\alpha,\delta\}$,
$F_4:\{\alpha\}$,
$F_5:\{\alpha,\beta\}$,
$F_6:\{\beta,\gamma\}$, and
$F_7:\{\gamma,\varepsilon\}$.

Let $Y'_2=\{y\in Y_2: y \text{ has a neighbor in }F_3\}$ and let $Y''_2=Y_2\setminus Y'_2$. Then $Y''_2$ is anticomplete to $F_3$, and by Lemma \ref{claim:thm3-star}, $Y'_2$ is anticomplete to $F_7$. Similarly, let $Y'_7=\{y\in Y_7: y \text{ has a neighbor in }F_5\}$ and let $Y''_7=Y_7\setminus Y'_7$. Then $Y''_7$ is anticomplete to $F_5$. Let $F'_3=\{f\in F_3: f \text{ has a neighbor in }Y'_2\}$ and let $F''_3=F_3\setminus F'_3$. Then $F''_3$ is anticomplete to $Y'_2$. Finally, let $F'_5=\{f\in F_5: f \text{ has a neighbor in }Y'_7\}$ and let $F''_5=F_5\setminus F'_5$. Then $F''_5$ is anticomplete to $Y'_7$. Also, by Lemma \ref{claim:thm3-3star}, $F'_3$ is anticomplete to $F'_5$.

The following table gives the recoloring sequence and verifies each step.

\[
\begin{array}{c|c|c}
\textup{step} & \textup{obstructions} & \textup{Justification: anticomplete by} \\
\hline
F_6\to \beta
& 2,3,Y_7,F_5
& \textup{definition of }F_6\textup{, }(\ref{item:5.4}),(\ref{item:5.5})
\\
F_2\to \delta
& 6,7,Y_2,Y_3,F_3
& \textup{definition of }F_2\textup{, }(\ref{item:5.4}),(\ref{item:5.5})
\\
Y_3\to \delta
& 6,7,F_2,Y_2,F_3
& \textup{definition of }Y_3\textup{, }(\ref{item:5.3}),(\ref{item:5.4})
\\
F_4\to \alpha
& 1,F_3,F_5
& \textup{definition of }F_4\textup{, }(\ref{item:5.5})
\\
Y''_2\to \delta
& 6,7,F_2,Y_3,Y'_2,F_3
& \textup{definition of }Y_2\textup{, }(\ref{item:5.1}),(\ref{item:5.3}),(\ref{item:5.4}),\textup{ definition of }Y''_2
\\
F''_3\to \delta
& 6,7,F_2,Y_3,Y'_2,Y''_2,F'_3
& \textup{definition of }F_3\textup{, }(\ref{item:5.1}),(\ref{item:5.4}),(\ref{item:5.5}),\textup{ definitions of }Y''_2\textup{ and }F''_3
\\
F''_5\to \beta
& 2,3,Y_7,F'_5,F_6
& \textup{definition of }F_5\textup{, }(\ref{item:5.1}),(\ref{item:5.5}),\textup{ definitions of }Y''_7\textup{ and }F''_5
\\
F'_5\overset{***}{\to}\alpha
& 1,F_4,F'_3
& \textup{definition of }F_5\textup{, }(\ref{item:5.5}),\textup{ Lemma \ref{claim:thm3-3star}}
\\
F'_3\overset{***}{\to}\alpha
& 1,F_4,F'_5
& \textup{definition of }F_3\textup{, }(\ref{item:5.5}),\textup{ Lemma \ref{claim:thm3-3star}}
\\
Y_7\to \beta
& 2,3,F_6,F''_5
& \textup{definition of }Y_7\textup{, }(\ref{item:5.4}),\textup{ definitions of }Y''_7\textup{ and }F''_5
\\
Y'_2\overset{*}{\to}\varepsilon
& 5,Y_1,F_1,F_7
& \textup{definition of }Y_2\textup{, }(\ref{item:5.3}),(\ref{item:5.4}),\textup{ Lemma \ref{claim:thm3-star}}
\\
F'_3\to \delta
& 6,7,F_2,Y_3,Y''_2,F''_3
& \textup{definition of }F_3\textup{, }(\ref{item:5.1}),(\ref{item:5.4}),(\ref{item:5.5}),\textup{ definition of }Y''_2
\end{array}
\]

Thus every step is legal, and the sequence ends with $I_4=\delta$. Hence Lemma \ref{claim:CoverI} applies. \subcaseend

\noindent\phantomsection\label{subcase:1.8}\textbf{Subcase 1.8.} Suppose $\varphi(3,5,7)=(\beta,\delta,\varepsilon)$. 

Since $\varphi(1,2,3,4,5,6,7)=(\alpha,\beta,\beta,\gamma,\delta,\delta,\varepsilon)$, the colors allowed by the cycle are
$Y_1,F_1:\{\gamma,\delta\}$,
$Y_2:\{\gamma,\delta,\varepsilon\}$,
$Y_3:\{\alpha,\delta,\varepsilon\}$,
$Y_7, F_6:\{\beta,\gamma\}$,
$F_2:\{\delta,\varepsilon\}$,
$F_3,F_4:\{\alpha,\varepsilon\}$,
$F_5:\{\alpha,\beta\}$, and
$F_7:\{\gamma\}$.

Let $Y'_3=\{y\in Y_3: y \text{ has a neighbor in }F_4\}$ and let $Y''_3=Y_3\setminus Y'_3$. Then $Y''_3$ is anticomplete to $F_4$, and by Lemma \ref{claim:thm3-star}, $Y'_3$ is anticomplete to $F_1$.

The following table gives the recoloring sequence and verifies each step.
\[
\begin{array}{c|c|c}
\textup{step} & \textup{obstructions} & \textup{Justification: anticomplete by} \\
\hline
F_6\to \beta
& 2,3,Y_7,F_5
& \textup{definition of }F_6\textup{, }(\ref{item:5.4}),(\ref{item:5.5})
\\
Y_1\to \gamma
& 4,Y_2,F_1,Y_7,F_7
& \textup{definition of }Y_1\textup{, }(\ref{item:5.3}),(\ref{item:5.4})
\\
Y''_3\to \varepsilon
& 7,Y_2,Y'_3,F_2,F_3,F_4
& \textup{definition of }Y_3\textup{, }(\ref{item:5.1}),(\ref{item:5.3}),(\ref{item:5.4}),\textup{ definition of }Y''_3
\\
Y'_3\overset{*}{\to} \delta
& 5,6,Y_2,F_1,F_2
& \textup{definition of }Y_3\textup{, }(\ref{item:5.3}),(\ref{item:5.4}),\textup{ Lemma \ref{claim:thm3-star}}
\\
F_4\to \alpha
& 1,F_3,F_5
& \textup{definition of }F_4\textup{, }(\ref{item:5.5})
\\
Y'_3\to \varepsilon
& 7,Y_2,F_2,F_3,Y''_3
& \textup{definition of }Y_3\textup{, }(\ref{item:5.1}),(\ref{item:5.3}),(\ref{item:5.4})
\\
Y_2\to \delta
& 5,6,F_1,F_2
& \textup{definition of }Y_2\textup{, }(\ref{item:5.4})
\\
F_3\to \varepsilon
& 7,F_2,Y_3
& \textup{definition of }F_3\textup{, }(\ref{item:5.4}),(\ref{item:5.5})
\\
F_2\to \varepsilon
& 7,F_3,Y_3
& \textup{definition of }F_2\textup{, }(\ref{item:5.4}),(\ref{item:5.5})
\end{array}
\]

Thus every step is legal, and the sequence ends with $I_4=\varepsilon$. Hence Lemma \ref{claim:CoverI} applies. \subcaseend

\noindent\phantomsection\label{subcase:1.9}\textbf{Subcase 1.9.} Suppose $\varphi(3,5,7)=(\gamma,\delta,\varepsilon)$.

Since $\varphi(1,2,3,4,5,6,7)=(\alpha,\beta,\gamma,\gamma,\delta,\delta,\varepsilon)$, the colors allowed by the cycle are
$Y_1:\{\gamma,\delta\}$,
$F_1:\{\delta\}$,
$Y_2,F_2:\{\delta,\varepsilon\}$,
$Y_3:\{\alpha,\delta,\varepsilon\}$,
$F_3:\{\alpha,\varepsilon\}$,
$Y_7, F_6:\{\beta,\gamma\}$,
$F_4:\{\alpha,\beta,\varepsilon\}$,
$F_5:\{\alpha,\beta\}$, and
$F_7:\{\gamma\}$.

Let $Y'_1=\{y\in Y_1: y \text{ has a neighbor in }F_6\}$ and let $Y''_1=Y_1\setminus Y'_1$. Then $Y''_1$ is anticomplete to $F_6$, and by Lemma \ref{claim:thm3-star}, $Y'_1$ is anticomplete to $F_2$. Similarly, let $Y'_3=\{y\in Y_3: y \text{ has a neighbor in }F_4\}$ and let $Y''_3=Y_3\setminus Y'_3$. Then $Y''_3$ is anticomplete to $F_4$, and by Lemma \ref{claim:thm3-star}, $Y'_3$ is anticomplete to $F_1$. Now let $F'_4=\{f\in F_4: f \text{ has a neighbor in }Y'_3\}$ and let $F''_4=F_4\setminus F'_4$, and let $F'_6=\{f\in F_6: f \text{ has a neighbor in }Y'_1\}$ and let $F''_6=F_6\setminus F'_6$. Then $F''_4$ is anticomplete to $Y'_3$, $F''_6$ is anticomplete to $Y'_1$, and by Lemma \ref{claim:thm3-3star}, $F'_4$ is anticomplete to $F'_6$.

The following table gives the recoloring sequence and verifies each step.
\[
\begin{array}{c|c|c}
\textup{step} & \textup{obstructions} & \textup{Justification: anticomplete by} \\
\hline
Y_2\to \delta
& 5,6,Y_1,F_1,F_2,Y_3
& \textup{definition of }Y_2\textup{, }(\ref{item:5.3}),(\ref{item:5.4})
\\
F_3\to \varepsilon
& 7,F_2,Y_3,F_4
& \textup{definition of }F_3\textup{, }(\ref{item:5.4}),(\ref{item:5.5})
\\
Y_7\to \gamma
& 3,4,Y_1,F_6,F_7
& \textup{definition of }Y_7\textup{, }(\ref{item:5.3}),(\ref{item:5.4})
\\
F_5\to \beta
& 2,F_4,F_6
& \textup{definition of }F_5\textup{, }(\ref{item:5.5})
\\
Y''_3\to \varepsilon
& 7,F_2,F_3,Y'_3,F_4
& \textup{definition of }Y_3\textup{, }(\ref{item:5.1}),(\ref{item:5.4}),\textup{ definition of }Y''_3
\\
Y''_1\to \gamma
& 3,4,Y_7,F_6,F_7,Y'_1
& \textup{definition of }Y_1\textup{, }(\ref{item:5.1}),(\ref{item:5.3}),(\ref{item:5.4}),\textup{ definition of }Y''_1
\\
F''_4\to \alpha
& 1,Y'_3,F'_4
& \textup{definition of }F_4\textup{, }(\ref{item:5.1}),\textup{ definition of }F''_4
\\
F''_6\to \gamma
& 3,4,Y_7,F_7,Y'_1,Y''_1,F'_6
& \textup{definition of }F_6\textup{, }(\ref{item:5.1}),(\ref{item:5.4}),(\ref{item:5.5}),\textup{ definition of }F''_6
\\
F'_4\overset{***}{\to}\beta
& 2,F_5,F'_6
& \textup{definition of }F_4\textup{, }(\ref{item:5.5}),\textup{ Lemma \ref{claim:thm3-3star}}
\\
F'_6\overset{***}{\to}\beta
& 2,F_5,F'_4
& \textup{definition of }F_6\textup{, }(\ref{item:5.5}),\textup{ Lemma \ref{claim:thm3-3star}}
\\
Y'_1\to \gamma
& 3,4,Y_7,F_7,Y''_1,F''_6
& \textup{definition of }Y_1\textup{, }(\ref{item:5.1}),(\ref{item:5.3}),(\ref{item:5.4}),\textup{ definition of }F''_6
\\
Y'_3\overset{*}{\to}\delta
& 5,6,Y_2,F_1,F_2
& \textup{definition of }Y_3\textup{, }(\ref{item:5.3}),(\ref{item:5.4}),\textup{ Lemma \ref{claim:thm3-star}}
\\
Y'_3\to \varepsilon
& 7,F_2,F_3,Y''_3
& \textup{definition of }Y_3\textup{, }(\ref{item:5.1}),(\ref{item:5.4})
\\
F_2\to \varepsilon
& 7,F_3,Y_3
& \textup{definition of }F_2\textup{, }(\ref{item:5.4}),(\ref{item:5.5})
\end{array}
\]

Thus every step is legal, and the sequence ends with $I_4=\varepsilon$. Hence Lemma \ref{claim:CoverI} applies. \subcaseend

\noindent\phantomsection\label{subcase:1.10}\textbf{Subcase 1.10.} Suppose $\varphi(3,5,7)=(\gamma,\varepsilon,\alpha)$.

Since $\varphi(1,2,3,4,5,6,7)=(\alpha,\beta,\gamma,\gamma,\varepsilon,\delta,\alpha)$, the colors allowed by the cycle are
$Y_1:\{\gamma,\delta,\varepsilon\}$,
$Y_2, F_1, F_2:\{\delta,\varepsilon\}$,
$Y_3:\{\alpha,\delta,\varepsilon\}$,
$Y_7:\{\beta,\gamma,\varepsilon\}$,
$F_3:\{\alpha,\delta\}$,
$F_4:\{\alpha,\beta\}$,
$F_5:\{\beta\}$,
$F_6:\{\beta,\gamma\}$, and
$F_7:\{\gamma,\varepsilon\}$.

Let $Y'_3=\{y\in Y_3: y \text{ has a neighbor in }F_1\}$ and let $Y''_3=Y_3\setminus Y'_3$. Then $Y''_3$ is anticomplete to $F_1$, and by Lemma \ref{claim:thm3-star}, $Y'_3$ is anticomplete to $F_4$. Similarly, let $Y'_1=\{y\in Y_1: y \text{ has a neighbor in }F_6\}$ and let $Y''_1=Y_1\setminus Y'_1$. Then $Y''_1$ is anticomplete to $F_6$, and by Lemma \ref{claim:thm3-star}, $Y'_1$ is anticomplete to $F_2$.

Now let $F'_4=\{f\in F_4: f \text{ has a neighbor in }Y''_3\}$ and let $F''_4=F_4\setminus F'_4$, and let $F'_6=\{f\in F_6: f \text{ has a neighbor in }Y'_1\}$ and let $F''_6=F_6\setminus F'_6$. Then $F''_4$ is anticomplete to $Y''_3$, $F''_6$ is anticomplete to $Y'_1$, and by Lemma \ref{claim:thm3-3star}, $F'_4$ is anticomplete to $F'_6$.

The following table gives the recoloring sequence and verifies each step.
\[
\begin{array}{c|c|c}
\textup{step} & \textup{obstructions} & \textup{Justification: anticomplete by} \\
\hline
Y'_3\overset{*}{\to}\alpha
& 1,7,Y''_3,F_3,F_4
& \textup{definition of }Y_3\textup{, }(\ref{item:5.1}),(\ref{item:5.4}),\textup{ Lemma \ref{claim:thm3-star}}
\\
F''_6\to\gamma
& 3,4,Y'_1,Y''_1,Y_7,F_7,F'_6
& \textup{definition of }F_6\textup{, }(\ref{item:5.1}),(\ref{item:5.4}),(\ref{item:5.5}),\textup{ definitions of }Y''_1\textup{ and }F''_6
\\
Y_7\to\gamma
& 3,4,Y_1,F''_6,F'_6,F_7
& \textup{definition of }Y_7\textup{, }(\ref{item:5.3}),(\ref{item:5.4})
\\
F'_4\overset{***}{\to}\beta
& 2,F''_4,F_5,F'_6
& \textup{definition of }F_4\textup{, }(\ref{item:5.1}),(\ref{item:5.5}),\textup{ Lemma \ref{claim:thm3-3star}}
\\
Y''_3\to\alpha
& 1,7,Y'_3,F_3,F''_4
& \textup{definition of }Y_3\textup{, }(\ref{item:5.1}),(\ref{item:5.4}),\textup{ definition of }F''_4
\\
Y'_1\overset{*}{\to}\varepsilon
& 5,Y''_1,Y_2,F_1,F_2,F_7
& \textup{definition of }Y_1\textup{, }(\ref{item:5.1}),(\ref{item:5.3}),(\ref{item:5.4}),\textup{ Lemma \ref{claim:thm3-star}}
\\
F'_6\to\gamma
& 3,4,Y''_1,Y_7,F_7,F''_6
& \textup{definition of }F_6\textup{, }(\ref{item:5.1}),(\ref{item:5.4}),(\ref{item:5.5}),\textup{ definition of }Y''_1
\\
F_7\to\gamma
& 3,4,Y''_1,Y_7,F_6
& \textup{definition of }F_7\textup{, }(\ref{item:5.4}),(\ref{item:5.5})
\end{array}
\]

Thus every step is legal, and the sequence ends with $I_2=\gamma$. Hence Lemma \ref{claim:CoverI} applies. \subcaseend

\noindent\phantomsection\label{subcase:1.11}\textbf{Subcase 1.11.} Suppose $\varphi(3,5,7)=(\gamma,\varepsilon,\delta)$. 

Since $\varphi(1,2,3,4,5,6,7)=(\alpha,\beta,\gamma,\gamma,\varepsilon,\delta,\delta)$, the colors allowed by the cycle are
$Y_1:\{\gamma,\varepsilon\}$,
$F_1:\{\varepsilon\}$,
$Y_2,F_2:\{\delta,\varepsilon\}$,
$Y_3:\{\alpha,\delta,\varepsilon\}$,
$F_3:\{\alpha,\delta\}$,
$Y_7:\{\beta,\gamma,\varepsilon\}$,
$F_4,F_5:\{\alpha,\beta\}$,
$F_6:\{\beta,\gamma\}$, and
$F_7:\{\gamma,\varepsilon\}$.

The following table gives the recoloring sequence and verifies each step.
\[
\begin{array}{c|c|c}
\textup{step} & \textup{obstructions} & \textup{Justification: anticomplete by} \\
\hline
F_2\to \delta
& 6,7,Y_2,Y_3,F_3
& \textup{definition of }F_2\textup{, }(\ref{item:5.4}),(\ref{item:5.5})
\\
Y_7\to \gamma
& 3,4,Y_1,F_6,F_7
& \textup{definition of }Y_7\textup{, }(\ref{item:5.3}),(\ref{item:5.4})
\\
Y_3\to \delta
& 6,7,Y_2,F_2,F_3
& \textup{definition of }Y_3\textup{, }(\ref{item:5.3}),(\ref{item:5.4})
\\
F_7\to \gamma
& 3,4,Y_1,Y_7,F_6
& \textup{definition of }F_7\textup{, }(\ref{item:5.4}),(\ref{item:5.5})
\\
Y_1\to \varepsilon
& 5,F_1,Y_2
& \textup{definition of }Y_1\textup{, }(\ref{item:5.3}),(\ref{item:5.4})
\\
Y_2\to \varepsilon
& 5,F_1,Y_1
& \textup{definition of }Y_2\textup{, }(\ref{item:5.3}),(\ref{item:5.4})
\\
6\to \varepsilon
& 5,F_1,Y_1,Y_2
& \textup{definition of }F_1\textup{, }Y_1\textup{, and }Y_2\textup{, and }56\notin E(G)
\end{array}
\]

Thus every step is legal, and the sequence ends with $I_3=\varepsilon$. Hence Lemma \ref{claim:CoverI} applies. \subcaseend

\noindent\phantomsection\label{subcase:1.12}\textbf{Subcase 1.12.} Suppose $\varphi(3,5,7)=(\varepsilon,\gamma,\alpha)$. 

Since $\varphi(1,2,3,4,5,6,7)=(\alpha,\beta,\varepsilon,\gamma,\gamma,\delta,\alpha)$, the colors allowed by the cycle are
$Y_1:\{\gamma,\delta,\varepsilon\}$,
$Y_2:\{\gamma,\delta\}$,
$Y_3,F_3:\{\alpha,\delta\}$,
$Y_7:\{\beta,\gamma,\varepsilon\}$,
$F_1:\{\gamma,\delta\}$,
$F_2:\{\delta\}$,
$F_4:\{\alpha,\beta\}$,
$F_5,F_6:\{\beta,\varepsilon\}$, and
$F_7:\{\gamma,\varepsilon\}$.

Let $Y'_7=\{y\in Y_7: y \text{ has a neighbor in }F_5\}$ and let $Y''_7=Y_7\setminus Y'_7$. Then $Y''_7$ is anticomplete to $F_5$, and by Lemma \ref{claim:thm3-star}, $Y'_7$ is anticomplete to $F_1$.

The following table gives the recoloring sequence and verifies each step.
\[
\begin{array}{c|c|c}
\textup{step} & \textup{obstructions} & \textup{Justification: anticomplete by} \\
\hline
F_3\to \alpha
& 1,7,Y_3,F_4
& \textup{definition of }F_3\textup{, }(\ref{item:5.4}),(\ref{item:5.5})
\\
Y_2\to \delta
& 6,Y_1,F_1,F_2,Y_3
& \textup{definition of }Y_2\textup{, }(\ref{item:5.3}),(\ref{item:5.4})
\\
Y_1\to \gamma
& 4,5,Y_2,F_1,Y_7,F_7
& \textup{definition of }Y_1\textup{, }(\ref{item:5.3}),(\ref{item:5.4})
\\
Y'_7\overset{*}{\to}\gamma
& 4,5,Y_1,Y''_7,F_1,F_7
& \textup{definition of }Y_7\textup{, }(\ref{item:5.1}),(\ref{item:5.3}),(\ref{item:5.4}),\textup{ Lemma \ref{claim:thm3-star}}
\\
Y''_7\to \varepsilon
& 3,F_5,F_6,F_7
& \textup{definition of }Y_7\textup{, }(\ref{item:5.4}),\textup{ definition of }Y''_7
\\
F_6\to \varepsilon
& 3,Y''_7,F_5,F_7
& \textup{definition of }F_6\textup{, }(\ref{item:5.4}),(\ref{item:5.5})
\\
F_5\to \beta
& 2,F_4
& \textup{definition of }F_5\textup{, }(\ref{item:5.5})
\\
F_4\to \beta
& 2,F_5
& \textup{definition of }F_4\textup{, }(\ref{item:5.5})
\\
1\to \beta
& 2,F_4,F_5
& \textup{definition of }F_4\textup{ and }F_5\textup{, and }12\notin E(G)
\end{array}
\]

Thus every step is legal, and the sequence ends with $I_1=\beta$. Hence Lemma \ref{claim:CoverI} applies. \subcaseend

\noindent\phantomsection\label{subcase:1.13}\textbf{Subcase 1.13.} Suppose $\varphi(3,5,7)=(\varepsilon,\gamma,\delta)$. 

Since $\varphi(1,2,3,4,5,6,7)=(\alpha,\beta,\varepsilon,\gamma,\gamma,\delta,\delta)$, the colors allowed by the cycle are
$Y_1:\{\gamma,\varepsilon\}$,
$F_1:\{\gamma\}$,
$Y_2:\{\gamma,\delta\}$,
$Y_3,F_3:\{\alpha,\delta\}$,
$Y_7:\{\beta,\gamma,\varepsilon\}$,
$F_2:\{\delta\}$,
$F_4:\{\alpha,\beta\}$,
$F_5:\{\alpha,\beta,\varepsilon\}$,
$F_6:\{\beta,\varepsilon\}$, and
$F_7:\{\gamma,\varepsilon\}$.

Let $Y'_2=\{y\in Y_2: y \text{ has a neighbor in }F_3\}$ and let $Y''_2=Y_2\setminus Y'_2$. Then $Y''_2$ is anticomplete to $F_3$, and by Lemma \ref{claim:thm3-star}, $Y'_2$ is anticomplete to $F_7$. Similarly, let $Y'_7=\{y\in Y_7: y \text{ has a neighbor in }F_5\}$ and let $Y''_7=Y_7\setminus Y'_7$. Then $Y''_7$ is anticomplete to $F_5$. Let $F'_3=\{f\in F_3: f \text{ has a neighbor in }Y'_2\}$ and let $F''_3=F_3\setminus F'_3$. Then $F''_3$ is anticomplete to $Y'_2$. Finally, let $F'_5=\{f\in F_5: f \text{ has a neighbor in }Y'_7\}$ and let $F''_5=F_5\setminus F'_5$. Then $F''_5$ is anticomplete to $Y'_7$. Also, by Lemma \ref{claim:thm3-3star}, $F'_3$ is anticomplete to $F'_5$.

The following table gives the recoloring sequence and verifies each step.
\[
\begin{array}{c|c|c}
\textup{step} & \textup{obstructions} & \textup{Justification: anticomplete by} \\
\hline
F_2\to \delta
& 6,7,Y_2,Y_3,F_3
& \textup{definition of }F_2\textup{, }(\ref{item:5.4}),(\ref{item:5.5})
\\
Y_1\to \gamma
& 4,5,F_1,Y_2,Y_7,F_7
& \textup{definition of }Y_1\textup{, }(\ref{item:5.3}),(\ref{item:5.4})
\\
Y_3\to \delta
& 6,7,F_2,Y_2,F_3
& \textup{definition of }Y_3\textup{, }(\ref{item:5.3}),(\ref{item:5.4})
\\
F''_3\to \delta
& 6,7,F_2,Y_3,Y''_2,Y'_2,F'_3
& \textup{definition of }F_3\textup{, }(\ref{item:5.1}),(\ref{item:5.4}),(\ref{item:5.5}),\textup{ definitions of }Y''_2\textup{ and }F''_3
\\
F'_5\overset{***}{\to}\alpha
& 1,F''_5,F'_3,F_4
& \textup{definition of }F_5\textup{, }(\ref{item:5.1}),(\ref{item:5.5}),\textup{ Lemma \ref{claim:thm3-3star}}
\\
F''_5\to \beta
& 2,Y_7,F_4,F_6
& \textup{definition of }F_5\textup{, }(\ref{item:5.5}),\textup{ definitions of }Y''_7\textup{ and }F''_5
\\
F_4\to \alpha
& 1,F'_3,F'_5
& \textup{definition of }F_4\textup{, }(\ref{item:5.5})
\\
Y_7\to \beta
& 2,F_6,F''_5
& \textup{definition of }Y_7\textup{, }(\ref{item:5.4}),\textup{ definitions of }Y''_7\textup{ and }F''_5
\\
F_7\to \varepsilon
& 3,F_6,Y_7
& \textup{definition of }F_7\textup{, }(\ref{item:5.4}),(\ref{item:5.5})
\\
Y'_2\to \gamma
& 4,5,Y_1,Y''_2,F_1
& \textup{definition of }Y_2\textup{, }(\ref{item:5.1}),(\ref{item:5.3}),(\ref{item:5.4})
\\
F'_3\to \delta
& 6,7,F_2,Y_3,F''_3,Y''_2
& \textup{definition of }F_3\textup{, }(\ref{item:5.1}),(\ref{item:5.4}),(\ref{item:5.5}),\textup{ definition of }Y''_2
\end{array}
\]

Thus every step is legal, and the sequence ends with $I_4=\delta$. Hence Lemma \ref{claim:CoverI} applies. \subcaseend

\noindent\phantomsection\label{subcase:1.14}\textbf{Subcase 1.14.} Suppose $\varphi(3,5,7)=(\varepsilon,\delta,\alpha)$. 

Since $\varphi(1,2,3,4,5,6,7)=(\alpha,\beta,\varepsilon,\gamma,\delta,\delta,\alpha)$, the colors allowed by the cycle are
$Y_1:\{\gamma,\delta,\varepsilon\}$,
$Y_2:\{\gamma,\delta\}$,
$F_2:\{\delta\}$,
$Y_3:\{\alpha,\delta\}$,
$Y_7:\{\beta,\gamma,\varepsilon\}$,
$F_1:\{\gamma,\delta\}$,
$F_3:\{\alpha\}$,
$F_4:\{\alpha,\beta\}$,
$F_5:\{\beta,\varepsilon\}$,
$F_6:\{\beta,\gamma,\varepsilon\}$, and
$F_7:\{\gamma,\varepsilon\}$.

Let $Y'_7=\{y\in Y_7: y \text{ has a neighbor in }F_5\}$ and let $Y''_7=Y_7\setminus Y'_7$. Then $Y''_7$ is anticomplete to $F_5$, and by Lemma \ref{claim:thm3-star}, $Y'_7$ is anticomplete to $F_1$. Similarly, let $Y'_1=\{y\in Y_1: y \text{ has a neighbor in }F_6\}$ and let $Y''_1=Y_1\setminus Y'_1$. Then $Y''_1$ is anticomplete to $F_6$, and by Lemma \ref{claim:thm3-star}, $Y'_1$ is anticomplete to $F_2$. Also, let $Y'_3=\{y\in Y_3: y \text{ has a neighbor in }F_4\}$ and let $Y''_3=Y_3\setminus Y'_3$. Then $Y''_3$ is anticomplete to $F_4$.

Now let $F'_6=\{f\in F_6: f \text{ has a neighbor in }Y'_1\}$ and let $F''_6=F_6\setminus F'_6$. Then $F''_6$ is anticomplete to $Y'_1$. Finally, let $F'_4=\{f\in F_4: f \text{ has a neighbor in }Y'_3\}$ and let $F''_4=F_4\setminus F'_4$. Then $F''_4$ is anticomplete to $Y'_3$. Also, by Lemma \ref{claim:thm3-3star}, $F'_4$ is anticomplete to $F'_6$.

The following table gives the recoloring sequence and verifies each step.
\[
\begin{array}{c|c|c}
\textup{step} & \textup{obstructions} & \textup{Justification: anticomplete by} \\
\hline
Y_2\to\delta
& 5,6,Y_1,F_1,F_2,Y_3
& \textup{definition of }Y_2\textup{, }(\ref{item:5.3}),(\ref{item:5.4})
\\
F_3\to\alpha
& 1,7,Y_3,F_4
& \textup{definition of }F_3\textup{, }(\ref{item:5.4}),(\ref{item:5.5})
\\
Y'_7\overset{*}{\to}\gamma
& 4,Y_1,Y''_7,F_1,F_6,F_7
& \textup{definition of }Y_7\textup{, }(\ref{item:5.1}),(\ref{item:5.3}),(\ref{item:5.4}),\textup{ Lemma \ref{claim:thm3-star}}
\\
F_5\to\beta
& 2,Y''_7,F_4,F_6
& \textup{definition of }F_5\textup{, }(\ref{item:5.5}),\textup{ definition of }Y''_7
\\
Y_7\to\varepsilon
& 3,Y_1,F_6,F_7
& \textup{definition of }Y_7\textup{, }(\ref{item:5.3}),(\ref{item:5.4})
\\
F_7\to\varepsilon
& 3,Y_1,Y_7,F_6
& \textup{definition of }F_7\textup{, }(\ref{item:5.4}),(\ref{item:5.5})
\\
F''_6\to\varepsilon
& 3,Y_1,Y_7,F_7,F'_6
& \textup{definition of }F_6\textup{, }(\ref{item:5.1}),(\ref{item:5.4}),(\ref{item:5.5}),\textup{ definitions of }Y''_1\textup{ and }F''_6
\\
F'_4\overset{***}{\to}\beta
& 2,F_5,F''_4,F'_6
& \textup{definition of }F_4\textup{, }(\ref{item:5.1}),(\ref{item:5.5}),\textup{ Lemma \ref{claim:thm3-3star}}
\\
Y_3\to\alpha
& 1,7,F_3,F''_4
& \textup{definition of }Y_3\textup{, }(\ref{item:5.4}),\textup{ definitions of }Y''_3\textup{ and }F''_4
\\
Y'_1\overset{*}{\to}\delta
& 5,6,Y''_1,Y_2,F_1,F_2
& \textup{definition of }Y_1\textup{, }(\ref{item:5.1}),(\ref{item:5.3}),(\ref{item:5.4}),\textup{ Lemma \ref{claim:thm3-star}}
\\
F'_6\to\varepsilon
& 3,Y''_1,Y_7,F_7,F''_6
& \textup{definition of }F_6\textup{, }(\ref{item:5.1}),(\ref{item:5.4}),(\ref{item:5.5}),\textup{ definition of }Y''_1
\\
4\to\varepsilon
& 3,Y''_1,Y_7,F_7,F_6
& \textup{definition of }Y_1\textup{, }Y_7\textup{, }F_6\textup{, and }F_7\textup{, and }34\notin E(G)
\end{array}
\]

Thus every step is legal, and the sequence ends with $I_2=\varepsilon$. Hence Lemma \ref{claim:CoverI} applies. \subcaseend


\noindent\phantomsection\label{subcase:2.1}\textbf{Subcase 2.1.} Suppose $\varphi(1,3,5)=(\beta,\gamma,\alpha)$. 

Since $\varphi(1,2,3,4,5,6,7)=(\beta,\gamma,\gamma,\delta,\alpha,\alpha,\beta)$, the colors allowed by the cycle are
$Y_1,F_1:\{\alpha,\delta,\varepsilon\}$,
$Y_6,Y_7:\{\gamma,\delta,\varepsilon\}$,
$F_2:\{\alpha,\varepsilon\}$,
$F_3,F_4:\{\beta,\varepsilon\}$,
$F_5:\{\gamma,\varepsilon\}$,
$F_6:\{\gamma,\delta,\varepsilon\}$, and
$F_7:\{\delta,\varepsilon\}$.

The following table gives the recoloring sequence and verifies each step.
\[
\begin{array}{c|c|c}
\textup{step} & \textup{obstructions} & \textup{Justification: anticomplete by} \\
\hline
F_4\to \beta
& 1,7,F_3
& \textup{definition of }F_4\textup{, }(\ref{item:5.5})
\\
F_1\to \alpha
& 5,6,Y_1,F_2
& \textup{definition of }F_1\textup{, }(\ref{item:5.4}),(\ref{item:5.5})
\\
F_3\to \beta
& 1,7,F_4
& \textup{definition of }F_3\textup{, }(\ref{item:5.5})
\\
Y_7\to \delta
& 4,Y_1,Y_6,F_6,F_7
& \textup{definition of }Y_7\textup{, }(\ref{item:5.3}),(\ref{item:5.4})
\\
F_5\to \gamma
& 2,3,Y_6,F_6
& \textup{definition of }F_5\textup{, }(\ref{item:5.4}),(\ref{item:5.5})
\\
Y_6\to \gamma
& 2,3,F_5,F_6
& \textup{definition of }Y_6\textup{, }(\ref{item:5.4})
\\
F_6\to \gamma
& 2,3,Y_6,F_5
& \textup{definition of }F_6\textup{, }(\ref{item:5.4}),(\ref{item:5.5})
\\
Y_1\to \delta
& 4,Y_7,F_7
& \textup{definition of }Y_1\textup{, }(\ref{item:5.3}),(\ref{item:5.4})
\\
F_7\to \delta
& 4,Y_1,Y_7
& \textup{definition of }F_7\textup{, }(\ref{item:5.4})
\\
F_2\to \varepsilon
& \emptyset
& \textup{none}
\\
F_3\to \varepsilon
& F_2
& \textup{(\ref{item:5.5})}
\\
6\to \varepsilon
& F_2,F_3
& \textup{definition of }F_2\textup{ and }F_3
\\
7\to \varepsilon
& F_2,F_3,6
& \textup{definition of }F_2\textup{ and }F_3\textup{, and }67\notin E(G)
\end{array}
\]

Thus every step is legal, and the sequence ends with $I_4=\varepsilon$. Hence Lemma \ref{claim:CoverI} applies. \subcaseend

\noindent\phantomsection\label{subcase:2.2}\textbf{Subcase 2.2.} Suppose $\varphi(1,3,5)=(\beta,\gamma,\delta)$.

Since $\varphi(1,2,3,4,5,6,7)=(\beta,\gamma,\gamma,\delta,\delta,\alpha,\beta)$, the colors allowed by the cycle are
$Y_1, F_1:\{\alpha,\delta,\varepsilon\}$,
$Y_6:\{\gamma,\varepsilon\}$,
$Y_7:\{\gamma,\delta,\varepsilon\}$,
$F_2:\{\alpha,\varepsilon\}$,
$F_3:\{\alpha,\beta,\varepsilon\}$,
$F_4:\{\beta,\varepsilon\}$,
$F_5, F_6:\{\gamma,\varepsilon\}$,
 and
$F_7:\{\delta,\varepsilon\}$.

The following table gives the recoloring sequence and verifies each step.
\[
\begin{array}{c|c|c}
\textup{step} & \textup{obstructions} & \textup{Justification: anticomplete by} \\
\hline
Y_1\to\delta
& 4,5,F_1,Y_7,F_7
& \textup{definition of }Y_1\textup{, }(\ref{item:5.3}),(\ref{item:5.4})
\\
F_3\to\beta
& 1,7,F_4
& \textup{definition of }F_3\textup{, }(\ref{item:5.5})
\\
F_2\to\alpha
& 6,F_1
& \textup{definition of }F_2\textup{, }(\ref{item:5.5})
\\
F_7\to\delta
& 4,5,Y_1,F_1,Y_7
& \textup{definition of }F_7\textup{, }(\ref{item:5.4}),(\ref{item:5.5})
\\
F_1\to\alpha
& 6,F_2
& \textup{definition of }F_1\textup{, }(\ref{item:5.5})
\\
Y_7\to\delta
& 4,5,Y_1,F_7
& \textup{definition of }Y_7\textup{, }(\ref{item:5.3}),(\ref{item:5.4})
\\
F_5\to\gamma
& 2,3,Y_6,F_6
& \textup{definition of }F_5\textup{, }(\ref{item:5.4}),(\ref{item:5.5})
\\
F_4\to\beta
& 1,7,F_3
& \textup{definition of }F_4\textup{, }(\ref{item:5.5})
\\
F_6\to\varepsilon
& Y_6
& \textup{(\ref{item:5.4})}
\\
Y_6\to\varepsilon
& F_6
& \textup{(\ref{item:5.4})}
\\
Y_7\to\varepsilon
& Y_6,F_6
& \textup{(\ref{item:5.3}),(\ref{item:5.4})}
\\
3\to\varepsilon
& Y_6,Y_7,F_6
& \textup{definition of }Y_6\textup{, }Y_7\textup{, and }F_6
\\
4\to\varepsilon
& 3,Y_6,Y_7,F_6
& \textup{definition of }Y_6\textup{, }Y_7\textup{, and }F_6\textup{, and }34\notin E(G)
\end{array}
\]

Thus every step is legal, and the sequence ends with $I_2=\varepsilon$. Hence Lemma \ref{claim:CoverI} applies. \subcaseend

\noindent\phantomsection\label{subcase:2.3}\textbf{Subcase 2.3.} Suppose $\varphi(1,3,5)=(\beta,\delta,\alpha)$. 

Since $\varphi(1,2,3,4,5,6,7)=(\beta,\gamma,\delta,\delta,\alpha,\alpha,\beta)$, the colors allowed by the cycle are
$Y_1:\{\alpha,\delta,\varepsilon\}$,
$Y_6,Y_7:\{\gamma,\delta,\varepsilon\}$,
$F_1, F_2:\{\alpha,\varepsilon\}$,
$F_3:\{\beta,\varepsilon\}$,
$F_4:\{\beta,\gamma,\varepsilon\}$,
$F_5:\{\gamma,\varepsilon\}$,
$F_6:\{\gamma,\delta,\varepsilon\}$, and
$F_7:\{\delta,\varepsilon\}$.

The following table gives the recoloring sequence and verifies each step.
\[
\begin{array}{c|c|c}
\textup{step} & \textup{obstructions} & \textup{Justification: anticomplete by} \\
\hline
Y_7\to\delta
& 3,4,Y_1,Y_6,F_6,F_7
& \textup{definition of }Y_7\textup{, }(\ref{item:5.3}),(\ref{item:5.4})
\\
F_5\to\gamma
& 2,Y_6,F_4,F_6
& \textup{definition of }F_5\textup{, }(\ref{item:5.4}),(\ref{item:5.5})
\\
F_4\to\beta
& 1,7,F_3
& \textup{definition of }F_4\textup{, }(\ref{item:5.5})
\\
F_6\to\gamma
& 2,Y_6,F_5
& \textup{definition of }F_6\textup{, }(\ref{item:5.4}),(\ref{item:5.5})
\\
Y_6\to\gamma
& 2,F_5,F_6
& \textup{definition of }Y_6\textup{, }(\ref{item:5.4})
\\
F_7\to\delta
& 3,4,Y_1,Y_7
& \textup{definition of }F_7\textup{, }(\ref{item:5.4})
\\
Y_1\to\delta
& 3,4,Y_7,F_7
& \textup{definition of }Y_1\textup{, }(\ref{item:5.3}),(\ref{item:5.4})
\\
F_1\to\alpha
& 5,6,F_2
& \textup{definition of }F_1\textup{, }(\ref{item:5.5})
\\
F_2\to\varepsilon
& F_3
& \textup{(\ref{item:5.5})}
\\
F_3\to\varepsilon
& F_2
& \textup{(\ref{item:5.5})}
\\
6\to\varepsilon
& F_2,F_3
& \textup{definition of }F_2\textup{ and }F_3
\\
7\to\varepsilon
& F_2,F_3,6
& \textup{definition of }F_2\textup{ and }F_3\textup{, and }67\notin E(G)
\end{array}
\]

Thus every step is legal, and the sequence ends with $I_4=\varepsilon$. Hence Lemma \ref{claim:CoverI} applies. \subcaseend

\noindent\phantomsection\label{subcase:2.4}\textbf{Subcase 2.4.} Suppose $\varphi(1,3,5)=(\gamma,\delta,\alpha)$.

Since $\varphi(1,2,3,4,5,6,7)=(\gamma,\gamma,\delta,\delta,\alpha,\alpha,\beta)$, the colors allowed by the cycle are
$Y_1:\{\alpha,\delta,\varepsilon\}$,
$Y_6:\{\gamma,\delta,\varepsilon\}$,
$Y_7:\{\delta,\varepsilon\}$,
$F_1:\{\alpha,\varepsilon\}$,
$F_2:\{\alpha,\beta,\varepsilon\}$,
$F_3:\{\beta,\varepsilon\}$,
$F_4:\{\beta,\gamma,\varepsilon\}$,
$F_5:\{\gamma,\varepsilon\}$,
$F_6:\{\delta,\varepsilon\}$, and
$F_7:\{\delta,\varepsilon\}$.

Let $Y'_6=\{y\in Y_6: y \text{ has a neighbor in }F_4\}$ and let $Y''_6=Y_6\setminus Y'_6$. Then $Y''_6$ is anticomplete to $F_4$. Similarly, let $Y'_1=\{y\in Y_1: y \text{ has a neighbor in }F_6\}$ and let $Y''_1=Y_1\setminus Y'_1$. Then $Y''_1$ is anticomplete to $F_6$, and by Lemma \ref{claim:thm3-star}, $Y'_1$ is anticomplete to $F_2$. Let $F'_4=\{f\in F_4: f \text{ has a neighbor in }Y'_6\}$ and let $F''_4=F_4\setminus F'_4$. Then $F''_4$ is anticomplete to $Y'_6$. Finally, let $F'_2=\{f\in F_2: f \text{ has a neighbor in }F'_4\}$ and let $F''_2=F_2\setminus F'_2$. Then $F''_2$ is anticomplete to $F'_4$. Also, by Lemma \ref{claim:thm3-2star}, $F'_2$ is anticomplete to $Y_1$.

The following table gives the recoloring sequence and verifies each step.
\[
\begin{array}{c|c|c}
\textup{step} & \textup{obstructions} & \textup{Justification: anticomplete by} \\
\hline
F_3\to\beta
& 7,F_2,F_4
& \textup{definition of }F_3\textup{, }(\ref{item:5.5})
\\
F_5\to\gamma
& 1,2,Y_6,F_4
& \textup{definition of }F_5\textup{, }(\ref{item:5.4}),(\ref{item:5.5})
\\
F_1\to\alpha
& 5,6,Y_1,F_2
& \textup{definition of }F_1\textup{, }(\ref{item:5.4}),(\ref{item:5.5})
\\
Y_7\to\delta
& 3,4,Y_1,Y_6,F_6,F_7
& \textup{definition of }Y_7\textup{, }(\ref{item:5.3}),(\ref{item:5.4})
\\
Y'_1\overset{*}{\to}\alpha
& 5,6,Y''_1,F_1,F_2
& \textup{definition of }Y_1\textup{, }(\ref{item:5.1}),(\ref{item:5.4}),\textup{ Lemma \ref{claim:thm3-star}}
\\
Y''_6\to\gamma
& 1,2,Y'_6,F_5,F_4
& \textup{definition of }Y_6\textup{, }(\ref{item:5.1}),(\ref{item:5.4}),\textup{ definition of }Y''_6
\\
F_6\to\delta
& 3,4,Y'_6,Y_7,F_7,Y''_1
& \textup{definition of }F_6\textup{, }(\ref{item:5.4}),(\ref{item:5.5}),\textup{ definition of }Y''_1
\\
F'_2\overset{**}{\to}\alpha
& 5,6,Y_1,F_1,F''_2
& \textup{definition of }F_2\textup{, }(\ref{item:5.1}),(\ref{item:5.5}),\textup{ Lemma \ref{claim:thm3-2star}}
\\
F'_4\to\beta
& 7,F_3,F''_4,F''_2
& \textup{definition of }F_4\textup{, }(\ref{item:5.1}),(\ref{item:5.5}),\textup{ definition of }F''_2
\\
Y'_6\to\gamma
& 1,2,Y''_6,F_5,F''_4
& \textup{definition of }Y_6\textup{, }(\ref{item:5.1}),(\ref{item:5.4}),\textup{ definition of }F''_4
\\
Y''_1\to\delta
& 3,4,Y_7,F_6,F_7
& \textup{definition of }Y_1\textup{, }(\ref{item:5.3}),(\ref{item:5.4}),\textup{ definition of }Y''_1
\\
F''_2 \overset{*}{\to} \alpha
& 5,6,Y'_1,F_1,F'_2
& \textup{definition of }F_2\textup{, }(\ref{item:5.1}),(\ref{item:5.5}),\textup{ Lemma \ref{claim:thm3-star}}
\\
F''_4\to\gamma
& 1,2,Y_6,F_5
& \textup{definition of }F_4\textup{, }(\ref{item:5.5}),\textup{ definitions of }Y''_6\textup{ and }F''_4
\\
Y_1\to\varepsilon
& F_7
& \textup{(\ref{item:5.4})}
\\
F_1\to\varepsilon
& Y_1,F_7
& \textup{(\ref{item:5.4}),(\ref{item:5.5})}
\\
F_7\to\varepsilon
& Y_1,F_1
& \textup{(\ref{item:5.4}),(\ref{item:5.5})}
\\
5\to\varepsilon
& Y_1,F_1,F_7
& \textup{definition of }Y_1\textup{, }F_1\textup{, and }F_7
\end{array}
\]

Thus every step is legal, and the sequence ends with $I_3=\varepsilon$. Hence Lemma \ref{claim:CoverI} applies. \subcaseend

\noindent\phantomsection\label{subcase:2.5}\textbf{Subcase 2.5.} Suppose $\varphi(1,3,5)=(\beta,\gamma,\varepsilon)$.

Since $\varphi(1,2,3,4,5,6,7)=(\beta,\gamma,\gamma,\delta,\varepsilon,\alpha,\beta)$, the colors allowed by the cycle are
$Y_1:\{\alpha,\delta,\varepsilon\}$,
$Y_6, F_6:\{\gamma,\delta\}$,
$Y_7:\{\gamma,\delta,\varepsilon\}$,
$F_1:\{\alpha,\delta,\varepsilon\}$,
$F_2:\{\alpha,\varepsilon\}$,
$F_3:\{\alpha,\beta\}$,
$F_4:\{\beta\}$,
$F_5:\{\gamma\}$,
and
$F_7:\{\delta,\varepsilon\}$.

The following table gives the recoloring sequence and verifies each step.
\[
\begin{array}{c|c|c}
\textup{step} & \textup{obstructions} & \textup{Justification: anticomplete by} \\
\hline
F_3\to\beta
& 1,7,F_4
& \textup{definition of }F_3\textup{, }(\ref{item:5.5})
\\
F_1\to\alpha
& 6,Y_1,F_2
& \textup{definition of }F_1\textup{, }(\ref{item:5.4}),(\ref{item:5.5})
\\
F_6\to\gamma
& 2,3,Y_6,Y_7,F_5
& \textup{definition of }F_6\textup{, }(\ref{item:5.4}),(\ref{item:5.5})
\\
Y_6\to\gamma
& 2,3,Y_7,F_5,F_6
& \textup{definition of }Y_6\textup{, }(\ref{item:5.3}),(\ref{item:5.4})
\\
Y_1\to\delta
& 4,Y_7,F_7
& \textup{definition of }Y_1\textup{, }(\ref{item:5.3}),(\ref{item:5.4})
\\
Y_7\to\delta
& 4,Y_1,F_7
& \textup{definition of }Y_7\textup{, }(\ref{item:5.3}),(\ref{item:5.4})
\\
F_7\to\delta
& 4,Y_1,Y_7
& \textup{definition of }F_7\textup{, }(\ref{item:5.4})
\\
F_2\to\alpha
& 6,F_1
& \textup{definition of }F_2\textup{, }(\ref{item:5.5})
\\
Y_1\to\varepsilon
& 5
& \textup{definition of }Y_1
\\
F_1\to\varepsilon
& 5,Y_1
& \textup{definition of }F_1\textup{, }(\ref{item:5.4})
\\
F_7\to\varepsilon
& 5,Y_1,F_1
& \textup{definition of }F_7\textup{, }(\ref{item:5.4}),(\ref{item:5.5})
\end{array}
\]

Thus every step is legal, and the sequence ends with $I_3=\varepsilon$. Hence Lemma \ref{claim:CoverI} applies. \subcaseend

\noindent\phantomsection\label{subcase:2.6}\textbf{Subcase 2.6.} Suppose $\varphi(1,3,5)=(\beta,\varepsilon,\alpha)$.

Since $\varphi(1,2,3,4,5,6,7)=(\beta,\gamma,\varepsilon,\delta,\alpha,\alpha,\beta)$, the colors allowed by the cycle are
$Y_1:\{\alpha,\delta,\varepsilon\}$,
$Y_6,Y_7, F_6:\{\gamma,\delta,\varepsilon\}$,
$F_1:\{\alpha,\delta\}$,
$F_2:\{\alpha\}$,
$F_3:\{\beta\}$,
$F_4:\{\beta,\gamma\}$,
$F_5:\{\gamma,\varepsilon\}$,
and
$F_7:\{\delta,\varepsilon\}$.

The following table gives the recoloring sequence and verifies each step.
\[
\begin{array}{c|c|c}
\textup{step} & \textup{obstructions} & \textup{Justification: anticomplete by} \\
\hline
F_1\to\alpha
& 5,6,Y_1,F_2
& \textup{definition of }F_1\textup{, }(\ref{item:5.4}),(\ref{item:5.5})
\\
F_4\to\beta
& 1,7,F_3
& \textup{definition of }F_4\textup{, }(\ref{item:5.5})
\\
Y_7\to\delta
& 4,Y_1,Y_6,F_6,F_7
& \textup{definition of }Y_7\textup{, }(\ref{item:5.3}),(\ref{item:5.4})
\\
F_5\to\gamma
& 2,Y_6,F_6
& \textup{definition of }F_5\textup{, }(\ref{item:5.4}),(\ref{item:5.5})
\\
Y_7\to\varepsilon
& 3,Y_1,Y_6,F_6,F_7
& \textup{definition of }Y_7\textup{, }(\ref{item:5.3}),(\ref{item:5.4})
\\
Y_6\to\gamma
& 2,F_5,F_6
& \textup{definition of }Y_6\textup{, }(\ref{item:5.4})
\\
F_6\to\gamma
& 2,Y_6,F_5
& \textup{definition of }F_6\textup{, }(\ref{item:5.4}),(\ref{item:5.5})
\\
F_1\to\delta
& 4,Y_1,F_7
& \textup{definition of }F_1\textup{, }(\ref{item:5.4}),(\ref{item:5.5})
\\
Y_1\to\delta
& 4,F_1,F_7
& \textup{definition of }Y_1\textup{, }(\ref{item:5.4})
\\
F_7\to\delta
& 4,Y_1,F_1
& \textup{definition of }F_7\textup{, }(\ref{item:5.4}),(\ref{item:5.5})
\\
5\to\delta
& 4,Y_1,F_1,F_7
& \textup{definition of }Y_1\textup{, }F_1\textup{, and }F_7\textup{, and }45\notin E(G)
\end{array}
\]

Thus every step is legal, and the sequence ends with $I_3=\delta$. Hence Lemma \ref{claim:CoverI} applies. \subcaseend

\noindent\phantomsection\label{subcase:2.7}\textbf{Subcase 2.7.} Suppose $\varphi(1,3,5)=(\beta,\varepsilon,\delta)$.

Since $\varphi(1,2,3,4,5,6,7)=(\beta,\gamma,\varepsilon,\delta,\delta,\alpha,\beta)$, the colors allowed by the cycle are
$Y_1:\{\alpha,\delta,\varepsilon\}$,
$Y_6:\{\gamma,\varepsilon\}$,
$Y_7:\{\gamma,\delta,\varepsilon\}$,
$F_1:\{\alpha,\delta\}$,
$F_2:\{\alpha\}$,
$F_3:\{\alpha,\beta\}$,
$F_4:\{\beta,\gamma\}$,
$F_5, F_6:\{\gamma,\varepsilon\}$,
and
$F_7:\{\delta,\varepsilon\}$.

The following table gives the recoloring sequence and verifies each step.
\[
\begin{array}{c|c|c}
\textup{step} & \textup{obstructions} & \textup{Justification: anticomplete by} \\
\hline
Y_1\to\delta
& 4,5,Y_7,F_1,F_7
& \textup{definition of }Y_1\textup{, }(\ref{item:5.3}),(\ref{item:5.4})
\\
F_3\to\beta
& 1,7,F_4
& \textup{definition of }F_3\textup{, }(\ref{item:5.5})
\\
F_1\to\alpha
& 6,F_2
& \textup{definition of }F_1\textup{, }(\ref{item:5.5})
\\
Y_7\to\delta
& 4,5,Y_1,F_7
& \textup{definition of }Y_7\textup{, }(\ref{item:5.3}),(\ref{item:5.4})
\\
F_5\to\gamma
& 2,Y_6,F_4,F_6
& \textup{definition of }F_5\textup{, }(\ref{item:5.4}),(\ref{item:5.5})
\\
Y_7\to\varepsilon
& 3,Y_6,F_6,F_7
& \textup{definition of }Y_7\textup{, }(\ref{item:5.3}),(\ref{item:5.4})
\\
F_1\to\delta
& 4,5,Y_1,F_7
& \textup{definition of }F_1\textup{, }(\ref{item:5.4}),(\ref{item:5.5})
\\
F_7\to\delta
& 4,5,Y_1,F_1
& \textup{definition of }F_7\textup{, }(\ref{item:5.4}),(\ref{item:5.5})
\end{array}
\]

Thus every step is legal, and the sequence ends with $I_3=\delta$. Hence Lemma \ref{claim:CoverI} applies. \subcaseend

\noindent\phantomsection\label{subcase:2.8}\textbf{Subcase 2.8.} Suppose $\varphi(1,3,5)=(\beta,\delta,\varepsilon)$.

Since $\varphi(1,2,3,4,5,6,7)=(\beta,\gamma,\delta,\delta,\varepsilon,\alpha,\beta)$, the colors allowed by the cycle are
$Y_1:\{\alpha,\delta,\varepsilon\}$,
$Y_6, F_6:\{\gamma,\delta\}$,
$Y_7:\{\gamma,\delta,\varepsilon\}$,
$F_1, F_2:\{\alpha,\varepsilon\}$,
$F_3:\{\alpha,\beta\}$,
$F_4:\{\beta,\gamma\}$,
$F_5:\{\gamma\}$,
and
$F_7:\{\delta,\varepsilon\}$.

The following table gives the recoloring sequence and verifies each step.
\[
\begin{array}{c|c|c}
\textup{step} & \textup{obstructions} & \textup{Justification: anticomplete by} \\
\hline
F_4\to\beta
& 1,7,F_3
& \textup{definition of }F_4\textup{, }(\ref{item:5.5})
\\
F_6\to\gamma
& 2,Y_6,Y_7,F_5
& \textup{definition of }F_6\textup{, }(\ref{item:5.4}),(\ref{item:5.5})
\\
Y_6\to\gamma
& 2,Y_7,F_5,F_6
& \textup{definition of }Y_6\textup{, }(\ref{item:5.3}),(\ref{item:5.4})
\\
Y_7\to\delta
& 3,4,Y_1,F_7
& \textup{definition of }Y_7\textup{, }(\ref{item:5.3}),(\ref{item:5.4})
\\
Y_1\to\delta
& 3,4,Y_7,F_7
& \textup{definition of }Y_1\textup{, }(\ref{item:5.3}),(\ref{item:5.4})
\\
F_2\to\alpha
& 6,F_1,F_3
& \textup{definition of }F_2\textup{, }(\ref{item:5.5})
\\
F_1\to\varepsilon
& 5,F_7
& \textup{definition of }F_1\textup{, }(\ref{item:5.5})
\\
F_7\to\varepsilon
& 5,F_1
& \textup{definition of }F_7\textup{, }(\ref{item:5.5})
\\
Y_1\to\varepsilon
& 5,F_1,F_7
& \textup{definition of }Y_1\textup{, }(\ref{item:5.4})
\end{array}
\]

Thus every step is legal, and the sequence ends with $I_3=\varepsilon$. Hence Lemma \ref{claim:CoverI} applies. \subcaseend

\noindent\phantomsection\label{subcase:2.9}\textbf{Subcase 2.9.} Suppose $\varphi(1,3,5)=(\gamma,\delta,\varepsilon)$.
Since $\varphi(1,2,3,4,5,6,7)=(\gamma,\gamma,\delta,\delta,\varepsilon,\alpha,\beta)$, the colors allowed by the cycle are
$Y_1:\{\alpha,\delta,\varepsilon\}$,
$Y_6:\{\gamma,\delta\}$,
$Y_7, F_7:\{\delta,\varepsilon\}$,
$F_1:\{\alpha,\varepsilon\}$,
$F_2:\{\alpha,\beta,\varepsilon\}$,
$F_3:\{\alpha,\beta\}$,
$F_4:\{\beta,\gamma\}$,
$F_5:\{\gamma\}$, and
$F_6:\{\delta\}$.

Let $Y'_6=\{y\in Y_6: y \text{ has a neighbor in }F_4\}$ and let $Y''_6=Y_6\setminus Y'_6$. Then $Y''_6$ is anticomplete to $F_4$, and by Lemma \ref{claim:thm3-star}, $Y'_6$ is anticomplete to $F_7$. Let $F'_4=\{f\in F_4: f \text{ has a neighbor in }Y'_6\}$ and let $F''_4=F_4\setminus F'_4$. Then $F''_4$ is anticomplete to $Y'_6$. Finally, let $F'_2=\{f\in F_2: f \text{ has a neighbor in }F'_4\}$ and let $F''_2=F_2\setminus F'_2$. Then $F''_2$ is anticomplete to $F'_4$. Also, by Lemma \ref{claim:thm3-2star}, $F'_2$ is anticomplete to $Y_1$.

The following table gives the recoloring sequence and verifies each step.
\[
\begin{array}{c|c|c}
\textup{step} & \textup{obstructions} & \textup{Justification: anticomplete by} \\
\hline
F_3\to\beta
& 7,F_2,F_4
& \textup{definition of }F_3\textup{, }(\ref{item:5.5})
\\
Y_7\to\delta
& 3,4,Y_1,Y_6,F_6,F_7
& \textup{definition of }Y_7\textup{, }(\ref{item:5.3}),(\ref{item:5.4})
\\
F_1\to\alpha
& 6,Y_1,F_2
& \textup{definition of }F_1\textup{, }(\ref{item:5.4}),(\ref{item:5.5})
\\
F''_4\to\gamma
& 1,2,Y'_6,Y''_6,F_5,F'_4
& \textup{definition of }F_4\textup{, }(\ref{item:5.1}),(\ref{item:5.5}),\textup{ definitions of }Y''_6\textup{ and }F''_4
\\
F'_2\overset{**}{\to}\alpha
& 6,Y_1,F_1,F''_2
& \textup{definition of }F_2\textup{, }(\ref{item:5.1}),(\ref{item:5.5}),\textup{ Lemma \ref{claim:thm3-2star}}
\\
F''_2\to\beta
& 7,F_3,F'_4
& \textup{definition of }F_2\textup{, }(\ref{item:5.5}),\textup{ definition of }F''_2
\\
Y_1\to\varepsilon
& 5,Y_7,F_7
& \textup{definition of }Y_1\textup{, }(\ref{item:5.3}),(\ref{item:5.4})
\\
Y'_6\overset{*}{\to}\delta
& 3,4,Y''_6,Y_7,F_6,F_7
& \textup{definition of }Y_6\textup{, }(\ref{item:5.1}),(\ref{item:5.3}),(\ref{item:5.4}),\textup{ Lemma \ref{claim:thm3-star}}
\\
F'_4\to\gamma
& 1,2,Y''_6,F_5,F''_4
& \textup{definition of }F_4\textup{, }(\ref{item:5.1}),(\ref{item:5.5}),\textup{ definition of }Y''_6
\end{array}
\]

Thus every step is legal, and the sequence ends with $I_1=\gamma$. Hence Lemma \ref{claim:CoverI} applies. \subcaseend

\noindent\phantomsection\label{subcase:2.10}\textbf{Subcase 2.10.} Suppose $\varphi(1,3,5)=(\gamma,\varepsilon,\alpha)$.

Since $\varphi(1,2,3,4,5,6,7)=(\gamma,\gamma,\varepsilon,\delta,\alpha,\alpha,\beta)$, the colors allowed by the cycle are
$Y_1:\{\alpha,\delta,\varepsilon\}$,
$Y_6:\{\gamma,\delta,\varepsilon\}$,
$Y_7, F_6, F_7:\{\delta,\varepsilon\}$,
$F_1:\{\alpha,\delta\}$,
$F_2:\{\alpha,\beta\}$,
$F_3:\{\beta\}$,
$F_4:\{\beta,\gamma\}$, and
$F_5:\{\gamma,\varepsilon\}$.

Let $Y'_6=\{y\in Y_6: y \text{ has a neighbor in }F_4\}$ and let $Y''_6=Y_6\setminus Y'_6$. Then $Y''_6$ is anticomplete to $F_4$, and by Lemma \ref{claim:thm3-star}, $Y'_6$ is anticomplete to $F_7$. Let $F'_4=\{f\in F_4: f \text{ has a neighbor in }Y'_6\}$ and let $F''_4=F_4\setminus F'_4$. Then $F''_4$ is anticomplete to $Y'_6$. Finally, let $F'_2=\{f\in F_2: f \text{ has a neighbor in }F'_4\}$ and let $F''_2=F_2\setminus F'_2$. Then $F''_2$ is anticomplete to $F'_4$. Also, by Lemma \ref{claim:thm3-2star}, $F'_2$ is anticomplete to $Y_1$.

The following table gives the recoloring sequence and verifies each step.
\[
\begin{array}{c|c|c}
\textup{step} & \textup{obstructions} & \textup{Justification: anticomplete by} \\
\hline
F_3\to\beta
& 7,F_2,F_4
& \textup{definition of }F_3\textup{, }(\ref{item:5.5})
\\
F_5\to\gamma
& 1,2,Y_6,F_4
& \textup{definition of }F_5\textup{, }(\ref{item:5.4}),(\ref{item:5.5})
\\
F''_4\to\gamma
& 1,2,Y'_6,Y''_6,F_5,F'_4
& \textup{definition of }F_4\textup{, }(\ref{item:5.1}),(\ref{item:5.5}),\textup{ definitions of }Y''_6\textup{ and }F''_4
\\
F''_2\to\beta
& 7,F_3,F'_2,F'_4
& \textup{definition of }F_2\textup{, }(\ref{item:5.1}),(\ref{item:5.5}),\textup{ definition of }F''_2
\\
F_1\to\alpha
& 5,6,Y_1,F'_2
& \textup{definition of }F_1\textup{, }(\ref{item:5.4}),(\ref{item:5.5})
\\
Y_1\overset{**}{\to}\alpha
& 5,6,F_1,F'_2
& \textup{definition of }Y_1\textup{, }(\ref{item:5.4}),\textup{ Lemma \ref{claim:thm3-2star}}
\\
Y'_6\overset{*}{\to}\delta
& 4,Y''_6,Y_7,F_6,F_7
& \textup{definition of }Y_6\textup{, }(\ref{item:5.1}),(\ref{item:5.3}),(\ref{item:5.4}),\textup{ Lemma \ref{claim:thm3-star}}
\\
F'_4\to\gamma
& 1,2,Y''_6,F_5,F''_4
& \textup{definition of }F_4\textup{, }(\ref{item:5.1}),(\ref{item:5.5}),\textup{ definition of }Y''_6
\end{array}
\]

Thus every step is legal, and the sequence ends with $I_1=\gamma$. Hence Lemma \ref{claim:CoverI} applies. \subcaseend

\noindent\phantomsection\label{subcase:2.11}\textbf{Subcase 2.11.} Suppose $\varphi(1,3,5)=(\gamma,\varepsilon,\delta)$. 

Since $\varphi(1,2,3,4,5,6,7)=(\gamma,\gamma,\varepsilon,\delta,\delta,\alpha,\beta)$, the colors allowed by the cycle are
$Y_1:\{\alpha,\delta,\varepsilon\}$,
$Y_6, F_5:\{\gamma,\varepsilon\}$,
$Y_7, F_7:\{\delta,\varepsilon\}$,
$F_1:\{\alpha,\delta\}$,
$F_2,F_3:\{\alpha,\beta\}$,
$F_4:\{\beta,\gamma\}$, and
$F_6:\{\varepsilon\}$.

The following table gives the recoloring sequence and verifies each step.
\[
\begin{array}{c|c|c}
\textup{step} & \textup{obstructions} & \textup{Justification: anticomplete by} \\
\hline
F_3\to\beta
& 7,F_2,F_4
& \textup{definition of }F_3\textup{, }(\ref{item:5.5})
\\
F_1\to\alpha
& 6,Y_1,F_2
& \textup{definition of }F_1\textup{, }(\ref{item:5.4}),(\ref{item:5.5})
\\
Y_1\to\delta
& 4,5,Y_7,F_7
& \textup{definition of }Y_1\textup{, }(\ref{item:5.3}),(\ref{item:5.4})
\\
F_7\to\delta
& 4,5,Y_1,Y_7
& \textup{definition of }F_7\textup{, }(\ref{item:5.4})
\\
Y_7\to\delta
& 4,5,Y_1,F_7
& \textup{definition of }Y_7\textup{, }(\ref{item:5.3}),(\ref{item:5.4})
\\
Y_6\to\varepsilon
& 3,F_5,F_6
& \textup{definition of }Y_6\textup{, }(\ref{item:5.4})
\\
F_4\to\gamma
& 1,2,F_5
& \textup{definition of }F_4\textup{, }(\ref{item:5.5})
\\
F_5\to\gamma
& 1,2,F_4
& \textup{definition of }F_5\textup{, }(\ref{item:5.5})
\end{array}
\]

Thus every step is legal, and the sequence ends with $I_1=\gamma$. Hence Lemma \ref{claim:CoverI} applies. \subcaseend

\noindent\phantomsection\label{subcase:2.12}\textbf{Subcase 2.12.} Suppose $\varphi(1,3,5)=(\varepsilon,\gamma,\alpha)$.

Since $\varphi(1,2,3,4,5,6,7)=(\varepsilon,\gamma,\gamma,\delta,\alpha,\alpha,\beta)$, the colors allowed by the cycle are
$Y_1, F_1:\{\alpha,\delta\}$,
$Y_6:\{\gamma,\delta,\varepsilon\}$,
$Y_7,F_6:\{\gamma,\delta\}$,
$F_2:\{\alpha,\beta\}$,
$F_3, F_4:\{\beta,\varepsilon\}$,
$F_5:\{\gamma,\varepsilon\}$,
and
$F_7:\{\delta\}$.

The following table gives the recoloring sequence and verifies each step.
\[
\begin{array}{c|c|c}
\textup{step} & \textup{obstructions} & \textup{Justification: anticomplete by} \\
\hline
Y_6\to\gamma
& 2,3,Y_7,F_5,F_6
& \textup{definition of }Y_6\textup{, }(\ref{item:5.3}),(\ref{item:5.4})
\\
F_3\to\beta
& 7,F_2,F_4
& \textup{definition of }F_3\textup{, }(\ref{item:5.5})
\\
F_4\to\varepsilon
& 1,F_5
& \textup{definition of }F_4\textup{, }(\ref{item:5.5})
\\
F_5\to\varepsilon
& 1,F_4
& \textup{definition of }F_5\textup{, }(\ref{item:5.5})
\\
2\to\varepsilon
& 1,F_4,F_5
& \textup{definition of }F_4\textup{ and }F_5\textup{, and }12\notin E(G)
\end{array}
\]

Thus every step is legal, and the sequence ends with $I_1=\varepsilon$. Hence Lemma \ref{claim:CoverI} applies. \subcaseend

\noindent\phantomsection\label{subcase:2.13}\textbf{Subcase 2.13.} Suppose $\varphi(1,3,5)=(\varepsilon,\gamma,\delta)$.

Since $\varphi(1,2,3,4,5,6,7)=(\varepsilon,\gamma,\gamma,\delta,\delta,\alpha,\beta)$, the colors allowed by the cycle are
$Y_1, F_1:\{\alpha,\delta\}$,
$Y_6, F_5:\{\gamma,\varepsilon\}$,
$Y_7:\{\gamma,\delta\}$,
$F_2:\{\alpha,\beta\}$,
$F_3:\{\alpha,\beta,\varepsilon\}$,
$F_4:\{\beta,\varepsilon\}$,
$F_6:\{\gamma\}$, and
$F_7:\{\delta\}$.

Let $Y'_7=\{y\in Y_7: y \text{ has a neighbor in }F_5\}$ and let $Y''_7=Y_7\setminus Y'_7$. Then $Y''_7$ is anticomplete to $F_5$, and by Lemma \ref{claim:thm3-star}, $Y'_7$ is anticomplete to $F_1$.

The following table gives the recoloring sequence and verifies each step.
\[
\begin{array}{c|c|c}
\textup{step} & \textup{obstructions} & \textup{Justification: anticomplete by} \\
\hline
Y''_7\to\gamma
& 2,3,Y_6,Y'_7,F_5,F_6
& \textup{definition of }Y_7\textup{, }(\ref{item:5.1}),(\ref{item:5.3}),(\ref{item:5.4}),\textup{ definition of }Y''_7
\\
Y_1\to\delta
& 4,5,Y'_7,F_1,F_7
& \textup{definition of }Y_1\textup{, }(\ref{item:5.3}),(\ref{item:5.4})
\\
F_7\to\delta
& 4,5,Y_1,Y'_7,F_1
& \textup{definition of }F_7\textup{, }(\ref{item:5.4}),(\ref{item:5.5})
\\
F_1\overset{*}{\to}\delta
& 4,5,Y_1,Y'_7,F_7
& \textup{definition of }F_1\textup{, }(\ref{item:5.4}),(\ref{item:5.5}),\textup{ Lemma \ref{claim:thm3-star}}
\end{array}
\]

Thus every step is legal, and the sequence ends with $I_3=\delta$. Hence Lemma \ref{claim:CoverI} applies. \subcaseend

\noindent\phantomsection\label{subcase:2.14}\textbf{Subcase 2.14.} Suppose $\varphi(1,3,5)=(\varepsilon,\delta,\alpha)$.

Since $\varphi(1,2,3,4,5,6,7)=(\varepsilon,\gamma,\delta,\delta,\alpha,\alpha,\beta)$, the colors allowed by the cycle are
$Y_1:\{\alpha,\delta\}$,
$Y_6:\{\gamma,\delta,\varepsilon\}$,
$Y_7, F_6:\{\gamma,\delta\}$,
$F_1:\{\alpha\}$,
$F_2:\{\alpha,\beta\}$,
$F_3:\{\beta,\varepsilon\}$,
$F_4:\{\beta,\gamma,\varepsilon\}$,
$F_5:\{\gamma,\varepsilon\}$, and
$F_7:\{\delta\}$.

Let $Y'_6=\{y\in Y_6: y \text{ has a neighbor in }F_4\}$ and let $Y''_6=Y_6\setminus Y'_6$. Then $Y''_6$ is anticomplete to $F_4$, and by Lemma \ref{claim:thm3-star}, $Y'_6$ is anticomplete to $F_7$. Let $F'_4=\{f\in F_4: f \text{ has a neighbor in }Y'_6\}$ and let $F''_4=F_4\setminus F'_4$. Then $F''_4$ is anticomplete to $Y'_6$. Finally, let $F'_2=\{f\in F_2: f \text{ has a neighbor in }F'_4\}$ and let $F''_2=F_2\setminus F'_2$. Then $F''_2$ is anticomplete to $F'_4$. Also, by Lemma \ref{claim:thm3-2star}, $F'_2$ is anticomplete to $Y_1$.

The following table gives the recoloring sequence and verifies each step.
\[
\begin{array}{c|c|c}
\textup{step} & \textup{obstructions} & \textup{Justification: anticomplete by} \\
\hline
F_3\to\beta
& 7,F_2,F_4
& \textup{definition of }F_3\textup{, }(\ref{item:5.5})
\\
F''_4\to\varepsilon
& 1,Y'_6,Y''_6,F_5,F'_4
& \textup{definition of }F_4\textup{, }(\ref{item:5.1}),(\ref{item:5.5}),\textup{ definitions of }Y''_6\textup{ and }F''_4
\\
F''_2\to\beta
& 7,F_3,F'_2,F'_4
& \textup{definition of }F_2\textup{, }(\ref{item:5.1}),(\ref{item:5.5}),\textup{ definition of }F''_2
\\
Y_1\overset{**}{\to}\alpha
& 5,6,F_1,F'_2
& \textup{definition of }Y_1\textup{, }(\ref{item:5.4}),\textup{ Lemma \ref{claim:thm3-2star}}
\\
Y'_6\overset{*}{\to}\delta
& 3,4,Y''_6,Y_7,F_6,F_7
& \textup{definition of }Y_6\textup{, }(\ref{item:5.1}),(\ref{item:5.3}),(\ref{item:5.4}),\textup{ Lemma \ref{claim:thm3-star}}
\\
Y''_6\to\gamma
& 2,Y_7,F_5,F_6,F'_4
& \textup{definition of }Y_6\textup{, }(\ref{item:5.3}),(\ref{item:5.4}),\textup{ definition of }Y''_6
\\
F'_4\to\varepsilon
& 1,F''_4,F_5
& \textup{definition of }F_4\textup{, }(\ref{item:5.1}),(\ref{item:5.5})
\\
F_5\to\varepsilon
& 1,F_4
& \textup{definition of }F_5\textup{, }(\ref{item:5.5})
\\
2\to\varepsilon
& 1,F_4,F_5
& \textup{definition of }F_4\textup{ and }F_5\textup{, and }12\notin E(G)
\end{array}
\]

Thus every step is legal, and the sequence ends with $I_1=\varepsilon$. Hence Lemma \ref{claim:CoverI} applies. \subcaseend

\section*{Acknowledgments}
This work was conducted during the first author's research stay funded by ANID Complementary Benefits and directed by Dr. Jessica McDonald at Auburn University.

\section*{Declaration of use of artificial intelligence}
During the preparation of this manuscript the authors used OpenAI's ChatGPT 5.5 Plus and Anthropic's Claude to assist with editing, grammar, and formatting. The authors verified and reviewed all AI-assisted edits and take full responsibility for all content within this manuscript.


\end{document}